\documentclass[a4paper,reqno]{amsart}

\usepackage[utf8]{inputenc}
\usepackage{tikz}
\usetikzlibrary{fadings,fpu,arrows,decorations.pathmorphing,backgrounds,positioning,fit,petri,patterns,decorations.markings}
\usepackage{pgfplots}
\usepackage{graphicx}
\usepackage{amsfonts}
\usepackage{todonotes}
\usepackage{mathtools}
\usepackage{amssymb}
\usepackage{enumitem}
\usepackage[colorlinks=true,linkcolor=blue,citecolor=blue]{hyperref}
\usepackage{units}

\renewcommand{\d}{\text{d}}
\newcommand{\bs}{\backslash}
\newcommand{\Mto}{\xrightarrow{\rm{M}}}
\newcommand{\wto}{\rightharpoonup}
\newcommand{\intt}{\text{\rm{int}}}

\newcommand{\dist}{\text{\rm dist}}
\newcommand{\vol}{\textrm{\rm vol}}
\newcommand{\SCI}{\operatorname{SCI}}
\newcommand{\diam}{\operatorname{diam}}
\newcommand{\diag}{\operatorname{diag}}

\newcommand{\cO}{\mathcal{O}}
\newcommand{\cU}{\mathcal{U}}
\newcommand{\dil}{\operatorname{\red{\mathfrak{N}}}}
\newcommand{\col}[1]{\partial^{#1}}
\newcommand{\out}{\partial^{\text{\rm{out}}}}
\newcommand{\R}{\mathbb R}
\newcommand{\C}{\mathbb C}
\newcommand{\N}{\mathbb N}
\newcommand{\Z}{\mathbb Z}
\newcommand{\dAW}[2]{\mathrm{d}_{\mathrm{AW}}\left(#1,#2\right)}
\newcommand{\dAWonly}{\mathrm{d}_{\mathrm{AW}}}

\newcommand{\Om}{\Omega}
\newcommand{\eps}{\epsilon}
\newcommand{\f}{\frac}
\newcommand{\del}{\partial}
\newcommand{\cl}{\mathrm{cl}}
\newcommand{\mat}{\mathrm{mat}}
\newcommand{\pix}{\mathrm{pix}}

\newcommand{\supp}{\operatorname{supp}}

\newcommand{\clg}{\set{c_{lg}^{(1)},c_{lg}^{(2)}}}

\renewcommand{\subseteq}{\subset}
\renewcommand{\vol}{\mu_{\text{\rm leb}}}
\renewcommand{\Omega}{\mathcal{S}}
\renewcommand{\Om}{\mathcal{S}}

\DeclarePairedDelimiter{\abs}{\lvert}{\rvert}
\DeclarePairedDelimiter{\norm}{\lVert}{\rVert}
\DeclarePairedDelimiter{\floor}{\lfloor}{\rfloor}
\DeclarePairedDelimiter{\ceil}{\lceil}{\rceil}
\DeclarePairedDelimiter{\inner}{\langle}{\rangle}
\DeclarePairedDelimiter{\set}{\lbrace}{\rbrace}
\DeclarePairedDelimiter{\br}{(}{)}
\DeclarePairedDelimiter{\sbr}{[}{]}

\theoremstyle{definition}
\newtheorem{de}{Definition}[section]

\theoremstyle{plain}
\newtheorem{prop}[de]{Proposition}
\newtheorem{lemma}[de]{Lemma}
\newtheorem{theorem}[de]{Theorem}

\theoremstyle{remark}
\newtheorem{remark}[de]{Remark}
\newtheorem*{remark*}{Remark}
\newtheorem{example}{Example}

\numberwithin{equation}{section}

\newlength{\figurewidth}
\newlength{\figureheight}

\newenvironment{Red}{\par\color{black}}{\par}

\newcommand{\red}[1]{\textcolor{black}{#1}}

\newcommand{\rednew}[1]{\textcolor{black}{#1}}

\title{Computing eigenvalues of the Laplacian on rough domains}
\author{Frank R\"osler}
\email{frank.roesler@unibe.ch}
\address{Department of Mathematics, University of Bern, Alpeneggstrasse 22, 3012 Bern, Switzerland}
\author{Alexei Stepanenko}
\email{StepanenkoA@cardiff.ac.uk}
\address{School of Mathematics, Cardiff University, Senghennydd Road, Cardiff CF24 4AG, Wales, UK}

\date{\today}
\thanks{The authors would like to extend their thanks to Jonathan Ben-Artzi and Marco Marletta for many enlightening discussions
  as well as to Victor Burenkov and Simon Chandler-Wilde for their helpful comments.
  \red{The authors also thank the anonymous referees whose helpful comments and suggestions motivated us to improve the presentation of the article
  as well as to prove the sharpness result Theorem \ref{thm:spec-exist} (ii). }
FR acknowledges support from the European Union's Horizon 2020 Research and Innovation Programme under the Marie Sklodowska-Curie grant agreement No 885904.
The research of AS is supported by the United Kingdom Engineering and Physical
Sciences Research Council, through its Doctoral Training Partnership with Cardiff University.
}


\begin{document}

\begin{abstract}
  We prove a general Mosco convergence theorem for bounded Euclidean domains satisfying a set of mild geometric hypotheses.
  For bounded domains, this notion implies norm-resolvent convergence for the Dirichlet Laplacian which in turn ensures spectral convergence.
  A key element of the proof is the development of a novel, explicit Poincar\'e-type inequality.
  These results allow us to construct a universal algorithm capable of computing the eigenvalues of the Dirichlet Laplacian
  on a wide class of rough domains. Many domains with fractal boundaries, such as the Koch snowflake and certain filled Julia sets, are included among this class.
  Conversely, we construct a counter example showing that there does not exist a universal algorithm of the same type capable
  of computing the eigenvalues of the Dirichlet Laplacian on an arbitrary bounded domain.
\end{abstract}
\maketitle

\section{Introduction}

The purpose of this paper is to investigate numerical methods for computing Dirichlet eigenvalues of bounded domains
with extremely rough, possibly fractal, boundaries and to develop analytical tools for dealing with such problems. 
Following \cite{seashell}, we utilise the framework of \textit{Solvability Complexity Indices (SCI)} \cite{Hansen2011} and consider sequences
of \textit{arithmetic algorithms} $(\Gamma_n)_{n \in \N}$ intended to approximate the spectrum $\sigma(-\Delta_\cO)$ of the Dirichlet Laplacian as $n \to \infty$
on any domain \red{(i.e. a non-empty, open, connected set)} $\cO\subset \R^2$ in a specified \textit{primary set} $\Omega \subset 2^{\R^2}$ (recall $2^{A}$ = power set of set $A$).
The sole input to each arithmetic algorithm $\Gamma_n$ is the information of whether or not a chosen finite number of points lie in the domain $\cO$ 
and the output is a closed subset of $\C$ which should approximate $\sigma(-\Delta_\cO)$ in an appropriate metric.
\red{Each $\Gamma_n$ obtains its output from the input via a finite number of arithmetic operations. The rigorous formulation  will be given in Sections \ref{sec:comp-probl-arithm} and \ref{subsec:comp_spec_prob}.}

The question we ask is: what is the ``largest'' primary set $\Omega$ of bounded domains we can identify such that there exists a single sequence of arithmetic
algorithms computing the Dirichlet eigenvalues of any domain in $\Omega$?
Note that boundary regularity is known to have a very real and highly non-
trivial impact on the spectral properties of the Dirichlet Laplacian.
For instance, in a neighbourhood
of a reentrant corner or a cusp, the eigenfunctions do not necessarily lie in the $H^2$ Sobolev space \cite{grisvardEllipticProblemsNonsmooth2011} \red{and the eigenvalues asymptotics must be modified to account for fractal boundaries (cf. \cite{lapidusFractalDrumInverse1991,fleckinger-pelleExampleTwoTermAsymptotics1993,huaFractalDrumsThendimensional1995,levitinSpectralAsymptoticsRenewal1996,LapidusCounter} and references therein).}

Our first finding shows that there is no hope of constructing a single sequence of arithmetic algorithms capable of computing the Dirichlet eigenvalues
of every bounded domain (cf. Proposition \ref{prop:spec-not-exist}).
The problem of proving the existence of sequences of arithmetic algorithms that do converge is approached via explicit construction.
We shall introduce an approximation $\cO_n$ for a domain $\cO$ which we refer to as a \textit{pixelated domain} for $\cO$ (cf. Definition \ref{defn:pixel}).
Each pixelated domain $\cO_n$ is constructed solely from the information of which points in the grid $(\tfrac{1}{n} \Z)^2$ lie in $\cO$.
Utilising computable error bounds for the finite element method, in particular the results of Liu and Oishi
\cite{LiuOishi2013}, we construct a sequence of arithmetic algorithms which compute the spectrum of the Dirichlet Laplacian on any domain
for which the corresponding pixelated domains converge in the \textit{Mosco sense} (cf. Proposition \ref{prop:alg-mosco}).

\red{
In this way, the problem reduces to the study of Mosco convergence, which occupies much of the paper.
This notion ensures convergence of Dirichlet eigenvalues and of solutions to the Poisson equation
hence has various other applications as well as to the above computational question. }
In Section \ref{sec:mosco-proof}, we prove that if the Hausdorff convergence condition
\begin{equation}
  \label{eq:hauss-cond-intro}
  \d_H(\cO,\cO_n) + \d_H(\partial \cO, \partial \cO_n) \to 0 \quad \text{as} \quad n \to \infty,
\end{equation}
holds and a collection of mild geometric conditions (such as topological regularity of $\cO$) are satisfied, then
$\cO_n$ converges to $\cO$ in the Mosco sense (cf. Theorem \ref{th:mosco}).
This result, which is valid for arbitrary sequences of domains, is applied to pixelated domains thus concluding the identification of a large primary set $\Omega_1$ of
bounded domains for which there exists a corresponding sequence of arithmetic algorithms (cf. Theorem \ref{thm:spec-exist}).
These arithmetic algorithms describe a simple numerical method that is guaranteed to converge on a very wide class of rough domains. 

An intermediate step in the proof of Theorem \ref{th:mosco} is the reduction of Mosco convergence $\cO_n \Mto \cO$
to the establishment of uniform Poincar\'e-type inequalities \red{of the form}
\begin{align}
  \forall u \in C^\infty_c(\cO):  \quad \norm{u}_{L^2(\col{r}\cO)} & \leq C r \norm{\nabla u}_{L^2(\col{\alpha r}\cO)} \label{eq:poin-O-intro} \\
  \forall u \in C^\infty_c(\cO_n): \quad \norm{u}_{L^2(\col{r}\cO_n)} & \leq C r \norm{\nabla u}_{L^2(\col{\alpha r}\cO_n)} \label{eq:poin-On-intro}
\end{align}
for all small enough $r > 0$, where $C, \alpha  > 0$ are numerical constants and $\col{r} \cO := \set{x \in \cO : \dist(x, \partial \cO) < r}$
(cf. Proposition \ref{prop:mosco}).
A Poincar\'e-type inequality of the form (\ref{eq:poin-O-intro}), for a single domain $\cO$, is proved in Section \ref{sec:poin} via
a geometric method involving the construction of a bundle of paths from every point in $\col{r}\cO$ to $\partial \cO$ (cf. Theorem \ref{thm:poin}).
The uniform Poincar\'e-type inequality (\ref{eq:poin-On-intro}), for a sequence of domains $\cO_n$, is established by combining
Theorem \ref{thm:poin} with a characterisation of the geometry of $\partial \cO_n$ for large $n$ (cf. Proposition \ref{prop:large+lonely}).

\subsection*{Organisation of the paper}

In Section \ref{sec:prel-overv-results}, we state our main results and provide preliminaries.
In Section \ref{sec:poin}, we prove an explicit Poincar\'e-type inequality.
Section \ref{sec:mosco} is dedicated to proving Mosco convergence results.
In Section \ref{sec:spectral}, we apply our analytical results to the theory of Solvability Complexity Indices.
In Section \ref{sec:numerical-results}, we illustrate our results with a numerical investigation for the Dirichlet Laplacian
on a filled Julia set.

\subsection*{Notation and conventions}
We shall adopt the following notation, which is \textit{not} necessarily standard
and which will be used frequently throughout the paper. \red{Let $d \in \N$ throughout.}
\begin{itemize}
 \item
  For any $r > 0$ and any set $A \subset \R^d$, the \textit{$r$-collar neighbourhood} $\col{r} A$ is defined (as above) by 
  \begin{equation}
    \label{eq:col-defn}
 \col{r} A := \set{x \in A : \dist(x, \partial A) < r}.
\end{equation}
\item
  For any set $A \subset \R^d$, we let $\#_c(A) \in  \N_0 \cup \set{\infty}$ denote the number of connected components of $A$. 
\item
  For any $r > 0$ and any set $A \subset \R^d$, we define the set $\dil_r(A)$ by 
  \begin{equation}
    \label{eq:dil-defn}
  \dil_r(A) := \set{ x \in \R^d : \dist(x, A) < r}.
\end{equation}
\item  $\sigma(\cO)$ shall denote the spectrum of the Dirichlet Laplacian $-\Delta_\cO$ on $L^2(\cO)$. 
\end{itemize}
 We shall also use the following notation. 
 \begin{itemize}
 \item For every $A \subset \R^d$, $\vol(A)$ denotes the $d$-dimensional Lebesgue outer measure. 
   
\item
  For any open set $U \subseteq \R^d$, 
  \begin{equation*}
    H^1(U) := \set{ u \in L^2(U) : \norm{\nabla u}_{L^2(U)} < \infty },
  \end{equation*}
  \begin{equation}
    \label{eq:H1norm-defn}
    \norm{\cdot}_{H^1(U)} := \br{\norm{\cdot}^2_{L^2(U)} + \norm{\nabla \cdot}^2_{L^2(U)}}^{1/2}
  \end{equation}
  and  $H^1_0(U)$ is defined as the closure of $C^\infty_c(U)$ in $H^1(U)$.
\item 
  For any non-empty, bounded sets $A, B \subset \R^d$, the Hausdorff distance between $A$ and $B$ is defined by
  \begin{equation}
    \label{eq:dH-defn}
  \d_H(A,B) : = \max \set*{ \sup_{x \in A}\dist(x, B) , \sup_{x \in B}\dist(x,A) }. 
\end{equation}
  We define $\d_H(\emptyset,A) = \infty$ for any non-empty bounded open set $A \subset \R^d$ and $\d_H(\emptyset,\emptyset) = 0$.

\item
  \rednew{We let $B_r(x) \subset \R^d$ denote an open ball of radius $ r > 0$ about $x \in \R^d$.}

 \item
  The diameter of a set $A \subset \R^d$ is denoted by
  \begin{equation}
    \label{eq:diam-defn}
     \diam(A) := \sup_{x \in A} \sup_{y \in A} |x - y| \in [0,\infty) \cup \set{\infty}.
   \end{equation}  
    
\end{itemize}

\section{Preliminaries and overview of results}\label{sec:prel-overv-results}
This section is devoted to providing preliminaries, stating our main results and
reviewing some closely related literature. 
In Sections \ref{subsec:mosco-intro} and \ref{subsec:poin-intro} we present our
analytical results on Mosco convergence and Poincar\'e-type inequalities respectively.
An introduction to the theory of Solvability Complexity Indices is given in Section \ref{sec:comp-probl-arithm} and we state our results
on the computational complexity of the eigenvalue problem in Section \ref{subsec:comp_spec_prob}.

\subsection{Mosco convergence}\label{subsec:mosco-intro}

The question of whether a given approximation for a domain gives a reliable spectral approximation for the Dirichlet Laplacian
leads us to study \textit{Mosco convergence}.
We shall give the definition for $H^1_0$ Sobolev spaces on Euclidean domains (as in \cite[Defn. 1.1]{danersDirichletProblemsVarying2003})
but the notion can be more generally formulated for convex subsets of Banach spaces \cite{moscoConvergenceConvexSets1969}.

\begin{de}\label{defn:mosco-convergence}
  The sequence of open sets $\cO_n \subseteq \R^d,\,n \in \N$, converges to an open set $\cO \subseteq \R^d$ in the \textit{Mosco sense},
  denoted by $\cO_n \Mto \cO$ as $n \to \infty$, if:
    \begin{enumerate}
     \item Any weak limit point $u$ of a sequence $u_n \in H_0^1(\cO_n)$, $n \in \N$, satisfies $u \in H_0^1(\cO)$.
     \item  For every $u \in H^1_0(\cO)$ there exists $u_n \in H^1_0(\cO_n)$ such that $u_n \to u$  as $n \to \infty$ in $H^1(\R^d)$.
    \end{enumerate}
\end{de}

Note that a function $u \in H^1_0(\cO)$ may be realised as a function in $H^1(\R^d)$ via extension by zero.
For an arbitrary open set $\cO \subset \R^d$, one may realise the Dirichlet Laplacian $- \Delta_\cO$ as a positive, self-adjoint operator
on $L^2(\cO)$  \cite[Th. VI.1.4]{EE}. 
In the case that $\cO$ is bounded, the Dirichlet Laplacian has compact resolvent hence purely discrete spectrum.

Provided the open sets $\cO \subset \R^d$ and $\cO_n \subset \R^d$, $n \in \N$, are bounded, Mosco convergence $\cO_n \Mto \cO$ as $n \to \infty$
implies that $- \Delta_{\cO_n}$ converges to $- \Delta_{\cO}$ in the norm-resolvent sense as $n \to \infty$ \cite[Th. 3.3 and 3.5]{danersDirichletProblemsVarying2003}.
In turn, norm-resolvent convergence implies spectral convergence \red{in the sense of the following lemma} \cite[Th. VIII.23]{reedMethodsModernMathematical2012}.

\begin{lemma}\label{lem:mosco-implies-Hauss}
  If  $\cO \subset \R^d$ and $\cO_n \subset \R^d$, $n \in \N$, are open and bounded, and $\cO_n \Mto \cO$ as $n \to \infty$,
  then for every bounded, open $S \subset \R$,
  \begin{equation*}
    \d_H(\sigma(\cO_n)\cap S, \sigma(\cO)\cap S) \to 0 \quad \text{as} \quad n \to \infty.
  \end{equation*}
\end{lemma}
An open set $\cO \subset \R^d$ is said to be \textit{regular} if
\begin{equation}
  \label{eq:regular-defn}
  \cO  = \intt(\overline{\cO}).
\end{equation}
For a bounded open set $\cO \subset \R^d$, the quantity \red{$Q(\partial \cO) \geq 0$} is defined by
\begin{equation}
  \label{eq:min-comp-diam}
  Q(\partial \cO) := \inf \set*{\diam(\Gamma) : \Gamma \subseteq \partial \cO \text{ path-connected component of }\partial \cO}.
\end{equation}
Recall that $\#_c$ denotes the number of connected components.
Recall that a set $A \subset \R^d$ is \textit{locally connected} if for every $x \in A$, there exists an open neighbourhood $U \subset \R^d$ of $x$ such that
$U \cap A$ is connected.
In Section \ref{sec:mosco-proof}, we shall prove the following result, which provides geometric hypotheses ensuring Mosco convergence for domains in $\R^2$.

\begin{theorem}\label{th:mosco}
  Suppose that $\cO \subset \R^2$ is a bounded, connected, regular open set such that $\vol(\partial \cO) = 0$,
  $Q(\partial \cO) > 0$ and $\#_c \intt(\cO^c) = \#_c (\cO^c) < \infty$.
  Suppose that $\cO_n \subset \R^2$, $n \in \N$, is a collection of bounded, open sets such that  $\partial \cO_n$ is locally connected for all $n \in \N$
  and such that
  \begin{equation}
     \d_H(\cO,\cO_n) + \d_H(\partial \cO,\partial \cO_n) \to 0 \quad \text{as} \quad n \to \infty.
  \end{equation}
  Then, $\cO_n$ converges to $\cO$ in the Mosco sense as $n \to \infty$.
\end{theorem}

The condition $Q(\partial \cO) > 0$ in the above theorem can be replaced by the condition that
each connected component of $\partial \cO$ is path-connected (cf. Remark \ref{rem:Q-hypothesis-replacement}).
In turn, the latter condition is satisfied if $\partial \cO$ is locally connected \cite[\S 16]{milnorDynamicsOneComplex1990}.
A sufficient condition for the hypothesis $\vol(\partial \cO) = 0$ is $\dim_H(\partial \cO) < 2$ where $\dim_H$ denotes the Hausdorff dimension \cite{falconerFractalGeometryMathematical2004}. 
The condition  $\#_c \intt(\cO^c) = \#_c(\cO^c) < \infty$ intuitively states that $\cO$ has a finite number of holes, which neither touch
each other nor the outer boundary component of the domain.

\subsubsection*{Examples}
The following canonical classes of domains satisfy the hypotheses of Theorem \ref{th:mosco}.
The first example includes the  classical Koch snowflake domain.
One could also modify this example to allow for domains with holes. 
\begin{example}[Interior of a Jordan curve]
  \label{ex:jordan-curve}
  Let $C \subset \R^2$ be any Jordan curve with $\vol(C) = 0$. By the Jordan curve theorem, $\R^2 \bs C$ is a disjoint union of two open, connected sets -
  a bounded interior $\cO$ and an unbounded exterior $S_{\text{ext}}$. Then, $\cO$ satisfies the hypotheses of Theorem \ref{th:mosco}.
  \begin{proof}
  It is known that $\partial \cO = C$, hence it holds that $\vol(\partial \cO) = 0$ and $Q(\partial \cO) > 0$.
  It is also known that $\partial(S_{\text{ext}}) = C$, hence any open set $U \subset \R^2$ satisfies either $U \subset \cO$ or $U \cap S_{\text{ext}} \neq \emptyset$.
  From this it follows that $\intt(\overline{\cO}) = \cO$, that is, $\cO$ is regular.
  Similarly, we have that $\intt(\cO^c) = S_{\text{ext}}$ so $\#_c \intt(\cO^c) = \#_c (\cO^c) = 1$.
  \end{proof}
\end{example}
The second example is a concrete special case of the above class of domains and
is the object study in a numerical investigation in Section \ref{sec:numerical-results}.
We naturally identify $\C \cong \R^2$.
\begin{example}[Filled Julia sets with connected interior]\label{ex:filled-julia}
  Let $f_c(z) := z^2 + c$, where $c \in \C$ satisfies $|c| <  \tfrac{1}{4}$.
  Consider the \textit{filled Julia set}
  \begin{equation}
    K(f_c) := \set{z \in \C : (f_c^{\circ n}(z))_{n \in \N} \text{ bounded}}
  \end{equation}
  where $f^{\circ n}(z) := \underbrace{f \circ \cdots \circ f}_{\red{n \text{  times}}} (z)$.
  The domain $\cO = \intt(K(f_c))$ satisfies the hypotheses of Theorem \ref{th:mosco}.
  \begin{proof}
  Firstly, $K(f_c)$ is compact and $\partial \cO = \partial K(f_c) = J(f_c)$, where $J(f_c)$ is the so-called \textit{Julia set} for $f_c$ \cite[Lem. 17.1]{milnorDynamicsOneComplex1990}.
    The Julia set can be thought of as the set of $z \in \C$ for which the dynamics of $f_c^{\circ{n}}(z)$ is chaotic.
    Since $|c| < \tfrac{1}{4}$, it is known that $J(f_c)$ is a Jordan curve \cite[Th. 14.16]{falconerFractalGeometryMathematical2004}.
    By Example \ref{ex:jordan-curve}, it suffices that $\vol(J(f_c)) = 0$.

    One may show that $B_{1/4}(0) \subset K(f_c)$ \cite[Ex. 14.3]{falconerFractalGeometryMathematical2004} hence $|f_c'(z)| > 0$ for every $z \in J(f_c)$,
    that is, there are no critical points on the Julia set.
    It turns out that this is enough to ensure that the Lebesgue measure of the Julia set vanishes (cf. \cite[pg. 2]{Buff_conference}
    and references therein). 
  \end{proof}
\end{example}
On the other hand, consider the \textit{Mandelbrot set}
\begin{equation}
  \label{eq:mandelbrot}
  M := \set{c \in \C : (f_c(0))_{n \in \N} \text{ bounded}}.
\end{equation}
Then, the domains $\cO = \intt(M)$ and $\cO = B_{X}(0) \bs M$ (where $X > \diam(M)$) do \textit{not} satisfy
the hypotheses of Theorem \ref{th:mosco}, since $\#_c\intt(M) = \infty$.
The questions of whether $\vol(\partial M) = 0$ and $\partial M$ is path connected are major open
problems, the latter being implied by the famous MLC conjecture (MLC = Mandelbrot set locally connected) \cite{douadyExploringMandelbrotSet}. 

\subsubsection*{Comparison to known results}

Let us now discuss some related results in the literature.
Firstly, it is known that nested approximations converge in the Mosco sense, that is,
for $\cO \subseteq \R^d$ and $\cO_n \subseteq \R^d$, $n \in \N$, open we have
\begin{equation*}
  \forall n \in \N:\, \cO_n \subseteq \cO_{n+1} \subseteq \cO \quad \text{and}\quad  \cO = \bigcup_{n=1}^\infty \cO_n \quad \Rightarrow \quad \cO_n \Mto \cO \quad \text{as}\quad  n \to \infty.
\end{equation*}
For non-nested approximations, such as those we consider in our study of the computational eigenvalue problem,
Mosco convergence is more difficult to prove.

An open set $\cO \subset \R^d$ is said to be \textit{stable} if \cite[Defn. 5.4.1]{Daners_Domain}
\begin{equation}\label{eq:stability}
  H^1_0(\cO) = H^1_0(\overline{\cO}) := \set{u|_ \cO: u \in H^1(\R^d),\, u = 0 \,\text{ a.e. on } \overline{\cO}^c}.
\end{equation}
This notion allows for the application of powerful spectral convergence results
\cite{rauchPotentialScatteringTheory1975,danersDirichletProblemsVarying2003,dancerRemarksClassicalProblems1996a}. 
It may be characterised in terms of stability of the Dirichlet problems \cite{Arendt2008}
 and in terms of capacities \cite{danersDirichletProblemsVarying2003}\cite[Ch. 11]{Hedberg2012}.
A sufficient geometric condition that ensures that a domain is stable is that it is bounded and
the boundary is locally the image of a continuous map \cite[Prop. 2.2]{Arendt2008}.
Spectral convergence results, with convergence rates, have also been obtained for the Laplacian on Reifenberg-flat domains \cite{lemenantSpectralStability2013}
and for more general non-negative, elliptic, self-adjoint operators \cite{daviesSharpBoundaryEstimates2000,burenkovSpectralStabilityNonnegative2008} .
As far as we are aware these results are not applicable to non-nested approximations of domains with fractal boundary. 

\begin{Red}
Recently, there has been a wealth of activity in the study of Mosco convergence for rough domains. 
In \cite{hinz-eps-inf}, the authors consider quadratic forms for the Helmholtz equation (with mixed boundary conditions) on $(\epsilon, \infty)$-uniform domains,
and prove a Mosco convergence result.
Note that the definition of Mosco convergence for forms \cite[Defn 2.1.1]{moscoCompositeMediaAsymptotic1994} differs from Definition \ref{defn:mosco-convergence}
but also implies spectral convergence results, for the associated operators \cite[Col. 2.7.1]{moscoCompositeMediaAsymptotic1994}.
Also, in \cite{HewettDensity}, the authors prove stability (in essentially the sense of (\ref{eq:stability})) for a variety of Sobolev spaces,
including $H^1_0$, in the case that the domain is thick in the Triebel sense.
Although there are various examples of $(\epsilon, \infty)$ and thick domains with fractal boundaries,
there exists domains satisfying the hypotheses of Theorem \ref{th:mosco} that do not satisfy these assumptions,
in particular, domains with cusps \cite[Remark 3.7 and Remark 4.9]{Triebel}. 
Apart from the computational questions we consider, there are a variety of applications of such Mosco convergence results,
including shape optimisation \cite{hinz-eps-inf} and numerical methods for scattering by fractal screens \cite{chandler-wildeBoundaryElementMethods2021}.  
\end{Red}

\subsection{An explicit Poincar\'e-type inequality}\label{subsec:poin-intro}

A key ingredient for the proof of Theorem \ref{th:mosco} is a Poincar\'e-type inequality for collar neighbourhoods of the boundary of a domain.
Theorem \ref{thm:poin} provides a bound with an explicit constant which is independent of the particular domain $\cO$. 	
As far as the authors are aware, this is the first Poincar\'e-type inequality of its form to be reported.
The proof is given in Section \ref{sec:proof-poin}.
Recall that $Q(\partial \cO)$ is defined by (\ref{eq:min-comp-diam}).
\begin{theorem}\label{thm:poin}
  Let $\cO \subseteq \R^2$ be any open set.
  If $Q(\partial \cO) > 0$ and $r > 0$ satisfies $4 \sqrt{2} r < Q(\partial \cO)$, then
  \begin{equation}\label{eq:poin-ineq-intro}
    \norm{u}_{L^2(\col{r}\cO)} \leq 10 \sqrt{3} r \norm{\nabla u}_{L^2(\col{2\sqrt{2}r}\cO)}
  \end{equation}
  for all $u \in H^1_0(\cO)$.
\end{theorem}
\red{
Note that the hypotheses for this result are weaker than those of Theorem \ref{th:mosco}, in particular, we do not require boundedness of the domain. 
}

\subsubsection*{Comparison to known results}

Precise bounds have recently been obtained in terms of \emph{Hardy inequalities}
(see \cite{balinsky, barbatis, kinnunen, brezis_hardy, ward} and the references therein).
These are bounds on the $L^p$ norm of $ u /\eta$ in terms of $\nabla u$, where $u\in W^{1,p}_0(\cO)$ and $\eta(x) = \dist(x,\del\cO)$.
Classically, the domain $\cO$ is assumed to be of class $C^1$, but relaxations are possible (see \cite{balinsky, ward} for an overview).
Hardy-type inequalities have also been studied in connection with questions of spectral convergence \cite{daviesSharpBoundaryEstimates2000}.

We mention the following result from \cite{kinnunen}\cite[Th. 3.4.8]{ward}.

\vspace{3pt}
\noindent
\textit{If $d\geq 2$ and $\cO\subset\R^d$ is open, connected, such that $\R^d\setminus\cO$ is connected and unbounded, then there exists $C>0$ such that for all $u\in W^{1,d}_0(\cO)$
	\begin{align}\label{eq:ward}
		\left\| \f{u}{\eta} \right\|_{L^d(\cO)} \leq C \left\| \nabla u \right\|_{L^d(\cO)}.
	\end{align}
        }

\noindent
The connectedness assumption on $\R^d\setminus\cO$ can be replaced by the weaker, technical condition of so-called \emph{uniform $m$-fatness} (cf. \cite[Th. 4.1]{kinnunen})
Applying inequality (\ref{eq:ward}) to the case $d=2$ immediately yields the bound $\| u\|_{L^2(\del^r\cO)} \leq Cr \| \nabla u \|_{L^2(\cO)}$ for a $r$-collar neighbourhood of $\del\cO$.
Note that this statement is weaker than Theorem \ref{thm:poin} in two ways: first, the constant $C$ is neither explicit, nor independent of $\cO$ and second, the $L^2$-norm of $\nabla u$ is over the entire domain $\cO$, rather than a neighbourhood of $\del\cO$.
These differences are key for application to our proof of Theorem \ref{th:mosco}.

 \subsection{Computational problems and arithmetic algorithms}\label{sec:comp-probl-arithm}

The theory of the Solvability Complexity Index (SCI) hierarchy was developed in \cite{AHCS2020, Hansen2011}, building on works of Smale, McMullen and Doyle \cite{smale1981,smale1985, mcmullen1987, mcmullen1988, doyle1989}. Broadly speaking, it studies the question \emph{Given a class $\Omega$ of computational problems, can the solutions always be computed by an algorithm?} In order to give a rigorous formulation of this question, it is necessary to introduce precise definitions of the terms ``computational problem'' and ``algorithm'' (the reader may think of a Turing machine for the time being). We will give a brief review of the central elements of the theory here and refer to \cite{AHCS2020, scishort} for further details.

\begin{Red}

\begin{de}[Computational problem]\label{def:computational_problem}
	A \emph{computational problem} is a quadruple $(\Omega,\Lambda,\mathcal M,\Xi)$, where 
	\begin{enumerate}
		\item[(i)] $\Omega$ is a set, called the \emph{primary set},
		\item[(ii)] $\Lambda$ is a set of complex-valued functions on $\Omega$, called the \emph{evaluation set},
		\item[(iii)] $\mathcal M$ is a metric space,
		\item[(iv)] $\Xi:\Omega\to \mathcal M$ is a map, called the \emph{problem function}.
	\end{enumerate}
\end{de}
      
Intuitively, elements of the primary set $\Omega$ are the objects giving rise to the computational problems, the evaluation set $\Lambda$ represents the information available
to an algorithm, the metric space $\mathcal{M}$ is the possible outputs of an algorithm and the problem function $\Xi$ represents the true solutions of the computational problems.

\begin{example}\label{exmaple:SCI}
	An instructive example of a computational problem in the sense of Definition \ref{def:computational_problem} is given by the following data. Let $\Om=\mathcal B(\ell^2(\N))$, the bounded operators on the space of square summable sequences, $\Lambda = \{A \mapsto \langle Ae_i,e_j\rangle_{\ell^2}\, :\, i,j\in\N\}$ the set of matrix elements in the canonical basis, $\mathcal M = (\operatorname{comp}(\C),d_{\mathrm H})$ the compact subsets of $\C$, together with the Hausdorff distance $d_{\mathrm H}$, and finally $\Xi(A)=\sigma(A)$, the spectrum of an operator. In words, this computational problem reads ``Compute the spectrum of a bounded operator on $\ell^2(\N)$ using its matrix entries as an input.''
\end{example}

Now we are in position to define the notion of a \emph{general algorithm}. The next definition is rather generic in nature in order to capture all instances of what is generally thought of as a computer algorithm.
\begin{de}[General algorithm]\label{def:Algorithm}
	Let $(\Omega,\Lambda,\mathcal M,\Xi)$ be a computational problem. A \emph{general algorithm} is a mapping $\Gamma:\Om\to\mathcal M$ such that for each $T\in\Om$ 
	\begin{enumerate}
		\item[(i)] there exists a finite (non-empty) subset $\Lambda_\Gamma(T)\subset\Lambda$,
		\item[(ii)] the action of $\Gamma$ on $T$ depends only on $\{f(T)\}_{f\in\Lambda_\Gamma(T)}$,
		\item[(iii)] for every $S\in\Om$ with $f(T)=f(S)$ for all $f\in\Lambda_\Gamma(T)$ one has $\Lambda_\Gamma(S)=\Lambda_\Gamma(T)$.
	\end{enumerate}
\end{de}
We will sometimes write $\Gamma(T) = \Gamma(\{f(T)\}_{f\in\Lambda_\Gamma(T)})$ to emphasise point (ii) above: the output $\Gamma(T)$ depends only on the (finitely many) evaluations $\{f(T)\}_{f\in\Lambda_\Gamma(T)}$.
\addtocounter{example}{-1}
\begin{example}[continued]
	The solvability of the computational problem defined in this example is thus equivalent to the existence of a sequence of algorithms $(\Gamma_n)_{n\in\N}$, where $\Gamma_n:\mathcal B(\ell^2(\N))\to \operatorname{comp}(\C)$ such that (i)-(iii) of Definition \ref{def:Algorithm} are satisfied and $d_{\mathrm H}(\Gamma_n(A),\sigma(A))\to 0$ as $n\to\infty$ for all $A\in \mathcal B(\ell^2(\N))$. In particular, for each fixed $n\in\N$, the image $\Gamma_n(A)$ must be computable from finitely many matrix elements of $A$.
\end{example}
In \cite{Hansen2011}, Hansen showed that it \emph{is} possible to compute $\sigma(A)$ for $A\in\mathcal B(\ell^2(\N))$ as above. However, rather than having algorithms $\Gamma_n$ with a single index $n\in\N$,   \emph{three} indices were required, satisfying $\sigma(A)=\lim_{n_3\to\infty}\lim_{n_2\to\infty}\lim_{n_1\to\infty}\Gamma_{n_1,n_2,n_3}(A)$. The algorithms $\Gamma_{n_1,n_2,n_3}$ are given explicitly, and can be implemented numerically.  In \cite{AHCS2020} it was proved that this is optimal: this computation cannot be performed with fewer than $3$ limits.

We formalise the foregoing example with the following definitions:
\begin{de}[Tower of general algorithms]\label{def:Tower}
	Let $(\Omega,\Lambda,\mathcal M,\Xi)$ be a computational problem. A \emph{tower of general algorithms} of height $k$ for $(\Omega,\Lambda,\mathcal M,\Xi)$ is a family $\Gamma_{n_k,n_{k-1},\dots,n_1}:\Omega\to\mathcal M$ of general algorithms (where $n_i\in\N$ for $1\leq i \leq k$) such that for all $T\in\Om$
	\begin{align*}
		\Xi(T) = \lim_{n_k\to+\infty}\cdots\lim_{n_1\to+\infty}\Gamma_{n_k,\dots,n_1}(T).
	\end{align*}
\end{de}
\begin{de}[Recursiveness]\label{def:recursive}
Suppose that for all $f\in\Lambda$ and for all $T\in\Omega$ we have $f(T)\in \R$ or $\C$. We say that $\Gamma=\Gamma_{n_k,n_{k-1},\dots,n_1}$ is \emph{recursive} if $\Gamma_{n_k,n_{k-1},\dots,n_1}(\{f(T)\}_{f\in\Lambda_\Gamma(T)})$ can be executed by a Blum-Shub-Smale (BSS) machine \cite{BSS} that takes $(n_1,n_2,\dots,n_k)$ as input and that has an oracle that can access  $f(T)$ for any $f\in\Lambda$.
\end{de}

\begin{de}[Tower of arithmetic algorithms]\label{def:Arithmetic-Tower}
Given a computational problem $(\Omega,\Lambda,\mathcal M,\Xi)$, where $\Lambda$ is countable, an \emph{tower of arithmetic algorithms}  for $(\Omega,\Lambda,\mathcal M,\Xi)$ is a general tower of algorithms where the lowest mappings $\Gamma_{n_k,\dots,n_1}:\Omega\to\mathcal{M}$ satisfy the following:
For each $T\in\Omega$ the mapping $\N^k\ni(n_1,\dots,n_k)\mapsto  \Gamma_{n_k,\dots,n_1}(T)=\Gamma_{n_k,\dots,n_1}(\{f(T)\}_{f\in\Lambda(T)})$ is recursive, and $\Gamma_{n_k,\dots,n_1}(T)$ is a finite string of complex numbers that can be identified with an element in $\mathcal{M}$.
\end{de}
A tower of arithmetic algorithms of height 0 corresponds to a single map $\Gamma: \Omega \to \mathcal{M}$ and is referred to simply as an \emph{arithmetic algorithm}.
\begin{remark}[Types of towers]
One can define many types of towers, see \cite{AHCS2020}. In this paper we write \emph{type $G$} as shorthand for a tower of \emph{general} algorithms, and \emph{type $A$} as shorthand for a tower of \emph{arithmetic} algorithms. If a tower $\{\Gamma_{n_k,n_{k-1},\dots,n_1}\}_{n_i\in\N,\ 1\leq i\leq k}$ is of type $\tau$ (where $\tau\in\{A,G\}$ in this paper) then we write
	\begin{equation*}
	\{\Gamma_{n_k,n_{k-1},\dots,n_1}\}\in\tau.
	\end{equation*}
\end{remark}
\begin{de}[SCI]\label{de:SCI}
	A computational problem $(\Omega,\Lambda,\mathcal M,\Xi)$ is said to have a \emph{Solvability Complexity Index (SCI)} of $k$ with respect to a tower of algorithms of type $\tau$ if $k$ is the smallest integer for which there exists a tower of algorithms of type $\tau$ of height $k$ for $(\Omega,\Lambda,\mathcal M,\Xi)$. We then write 
	\begin{align*}
 		\SCI(\Omega,\Lambda,\mathcal M,\Xi)_\tau=k.
	\end{align*}
	If there exists a tower $\{\Gamma_n\}_{n\in\N}\in\tau$ and a finite $N_1\in\N$ such that $\Xi=\Gamma_{N_1}$ then we define $\SCI(\Omega,\Lambda,\mathcal M,\Xi)_\tau = 0$.
\end{de}
Definition \ref{de:SCI} naturally places computational problems into a \emph{hierarchy}: the higher the SCI of a problem, the more limits are needed to solve it, thus the higher its computational complexity.
\begin{de}[The SCI Hierarchy]
\label{1st_SCI}
The \emph{$\SCI$ Hierarchy} is a hierarchy $\{\Delta_k^\tau\}_{k\in{\N_0}}$ of classes of computational problems $(\Om,\Lambda,\Xi,\mathcal M)$, where each $\Delta_k^\tau$ is defined as the collection of all computational problems satisfying:
\begin{align*}
(\Om,\Lambda,\Xi,\mathcal M)\in\Delta_0^\tau\quad &\qquad\Longleftrightarrow\qquad \mathrm{SCI}(\Om,\Lambda,\Xi,\mathcal M)_\tau= 0,\\
(\Om,\Lambda,\Xi,\mathcal M)\in\Delta_{k+1}^\tau &\qquad\Longleftrightarrow\qquad \mathrm{SCI}(\Om,\Lambda,\Xi,\mathcal M)_\tau\leq k,\qquad k\in\N,
\end{align*}
with the special class $\Delta_1^\tau$  defined as the class of all computational problems in $\Delta_2^\tau$ for which we have explicit error control:
\begin{equation*}
(\Om,\Lambda,\Xi,\mathcal M)\in\Delta_{1}^\tau \qquad\Longleftrightarrow\qquad
 \exists  \{\Gamma_n\}_{n\in \mathbb{N}}, \quad\text{ s.t. }\quad \forall  T\in\Omega, \ d(\Gamma_n(T),\Xi(T)) \leq 2^{-n}.
\end{equation*}
Hence we have that $\Delta_0^\tau\subset\Delta_1^\tau\subset\Delta_2^\tau\subset\cdots$.
\end{de}

When the metric space $\mathcal{M}$ has certain ordering properties, one can define further classes that take into account convergence from below/above and associated error bounds. In order to not burden the reader with unnecessary definitions, we provide the definition that is relevant to the case where $\mathcal{M}$ is the space of closed  subsets of $\R^d$ together with the Attouch-Wets distance \cite{beerTopologies}  (for a more comprehensive and abstract definition we refer to \cite{AHCS2020}), which is defined as follows:

\begin{de}[Attouch-Wets distance]\label{def:att-wet} Let $A,B$ be closed, nonempty sets in $\R^d$. The  \emph{Attouch-Wets distance} between them is defined as
\begin{align}\label{eq:att-wet}
	d_{\text{AW}}(A,B) = \sum_{k=1}^\infty 2^{-k}\min\left\{ 1\,,\,\sup_{|x|<k}\left| \dist(x,A) - \dist(x,B) \right| \right\}.
\end{align} 
Note that if $A,B\subset\R^d$ are bounded, then $d_{\text{AW}}$ is equivalent to the Hausdorff distance.
\end{de}
\begin{de}[The SCI Hierarchy (Attouch-Wets  metric)]
\label{def:pi-sigma}
Consider the setup in Definition \ref{1st_SCI} assuming further that $\mathcal{M}=(\mathrm{cl}(\R^d),d_{\mathrm{AW}})$.   Then for $k\in\N$ we can define the following subsets of $\Delta_{k+1}^\tau$:
\begin{align}
\begin{split}\label{eq:Sigma_k_def}
	\Sigma_{k}^\tau
	=
	\Big\{(\Om,\Lambda,\Xi,\mathcal M) \in \Delta_{k+1}^\tau \ &: \  \exists\{ \Gamma_{n_k,\dots,n_1}\}\in\tau\text{ s.t. }    \forall T \in \Omega,\, \exists \{X_{n_k}(T)\}\subset\mathcal{M}, \text{ s.t. } 
	\\
	&\; \lim_{n_k\to\infty}\cdots\lim_{n_1\to\infty}\Gamma_{n_k,\dots,n_1}(T)=\Xi(T),
	\\
	& \lim_{n_{k-1}\to\infty}\!\cdots\lim_{n_1\to\infty}\Gamma_{n_k,\dots,n_1}(T)\subset X_{n_k}(T),
	\\
	&\; d\left(X_{n_k}(T),\Xi(T)\right)\leq 2^{-n_k} \Big\},
\end{split}
	\\[1mm]
\begin{split}\label{eq:Pi_k_def}
	\Pi_{k}^\tau
	=
	\Big\{(\Om,\Lambda,\Xi,\mathcal M) \in \Delta_{k+1}^\tau \ &: \ \exists\{ \Gamma_{n_k,\dots,n_1}\}\in\tau\text{ s.t. }    \forall T \in \Omega,\, \exists \{X_{n_k}(T)\}\subset\mathcal{M}, \text{ s.t. } 
	\\
	& \lim_{n_k\to\infty}\cdots\lim_{n_1\to\infty}\Gamma_{n_k,\dots,n_1}(T)=\Xi(T),
	\\
	&\; \Xi(T)\subset X_{n_k}(T),
	\\
	&\; d\Bigl(X_{n_k}(T),\lim_{n_{k-1}\to\infty}\cdots\lim_{n_1\to\infty}\Gamma_{n_k,\dots,n_1}(T)\Bigr)\leq 2^{-n_k} \Big\}.
\end{split}
\end{align}
It can be shown that $\Delta_k^\tau=\Sigma_k^\tau\cap\Pi_k^\tau$ for $k\in\{1,2,3\}$ (we  refer to \cite{AHCS2020} for a detailed treatise).
\end{de}
Informally, these sets can be characterised as follows:

	\begin{itemize}
	\item[$\Delta_k^\tau:$]
	For $k\geq2$, $\Delta_k^\tau$ is the class of problems that require  at most $k-1$ successive limits to solve. We also say that these problem have an $\SCI$ value of at most $k-1$. Problems in $\Delta_1^\tau$ can be solved in one limit with known error bounds.
	\item[$\Sigma_k^\tau:$]
	For all $k\in\N$, $\Sigma_k^\tau\subset\Delta_{k+1}^\tau$ is the class of problems in $\Delta_{k+1}^\tau$ that can be approximated from ``below'' with known error bounds. 
	\item[$\Pi_k^\tau:$]
	For all $k\in\N$, $\Pi_k^\tau\subset\Delta_{k+1}^\tau$ is the class of problems in $\Delta_{k+1}^\tau$ that can be approximated from ``above'' with known error bounds.
	\end{itemize}
By an approximation from ``above''  (resp. ``below'') we mean that the output of the algorithm is a superset (resp. subset) of the object we are computing (this clearly requires that this object and its approximations  belong to a certain  topological space).

\end{Red}

Several kinds of computational (spectral and other) problems have been classified in the SCI hierarchy in recent years, not just in the abstract bounded setting of Example \ref{exmaple:SCI}, but also in more applied PDE problems. Recent results include classification of abstract spectral problems \cite{AHCS2020, Colbrook2019a}, spectral problems (forward and inverse) for PDEs on $\R^d$ \cite{AHCS2020,Ben-Artzi:2022aa,BMR_inv22,Colbrook2019,Colbrook2019c,RoslerSchrod, RT22}, resonance problems for potential scattering \cite{BMR2020} and obstacle scattering \cite{seashell}. Furthermore, SCI has been applied to other problems, such as those arising in AI \cite{sciai2,sciai1,sciai3}. The computability of spectral problems on domains in $\R^d$ and its relation to boundary regularity has not yet been studied as far as the authors are aware.

%



\subsection{Computational eigenvalue problem for the Laplacian}\label{subsec:comp_spec_prob}

Now we describe our contribution to the SCI hierarchy.
\subsubsection*{Statement of SCI results}
We shall consider the following computational problem.
\begin{enumerate}[label= \rm(\Alph*)]
 \item \red{The primary set is the set of bounded domains,}
 \begin{equation*}
  \Omega_0 := \set*{\cO \subset \R^2 : \cO \text{ open, bounded and connected} }.
 \end{equation*}
\item 
The evaluation set is 
\begin{equation*}
 \Lambda_0 := \set*{ \cO \mapsto \chi_\cO(x) : x \in \R^2 }
\end{equation*}
where $\chi$ is the characteristic function.
\item 
  The metric space is $\mathcal{M} := (\textrm{cl}(\C), \dAWonly)$, where $\cl(\C)$ denotes the set of closed, nonempty subsets of $\C$ and $\dAWonly$ denotes the
  Attouch-Wets metric.
  Note that the spectrum of the Dirichlet Laplacian on a bounded domain is always closed and nonempty by classical results.
\item The problem function $\Xi_\sigma: \Omega \to \mathcal{M}$ is defined by $\Xi_\sigma(\cO) := \sigma(\cO)$, where recall that $\sigma(\cO)$ denotes the spectrum of the Dirichlet Laplacian $-\Delta_\cO$ on $L^2(\cO)$.
\end{enumerate}

The following result follows immediately from Proposition \ref{prop:counterexample}.
The proof is based on the construction of a certain counter-example which ``fools" a sequence of arithmetic algorithms aiming to compute the spectrum
on an arbitrary domain in $\Omega_0$.

\begin{Red}
\begin{prop}\label{prop:spec-not-exist}
 There does not exist a sequence of general algorithms  $\Gamma_n:\Omega_0 \to \cl(\C)$  which satisfy
 \begin{equation*}
    \dAW{\Gamma_n(\cO)}{\sigma(\cO)} \to 0 \quad \text{as} \quad n \to \infty \qquad \text{for all} \qquad  \cO \in \Omega_0. 
  \end{equation*}
  That is,
  \begin{equation*}
    (\Omega_0, \Lambda_0, \mathcal{M},\Xi_\sigma) \notin\Delta_2^G.
  \end{equation*}
\end{prop}
\end{Red}

Our final result is an explicit construction of a sequence of arithmetic algorithms,
describing a simple numerical method for the computation of eigenvalues of the Dirichlet Laplacian on a large class of bounded domains.
Recall that $Q(\partial \cO)$ is defined by (\ref{eq:min-comp-diam}) and $\#_c$ is the number of connected components.

\begin{Red}
\begin{theorem}\label{thm:spec-exist}
 Let
 \begin{equation}\label{eq:Omega1_def}
  \Omega_1 := \Bigg\{ \cO \in \Omega_0\; :\; \begin{aligned}
 	 &\cO = \intt(\overline{\cO}),\,\vol(\partial \cO)=0,\, Q(\partial \cO) > 0 \textnormal{ and }
 	 \\
 	 &\#_c \intt(\cO^c) = \#_c (\cO^c) < \infty 
 \end{aligned}
 \Bigg\}.
 \end{equation} 
\begin{enumerate}
\item[(i)]	
 There exists a sequence of arithmetic algorithms $\Gamma_n:\Omega_1 \to \cl(\C)$  such that 
 \begin{equation*}
  \dAW{\Gamma_n(\cO) }{\sigma(\cO)} \to 0 \quad \text{as} \quad n \to \infty \qquad \text{for all} \qquad  \cO \in \Omega_1,
\end{equation*}
that is, 
\begin{equation*}
  (\Omega_1, \Lambda_0, \mathcal{M},\Xi_\sigma) \in\Delta_2^A.
\end{equation*}
\item[(ii)]
The above result is sharp in the sense that explicit error control is neither possible from below, nor from above, i.e.
\begin{equation*}
  (\Omega_1, \Lambda_0, \mathcal{M},\Xi_\sigma) \notin\Sigma_1^G\cup\Pi_1^G.
\end{equation*}
\end{enumerate}
\end{theorem}
\begin{proof}[Proof of (ii)]
	We give the proof of part (ii) here. The proof of (i) requires more effort and will be done in the subsequent sections.

	\noindent
 \textit{Step 1.}
	We prove by contradiction that $(\Omega_1, \Lambda_0, \mathcal{M},\Xi_\sigma) \notin\Sigma_1^G$. To this end, assume that there exists a general tower $(\Gamma_n)_{n\in\N}$ of height 1 and a sequence of sets $X_n$ as in \eqref{eq:Sigma_k_def}. Then in particular $\Gamma_n(\cO)\subset X_n(\cO)$ and $d_{\text{AW}}\left(X_{n}(\cO),\sigma(\cO)\right)\leq 2^{-n}$ for all $n\in\N$ and all $\cO\in\Omega_1$. Let $\cO_0:=(0,\pi)^2$ so that the lowest Dirichlet eigenvalue of $\cO_0$ is $\lambda_1(\cO_0)=2$. 
By the convergence of $\Gamma_n$ in the Attouch-Wets metric, we can choose $n_0$ such that 
	\begin{align}\label{eq:BcapGamma}
		B_{\f12}(2)\cap\Gamma_{n_0}(\cO_0)\neq\emptyset.
	\end{align}
	The finite set $\Lambda_{\Gamma_{n_0}}(\cO_0)$ can be written as $\{\chi_{\bullet}(x_1),\dots,\chi_{\bullet}(x_{N})\}$ for finitely many points $x_1,\dots,x_{N}\in\R^2$. Next, define a new domain as follows. Denote $\{y_1,\dots,y_M\}:=\{x_1,\dots,x_{N}\}\cap\cO_0$ and let $l$ be any finite path of linear line segments that connects all $y_i$ (since $\cO_0$ is convex we necessarily have $l\subset\cO_0$). Choose $\eps>0$ small enough to ensure that $\dil_\eps(l)\subset\cO_0$. Finally, define $\cO_\eps := \dil_\eps(l)$. By construction we have $\cO_\eps\in\Omega_1$ and $\Gamma_{n_0}(\cO_\eps) = \Gamma_{n_0}(\cO_0)$ for all $\eps>0$. However, as $\eps\to 0$ we have $\lambda_1(\cO_\eps)\to+\infty$ (this can be seen from Theorem \ref{thm:poin}, for instance). Therefore we can choose $\eps$ small enough such that
	\begin{align}\label{eq:BcapSigma}
		B_{1}(2)\cap\sigma(\cO_\eps) = \emptyset.
	\end{align}
	By \eqref{eq:BcapGamma} any $X_{n_0}(\cO_\eps)$ satisfying $\Gamma_{n_0}(\cO_\eps)\subset X_{n_0}(\cO_\eps)$ necessarily contains a point in $B_{1/2}(2)$. However, \eqref{eq:BcapSigma} implies that the condition $d_{\text{AW}}\left(X_{n_0}(\cO_\eps),\sigma(\cO_\eps)\right)\leq 2^{-n_0}$ cannot be satisfied, which contradicts our assumption.

        \noindent
 \textit{Step 2.} We prove by contradiction that $(\Omega_1, \Lambda_0, \mathcal{M},\Xi_\sigma) \notin\Pi_1^G$. To this end, assume that there exists a general tower $(\Gamma_n)_{n\in\N}$ of height 1 and a sequence of sets $X_n$ as in \eqref{eq:Pi_k_def}. Then in particular $\sigma(\cO)\subset X_n(\cO)$ and $d_{\text{AW}}\left(X_{n}(\cO),\Gamma_n(\cO)\right)\leq 2^{-n}$ for all $n\in\N$ and all $\cO\in\Omega_1$. 
	Let $\cO_0:=(0,\pi)^2$ so that the lowest Dirichlet eigenvalue of $\cO_0$ is $\lambda_1(\cO_0)=2$. Then by assumption we have $d_{\text{AW}}\left(X_2(\cO_0),\sigma(\cO_0)\right)\leq \f14$ and therefore $\Gamma_2(\cO_0)\cap B_1(0) \subset X_2(\cO_0)\cap B_1(0)=\emptyset$. 
	Indeed, if there existed a point $p\in X_2(\cO_0)\cap B_1(0)$, then by \eqref{eq:att-wet} one would have
	\begin{align}\label{eq:dAW_calculation}
	\begin{split}
		d_{\text{AW}}\left(X_2(\cO_0),\sigma(\cO_0)\right) &= \sum_{k=1}^\infty 2^{-k}\min\bigg\{ 1\,,\,\sup_{|y|<k}\left| \dist(y,X_2(\cO_0)) - \dist(y,\sigma(\cO_0)) \right| \bigg\}
		\\
		&\geq \sum_{k=1}^\infty 2^{-k}\min\big\{ 1\,,\,\left| \dist(p,X_2(\cO_0)) - \dist(p,\sigma(\cO_0)) \right| \big\}
		\\
		&= \sum_{k=1}^\infty 2^{-k}\min\big\{ 1\,,\,\dist(p,\sigma(\cO_0)) \big\}
		\\
		&= \sum_{k=1}^\infty 2^{-k}\min\big\{ 1\,,\,|p-2| \big\}
		\\
		&= 1.
	\end{split}
	\end{align}
	Using convergence of  $\Gamma_n$, choose $n$ large enough such that $d_{\text{AW}}(\Gamma_n(\cO_0),\sigma(\cO_0))<\f14$ (without loss of generality, $n\geq 2$). Then necessarily $\Gamma_n(\cO_0)\cap B_1(0)=\emptyset$. 
	The finite set $\Lambda_{\Gamma_n}(\cO_0)$ can be written as $\{\chi_{\bullet}(x_1),\dots,\chi_{\bullet}(x_{N})\}$ for finitely many points $x_1,\dots,x_{N}\in\R^2$. Next, define a new domain $\cO_2$ as follows. 
	Let  $\cO_1:= (0,2\pi)^2 + p_R$, where $p_R$ is the point $(R,0)\in \R^2$ and  $R:=2\max\{|x_1|,\dots,|x_{N}|\}$. Because the set $\{x_1,\dots,x_{N}\}$ is finite, there exists a straight line segment $l$ connecting $\cO_0$ and $\cO_1$ and $\eps>0$ such that $\dil_\eps(l)\cap \{x_1,\dots,x_{N}\} = \emptyset$. Thus if we define $\cO_2 := \cO_0\cup\dil_\eps(l)\cup\cO_1$, then $\chi_{\cO_0}(x_i) = \chi_{\cO_2}(x_i)$ for all $i\in\{1,\dots,N\}$ and consequently $\Gamma_n(\cO_2)=\Gamma_n(\cO_0)$. But by domain monotonicity we have $\lambda_1(\cO_2)\leq\lambda_1(\cO_1)=\f12$. 
	Now, by definition of $\Pi_1^G$ there must exist $X_n(\cO_2)$ with 
	\begin{align}
		\sigma(\cO_2) &\subset X_{n}(\cO_2)
		\label{eq:sigma_subset_X}
		\\
		d_{\text{AW}}\left(X_{n}(\cO_2),\Gamma_n(\cO_2)\right) &\leq 2^{-n}.
		\label{eq:dAW(X,Gamma)}
	\end{align}
	Since $\Gamma_n(\cO_2)\cap B_1(0)=\Gamma_n(\cO_0)\cap B_1(0)=\emptyset$ and $\lambda_1(\cO_2)\in[0,\f12]$, equation \eqref{eq:sigma_subset_X} and a calculation similar to \eqref{eq:dAW_calculation} implies that $d_{\text{AW}}(X_{n}(\cO_2),\Gamma_n(\cO_2)) \geq \min\{1,\dist(\lambda_1(\cO_2), \Gamma_n(\cO_2))\} \geq \f12$, which contradicts \eqref{eq:dAW(X,Gamma)}.
%
\end{proof}
\end{Red}

\noindent
Note that any domain described in Example \ref{ex:jordan-curve} or \ref{ex:filled-julia} lies in $\Omega_1$.

The arithmetic algorithms in the above theorem are based on the following approximation for a Euclidean domain.
\begin{figure}
	\centering
	\includegraphics[width=0.8\textwidth]{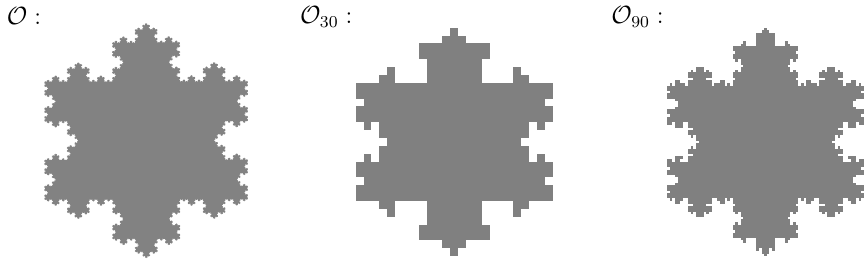}
	\caption{Sketch of a domain $\cO$ and its pixelated analogue $\cO_n$}
\end{figure}
\begin{de}\label{defn:pixel}
For any open set $\cO \subseteq \R^d$, \textit{pixelated domains} for $\cO$ are the open sets $\cO_n \subseteq \R^d,\,n \in \N$, defined by 
\begin{align*}
	\cO_n:=\intt \bigg(\bigcup_{j\in L_n} (j + [- \tfrac{1}{2n},\tfrac{1}{2n}]^d)\bigg),
\end{align*}
where
\begin{equation*}
	L_n:=\set*{ j\in \Z_n^d : j \in \cO } \quad \text{and}\quad  \Z_n^d := \br{n^{-1} \Z}^d.
\end{equation*}
\end{de}

The basic idea behind the construction of the algorithm is to combine pixelation approximations of the domain
with computable error bounds for the finite-element method. 
The computable error bounds for the finite element method that we employ are those of Liu and Oishi \cite{LiuOishi2013},
but similar bounds have also been obtained in \cite{cancesGuaranteedRobustPosteriori2018, carstensenGuaranteedLowerBounds2014}.
The algorithm of Theorem \ref{thm:spec-exist} can be summarised as:
\begin{enumerate}
\item[\textit{Step 1}] Approximate $\cO$ by a corresponding pixelated domain $\cO_n$. 
\item[\textit{Step 2}] Approximate the eigenvalues of $\cO_n$ to an error $1/n$ in the Attouch-Wets metric,
  using computable error bounds for the finite element method on a uniform triangulation of $\cO_n$
  and the Jacobi method combined a-posteriori error bounds for the associated matrix pencils.
\end{enumerate}
In Proposition \ref{prop:alg-mosco}, this algorithm is shown to converge on any bounded domains for which 
the pixelation approximations converge in the Mosco sense, that is, for any domain in 
 \begin{equation*}
    \Omega_M := \set*{ \cO \subset \R^2: \cO \text{\rm{ open, bounded and }}\cO_n \Mto \cO \text{\rm{  where }}\cO_n\text{\rm{  pixelated domains for }}\cO }.
  \end{equation*}
In Proposition \ref{prop:pixel}, we show that for the pixelated domains $\cO_n$ for a bounded, open set $\cO \subset \R^d$ satisfying  $\cO = \intt(\overline{\cO})$ and $\vol(\partial \cO) = 0$, we have
  \begin{equation*}
    \d_H(\cO_n,\cO) + \d_H(\partial \cO_n,\partial \cO) \to 0 \quad \text{as} \quad n \to \infty.
  \end{equation*}
Theorem \ref{thm:spec-exist} therefore follows by an application of our general Mosco convergence result Theorem \ref{th:mosco}.   

\subsubsection*{Discussion of SCI results}

Of course, for a specific given domain $\cO \in \Omega_1$ the algorithm of Theorem \ref{thm:spec-exist}  may not necessarily be the most efficient possible.
For instance, the Koch snowflake domain would be better approximated by the standard ``pre-fractal" domains (as in
\cite{specific_Levitin1995,specific_Bangerini2006,prefractal_Lancia2012,lapidusSnowflakeHarmonicsComputer1996,gabbardDiscretizationKochSnowflake2020} for instance) rather than the pixelated domains.
This is because, in this particular case, the pixelated domains have sharper re-entrant corners
hence the finite element method on each approximate domain converges slower.
Numerical methods for the Laplacian have also been devised  based on the smoothness of the boundary \cite{pos_Tiih1997} or on
affine self-similarity properties \cite{achdou}.
The novelty of our approach lies in the fact that we have constructed a \textit{single} algorithm that is capable of computing
\textit{any} domain in $\Omega_1$.

We mention that robust numerical methods have been constructed for specific non-regular domains (cf. \cite[Section 5.1]{LiuOishi2013} for instance).
Further, numerical methods for the Laplacian on more exotic spaces such the Sierpi\'nski triangle
and Julia sets have been studied in \cite{other_Strichartz,berryOuterApproximationSpectrum2009a,flockLaplaciansFamilyQuadratic2012}.

Naturally, any theoretical description of a real-world computational problem is an idealisation to some extent.
We identify two ways in which this is so in our work. 
Firstly, the arithmetic algorithms we consider may perform real-arithmetic computations to infinite precision,
as opposed to floating point operations.
We expect that the construction of the algorithm in Theorem \ref{thm:spec-exist} can be adapted to the floating point case.
Indeed, the results of Oishi \cite{Oishi2001}, which we use for the matrix eigenvalue computation component of our construction,
are formulated in terms of rigorous floating point arithmetic computations.

Secondly, the arithmetic algorithms we study have access to the characteristic function of any domain in the primary set.
The question of computability for characteristic functions of rough Euclidean sets is a subtle one which is sensitive
to the particular model of computation that one considers \cite{zhongRecursivelyEnumerableSubsets1998}.
In \cite{rettingerComputationalComplexityJulia2002}, Rettinger and Weihrauch  have proven that a certain class of Julia sets are computable by a Turing machine.
In fact, these Julia sets are exactly the boundaries of the domains considered in Example \ref{ex:filled-julia}, which belong to $\Omega_1$.
In \cite{Hertling2005}, Hertling has shown that that the Mandelbrot set is also computable
(in an appropriate sense), assuming that the MLC conjecture holds true.

\section{An explicit Poincar\'e-type inequality}\label{sec:poin}

This section is devoted to the proof of the Poincar\'e-type inequality Theorem \ref{thm:poin} on the collar neighbourhood $\col{r}\cO$ of a domain $\cO \subset \R^2$.
Our approach is inspired by the simple proof of the Poincar\'e inequality in the textbook of Adams and the Fournier \cite[Theorem 6.30]{adamsSobolevSpaces2003}.
The method consists in expressing the value of a function in $H^1_0(\cO)$ at a given point $x\in \col{r} \cO$ as an integral over a path from $x$ to the boundary $\partial \cO$.
We shall explicitly construct these paths.
This must be done in a way such that the bundle of paths corresponding to the different points in $\col{r} \cO$
do not ``concentrate'' too much at any given point on the boundary.
This is made possible by the assumption $Q(\del\cO)>0$ (cf. (\ref{eq:min-comp-diam})). In fact, this assumption is necessary, as the following example shows. 
\begin{example}
  Let $1>\eps>0$ and consider the domain $\cO := B_1(0)\setminus B_\eps(0)\subset\R^2$ (hence $Q(\del\cO) = 2\eps$). In polar coordinates, define the function $f_\eps(r) = \f{\log(\eps) - \log(r)}{\log(\eps)}$. An explicit calculation shows 
	\begin{align*}
		\|f_\eps\|_{L^2(\cO)}^2 
		&= 2 \pi \br*{ \frac{1}{2} + \frac{1}{2 \log(\epsilon)} + \frac{1 - \epsilon^2}{4 \log^2(\epsilon)}}		
		\\
		\|\nabla f_\eps\|_{L^2(\cO)}^2 &= -\f{2\pi}{\log(\eps)}
	\end{align*}
	And thus
	\begin{align*}
		\f{\|f_\eps\|_{L^2(\cO)}^2}{\|\nabla f_\eps\|_{L^2(\cO)}^2} \geq C|\log(\eps)|
	\end{align*}
	 as $\eps\to 0$. Since $f_\eps(x) \to 1$ as $x\to \del B_1(0)$, $f$ can be extended to a $H^1_0$ function on any domain that contains $\overline{B_1(0)}\setminus B_\eps(0)$.
	 This example shows that no uniform Poincar\'e inequality can hold on domains with arbitrarily small holes. Similar statements can be proved in higher dimensions (cf. \cite[Lemma 4.5]{rauchPotentialScatteringTheory1975}).
\end{example}

Throughout the section, let $\cO \subset \R^2$ be an arbitrary open set and fix the value $r > 0$,
corresponding to the size of the collar neighbourhood $\col{r} \cO$.

\subsection{Some geometric notions}\label{subsec:some-geom}

The construction of the bundle of paths shall be assisted
by the introduction of a grid of boxes covering $\R^2$.
We choose the boxes to have edge of length $r > 0$ - exactly the size of the collar neighbourhood $\col{r} \cO$.
We shall introduce, for the purpose of the proof, various notions such as \textit{cell-paths}, \textit{g-cells} and \textit{lg-cells}. 
Cell-paths can be thought of as a higher level structure within which the bundles of paths shall be constructed.
Then, g-cells and lg-cells (good cells and long good cells) are cells which $\partial \cO$ intersects in a way
such that a bundle of paths can be terminated at that cell.  

\begin{de}
\begin{enumerate}[label=\rm{(\alph*)},wide, labelindent=0pt]
 \item A \textit{cell} is a closed box $j + [0,r]^2$ for some $j \in (r \Z)^2$.
 \item An \textit{edge} of a cell is one of the 4 connected, closed straight line segments whose union comprises the boundary of the cell. 
\end{enumerate}
\end{de}

\begin{de}\label{de:g-cell}
 A \textit{g-cell} is a cell $c_0$ such that \red{there exist} two distinct, parallel edges $e_1$ and $e_2$ of $c_0$, which are connected by a path-connected segment of $\partial \cO$ in $c_0$, that is, 
 \begin{equation}\label{eq:g-cell-cond}
  \exists \Gamma \subseteq \partial \cO \cap c_0: \Gamma \text{ path-connected}, \Gamma \cap e_1 \neq \emptyset\text{ and }\Gamma \cap e_2 \neq \emptyset.
 \end{equation} 
 \red{Given any two distinct, parallel edges $e_1$ and $e_2$ of $c_0$ satisfying (\ref{eq:g-cell-cond}),
   the remaining two edges are referred to as \textit{normal edges}.}
\end{de}

\begin{remark}
  \red{Note that it is possible that a given cell $c_0$ may be made into a g-cell in more than one way, in the sense that both pairs of distinct,
    parallel edges  satisfy (\ref{eq:g-cell-cond})
    (e.g. if $c_0 \cap \Gamma$ is the union of a vertical and horizontal line).
    We shall always think of a g-cell as having a fixed pair of edges $e_1$ and $e_2$ satisfying (\ref{eq:g-cell-cond}),
    as well as fixed pair of normal edges. }
  
\end{remark}

\begin{de}
\begin{enumerate}[label=\rm{(\alph*)},wide, labelindent=0pt]
 \item A \textit{long-cell} is a set of two cells $\set{c_1,c_2}$ such that $c_1$ and $c_2$ share a common edge. 
 \item  An \textit{edge} of a long-cell $\set{c_1,c_2}$ is one of the 4 connected, closed straight line segments whose union comprises the boundary of the set $c_1 \cup c_2$. 
 \item A \textit{short-edge} of the long-cell is an edge of the long cell which is also an edge of a cell.
 \item A \textit{long-edge} of a long-cell is an edge of the long-cell which is not a short-edge.
\end{enumerate}
\end{de}

\begin{de}\label{de:lg-cell}
 An \textit{lg-cell} is a long cell $\set{c_1,c_2}$ for which there exist distinct long-edges $e_1$ and $e_2$  connected by a path-connected segment of $\partial \cO$ in $c_0$, that is, 
 \begin{equation*}
  \exists \Gamma \subseteq \partial \cO \cap \br*{c_1 \cup c_2} : \Gamma \text{ path-connected}, \Gamma \cap e_1 \neq \emptyset\text{ and }\Gamma \cap e_2 \neq \emptyset.
 \end{equation*}
 The \textit{normal edges} of an lg-cell refers to its short edges. We shall often say that an lg-cell $\set{c_1,c_2}$ is contained in a set $A$ to mean that $c_1 \cup c_2 \subseteq A $.
\end{de}

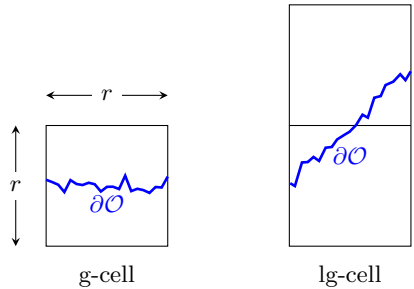
\begin{figure}
%
%

\begin{tikzpicture}[>=stealth, scale=0.8]
	\draw (0,0) rectangle (2,2);
	
	\draw[<->] (0,2.5) -- (2,2.5);
	\draw (1,2.5) node[fill=white]{\small$r$};
	
	\draw[<->] (-0.5,0) -- (-0.5,2);
	\draw (-0.5,1) node[fill=white]{\small$r$};
	
	\draw[blue, line width=1] plot file {figures/bdry.table};
	\draw (1,0.7) node{\small{\color{blue}$\partial\mathcal O$}};
	
	\draw (1,-0.5) node{\small g-cell};

	\draw (4,0) rectangle (6,4);
	\draw (4,2) -- (6,2);
	
	\draw[blue, line width=1] plot file {figures/bdry2.table};
	\draw (5,1.5) node{\small{\color{blue}$\partial\mathcal O$}};
	
	\draw (5,-0.5) node{\small lg-cell};
\end{tikzpicture}

	\caption{Illustration for Definitions \ref{de:g-cell} and \ref{de:lg-cell}.}  
\end{figure}

\begin{de}\label{de:cell-path}
 
 A \textit{cell path} from a cell $c_0$ to a g-cell $c_n$ (or to an lg-cell $\set{c_n,c_{n+1}}$) is a sequence of cells $(c_1,...,c_{n-1})$ such that 
\begin{enumerate}
 \item if $n \geq 2$, then $c_j$ shares a common edge with $c_{j-1}$ for each $j \in \set{1,...,n-1}$,
 \item if $n \geq 1$, then there exists an edge of $c_{n-1}$ which is also a normal edge of the g-cell $c_n$ (or of the lg-cell $\set{c_n,c_{n+1}}$ resp.) and
 \item  $(c_0,...,c_n)$ (or $(c_0,...,c_{n+1})$ resp.) consists of distinct elements.
\end{enumerate}
Here, we allow the possibility that $n=1$,
corresponding to the case that there exists an edge of $c_0$ which is also a normal edge of the g-cell $c_n$ (or of the lg-cell $\set{c_n, c_{n+1}}$ resp.)
and we allow the possibility that $n=0$, corresponding to the case that $c_0$ is itself a g-cell (or in an lg-cell resp.).
In both of these cases, the cell-path is empty. 
\end{de}

\begin{de}
 \begin{enumerate}[label=\rm{(\alph*)},wide, labelindent=0pt]
  \item The \textit{1-cell neighbourhood} $D_1[c_0]$ of a cell $c_0$ is the union of all cells sharing an edge or a corner with $c_0$, that is, 
  \begin{equation*}
    D_1[c_0] = \bigcup \set{c: c \text{ is a cell and } c \cap c_0 \neq \emptyset}.
  \end{equation*}
  \item The \textit{2-cell neighbourhood} $D_2[c_0]$ of a cell $c_0$ is the union of all cells sharing an edge or a corner with $D_1[c_0]$, that is, 
  \begin{equation*}
   D_2[c_0] = \bigcup \set{c: c \text{ is a cell and } c \cap D_1[c_0] \neq \emptyset}. 
 \end{equation*}
 \end{enumerate}
\end{de}

\begin{de}
  \begin{enumerate}[label=\rm{(\alph*)},wide, labelindent=0pt]
  \item A \textit{filled cell} is a cell $c$ such that  $c \cap \partial \cO \neq \emptyset$.
  \item A \textit{covering cell} is a cell $c$ which shares an edge or a corner with a filled cell $c_f$, i.e. $c \cap c_f \neq \emptyset$.
   \end{enumerate}
\end{de}

\subsection{Poincar\'e-type inequality for cell-paths}\label{subsec:poinc-type-ineq}

Given a cell $c_0$ and a cell-path from $c_0$ to either a g-cell or an lg-cell, one may express the value of a function $u \in C_0^\infty(c_0)$
\red{at any point} in $c_0$ as a line integral over a path within the cell path from that point to the boundary (cf. equation (\ref{eq:L2u-line})).
With this representation for $u$, one may proceed in a way similar to the proof of \cite[Theorem 6.30]{adamsSobolevSpaces2003} to obtain a Poincar\'e-type inequality
for $c_0$.

\begin{lemma}\label{lem:tetriminos}
 Let $c_0$ be a cell and let $(c_1,...,c_{n-1})$ be a cell path from $c_0$ to a g-cell $c_n$ or an lg-cell $\set{c_n,c_{n+1}}$. Then, for any $u \in H^1_0(\cO)$,
 \begin{equation}\label{eq:tetrimino-poin}
  \norm{u}^2_{L^2(c_0)} \leq 2 (n + 1) r^2 \sum_{j=0}^{n+1} \norm{\nabla u}^2_{L^2(c_j)}.
 \end{equation}
 In \eqref{eq:tetrimino-poin}, $c_{n+1}$ is considered to be the empty set in the case of a cell path to a g-cell.
\end{lemma}

\begin{proof}

Assume without loss of generality that $c_0 = [0,r]^2$. 
In the case of a cell path to an lg-cell, assume without loss of generality that $c_n$ shares an edge with $c_{n-1}$.
We first deal with the case that $n \geq 1$, so that  $c_0 \neq c_n$ (or $c_0 \notin \set{c_n,c_{n+1}}$ in the case of a cell-path to an lg-cell).
The easier case $n = 0$ will be treated separately.

For each $j \in \set{0,...,n-1}$, let $e_j$ denote the unique edge shared by $c_j$ and $c_{j+1}$ (note that $c_j \neq c_{j+1}$ by the definition of a cell path).
Assume without loss of generality that $e_0 = [0,r] \times \set{0}$.
Let us parameterise each of the edges $e_j$ by $(e_j(s))_{s \in [0,r]}$ such that the path $s \mapsto e_j(s)$ has unit speed.
It suffices to specify the point $e_j(0)$ or the point $e_j(r)$ in order to define the entire parameterisation $(e_j(s))_{s \in [0,r]}$.
\begin{enumerate}
\item Define $(e_0(s))_{s \in [0,r]}$ by $e_0(0) = (0,0)$, so that $e_0(s) = (s,0)$.
\end{enumerate}
If $n = 1$, then we are done. If $n \geq 2$, then the parameterisations are defined recursively as follows.
Note that for each $j \in \set{1,...,n-1}$, we have $c_{j-1} \neq c_{j+1}$ by the definition of a cell-path and so $e_j \neq e_{j-1}$.
\begin{enumerate}
 \item[(2)] For $j \in \set{1,...,n-1}$, if $e_{j-1}$ is parallel to $e_j$, then we call $c_j$ a \textit{straight tile}. In this case, define $(e_j(s))_{s \in [0,r]}$ by the condition
   that $e_j(0)$ is connected by an edge of $c_j$ to $e_{j-1}(0)$, so that $e_j(r)$ is connected by an edge of $c_j$ to $e_{j-1}(r)$.
 \item[(3)] For $j \in \set{1,...,n-1}$, if $e_{j-1}$ is perpendicular to $e_j$, then we call $c_j$ an \textit{corner tile}.
 If $e_j$ and $e_{j-1}$ share the point $e_{j-1}(0)$, then $c_j$ is said to be \textit{positively oriented}. 
In this case, define $(e_j(s))_{s \in [0,r]}$  by the condition that $e_j(0) = e_{j-1}(0)$.
On the other hand, if $e_j$ and $e_{j-1}$ share the point $e_{j-1}(r)$, then $c_j$ is said to be \textit{negatively oriented}.
In this case, define $(e_j(s))_{s \in [0,r]}$ by the condition that $e_j(r) = e_{j-1}(r)$.
\end{enumerate}
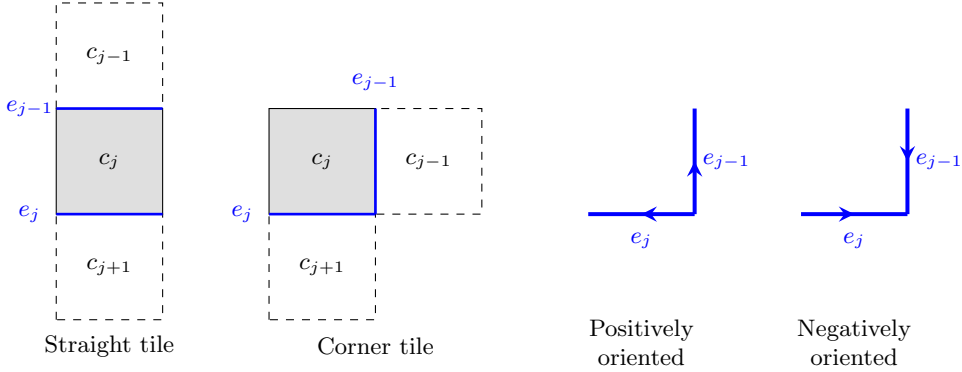
\begin{figure}[htbp]
  \centering
%
%
\tikzset{->-/.style={decoration={
  markings,
  mark=at position .5 with {\arrow{>}}},postaction={decorate}}}

\begin{tikzpicture}[>=stealth, scale=0.7]
	\filldraw[fill=gray!25!white] (0,2) rectangle (2,4);
	\draw[dashed] (0,0) rectangle (2,2);
	\draw[dashed] (0,4) rectangle (2,6);
	
	\draw[blue, line width=1] (0,2) -- (2,2);
	\draw[blue, line width=1] (0,4) -- (2,4);
	
	\draw (1,-0.5) node{\small Straight tile};
	\draw (-0.5,2) node{{\color{blue}\small $e_j$}};
	\draw (-0.5,4) node{{\color{blue}\small $e_{j-1}$}};
	\draw (1,3) node{\small $c_j$};
	\draw (1,5) node{\small $c_{j-1}$};
	\draw (1,1) node{\small $c_{j+1}$};
	
	\draw[dashed] (4,0) rectangle (6,2);
	\filldraw[fill=gray!25!white] (4,2) rectangle (6,4);
	\draw[dashed] (6,2) rectangle (8,4);
	
	\draw[blue, line width=1] (4,2) -- (6,2);
	\draw[blue, line width=1] (6,2) -- (6,4);
	
	\draw (6,-0.5) node{\small Corner tile};
	\draw (3.5,2) node{{\color{blue}\small $e_j$}};
	\draw (6,4.5) node{{\color{blue}\small $e_{j-1}$}};
	\draw (5,3) node{\small $c_j$};
	\draw (7,3) node{\small $c_{j-1}$};
	\draw (5,1) node{\small $c_{j+1}$};
	
	\draw[->-, blue, line width=1.5] (12,2) -- (10,2);
	\draw[->-, blue, line width=1.5] (12,2) -- (12,4);
	\draw (12.6,3) node{{\color{blue}\small $e_{j-1}$}};
	\draw (11,1.5) node{{\color{blue}\small $e_{j}$}};
	\draw (11,-0.2) node{\small Positively};
	\draw (11,-0.7) node{\small oriented};
	
	\draw[->-, blue, line width=1.5] (14,2) -- (16,2);
	\draw[->-, blue, line width=1.5] (16,4) -- (16,2);
	\draw (16.6,3) node{{\color{blue}\small $e_{j-1}$}};
	\draw (15,1.5) node{{\color{blue}\small $e_{j}$}};
	\draw (15,-0.2) node{\small Negatively};
	\draw (15,-0.7) node{\small oriented};
	
\end{tikzpicture}

   \caption{Sketch of the different types of tiles and orientation.}
\end{figure}
Next, we construct a family of isometries $(\iota_j:\R^2 \to \R^2)_{j \in \set{1,...,n}}$ each of which maps $[0,r]^2$ to the cell $c_j$. 
The purpose of this is to simplify the later construction of paths within each cell.
Recall that any composition of translations, rotations and reflections in the plane is an isometry.
This, along with the fact that $e_j \neq e_{j-1}$ for all $j \in \set{1,...,n-1}$, is what guarantees the existence of isometries satisfying the below conditions. 
\begin{enumerate}
 \item For $j \in \set{1,...,n-1}$, if $c_j$ is a straight tile, then choose $\iota_j$ such that $\iota_j(s,0) = e_j(s)$ and $\iota_j(s,r) = e_{j-1}(s) $.
 \item  For $j \in \set{1,...,n-1}$, if $c_j$ is a positively oriented corner tile, then choose $\iota_j$ so that $\iota_j(r - s,0) = e_j(s)$ and $\iota_j(r,s) = e_{j-1}(s)$.
 \item For $j \in \set{1,...,n-1}$, if $c_j$ is a negatively oriented corner tile, then choose $\iota_j$ so that $\iota_j(s,0) = e_j(s)$ and $\iota_j(r,r - s) = e_{j-1}(s)$.
 \item Choose $\iota_n$ so that $\iota_n(s,r) = e_{n-1}(s)$ and  $\iota_n([0,r]^2) = c_n$. In the case of a cell-path to an lg-cell, this implies that $\iota_n([0,r]\times[-r,0]) = c_{n+1}$.
\end{enumerate}

By the density of $C_c^\infty(\cO)$ in $H^1_0(\cO)$, it suffices to show that \eqref{eq:tetrimino-poin} holds for all $u \in C_c^\infty(\cO)$. Hence, let $u \in C^\infty_c(\cO)$.

By Definitions \ref{de:g-cell} and \ref{de:lg-cell} for a g-cell and an lg-cell respectively, 
there exists a function $w:[0,r] \to [-r,r]$ such that $u \circ \iota_n(s,w(s)) = 0$ for all $s \in [0,r]$.
Note that in the case of a cell-path to a g-cell, $w$ only takes values in $[0,r]$.

Firstly, for any $s \in [0,r]$,
\begin{equation*}
 u(e_{n-1}(s)) = u \circ \iota_n(s,r) = \int_{w(s)}^r \frac{\partial}{\partial t} u \circ \iota_n (s,t) \d t =: I_g(s).
\end{equation*}
Let $j \in \set{1,...,n-1}$.
If $c_j$ is a straight tile, then for any $s \in [0,r]$,
\begin{equation*}
 u(e_{j-1}(s)) - u(e_j(s)) = u \circ \iota_j(s,r) - u \circ \iota_j(s,0) = \int_0^r \frac{\partial}{\partial t}u \circ \iota_j(s,t) \d t =: I_j(s).
\end{equation*}
If $c_j$, is a corner tile, then let
\begin{equation*}
 \tilde{I}_j(s) := \int_0^s \frac{\partial}{ \partial t} u \circ \iota_j (r-s,t) \d t + \int_{r-s}^r \frac{\partial}{ \partial t} u \circ \iota_j (t,s) \d t. 
\end{equation*}
If $c_j$ is a positively oriented corner tile, then for any $s \in [0,r]$,
\begin{equation*}
 u(e_{j-1}(s)) - u(e_j(s)) = u \circ \iota_j(r,s) - u \circ \iota_j(r-s,0) = \tilde{I}_j(s) =: I_j(s).
\end{equation*}
If $c_j$ is a negatively oriented corner tile, then for any $s \in [0,r]$,
\begin{equation*}
 u(e_{j-1}(s)) - u(e_j(s)) = u \circ \iota_j(r,r - s) - u \circ \iota_j(s,0) = \tilde{I}_j(r- s) =: I_j(s).
\end{equation*}

We can now express the value of $u$ at any point in $c_0 = [0,r]^2$ as sum of line integrals. 
For any $x,y \in [0,r]$, 
\begin{equation}\label{eq:u-as-int}
 u(x,y) = u(e_0(x)) + \int_0^y \frac{\partial}{\partial t} u(x,t) \d t = I_g(x) + \sum_{j=1}^{n-1} I_j(x) + \int_0^y \frac{\partial}{\partial t} u(x,t) \d t
\end{equation}
hence 
\begin{align}\label{eq:L2u-line}
 \norm{u}^2_{L^2(c_0)} & = \int_0^r  \int_0^r  |u(x,y)|^2 \d x \d y  \nonumber\\
 & \leq (n+1)r \sbr*{ \int_0^r |I_g(x)|^2 \d x +  \sum_{j=1}^{n-1} \int_0^r |I_j(x)|^2 \d x + \int_0^r \br*{ \int_0^r \abs*{ \frac{\partial}{\partial t} u(x,t) } \d t }^2 \d x }.
\end{align}

Focusing on the final term in the square brackets of \eqref{eq:L2u-line} and applying Cauchy-Schwarz, 
\begin{equation}\label{eqpr:tetrimino-1}
 \int_0^r \br*{ \int_0^r \abs*{ \frac{\partial}{\partial t} u(x,t)} \d t }^2 \d x \leq r \int_0^r \int_0^r \abs*{ \frac{\partial}{\partial t} u(x,t)}^2 \d x \d t \leq r \norm{\nabla u}^2_{L^2(c_0)}.
\end{equation}
To estimate the remaining terms, we need to use the fact that
\begin{align*}
 \abs*{\frac{\partial}{\partial t} u \circ \iota_j (x,t)}  \leq \abs*{ \nabla u (\iota_j(x,t)) } \abs*{\frac{\partial \iota_j}{\partial t}(x,t)} \leq \abs*{\nabla u (\iota_j(x,t))},
\end{align*}
where the final inequality holds since $\iota_j$ is an isometry,  and similarly, 
\begin{align*}
 \abs*{\frac{\partial}{\partial t} u \circ \iota_j (t,y)} \leq \abs*{\nabla u (\iota_j(t,y))}.
\end{align*}
Focusing on the middle terms in the square brackets of \eqref{eq:L2u-line}, let $j \in \set{1,...,n-1}$.
If $c_j$ is a straight tile, then 
\begin{equation}\label{eqpr:tetrimino-2}
 \int_0^r |I_j(x)|^2 d x \leq \int_0^r \br*{ \int_0^r \abs*{\nabla u (\iota_j(x,t))} \d t }^2  \d x \leq r \int_0^r \int_0^r \abs*{\nabla u (\iota_j(x,t))}^2 \d x \d t = r \norm{\nabla u}^2_{L^2(c_j)}.
\end{equation}
If $c_j$ is an corner tile, then 
\begin{align*}
 \int_0^r \abs{\tilde{I}_j(x)}^2 \d x & \leq 2 \sbr*{\int_0^r \br*{ \int_0^x \abs*{ \nabla u (\iota_j (r-x,t))} \d t }^2 \d x + \int_0^r \br*{\int_{r-x}^r \abs*{ \nabla u ( \iota_j (t,x) ) } \d t}^2 \d x }\\
                   &    \leq 2 r \br*{\int_0^r  \int_0^x \abs*{ \nabla u (\iota_j (r-x,t))}^2 \d t  \d x + \int_0^r \int_{r-x}^r \abs*{ \nabla u ( \iota_j (t,x) ) }^2 \d t \d x  }             \\
                   & = 2 r \norm{\nabla u}^2_{L^2(c_j)}.
\end{align*}
Hence, if $c_j$ is a positively oriented corner tile, then 
\begin{equation}\label{eqpr:tetrimino-3}
 \int_0^r \abs{I_j(x)}^2 \d x = \int_0^r \abs{\tilde{I}_j(x)}^2 \d x \leq 2 r \norm{\nabla u}^2_{L^2(c_j)}
\end{equation}
and, similarly, if $c_j$ is a negatively oriented corner tile, then 
\begin{equation}\label{eqpr:tetrimino-4}
 \int_0^r \abs{I_j(x)}^2 \d x = \int_0^r \abs{\tilde{I}_j(r - x)}^2 \d x = \int_0^r \abs{\tilde{I}_j(x)}^2 \d x \leq 2 r \norm{\nabla u}^2_{L^2(c_j)}.
\end{equation}
Finally, letting $h = 0$ in the case of cell-path to a g-cell and $h = -r$ in the case of a cell-path to an lg-cell, we have
\begin{align}\label{eqpr:tetrimino-5}
  \int_0^r |I_g(x)|^2 \d x & \leq \int_0^r \br*{\int_{w(x)}^r \abs*{\nabla u(\iota_n(x,t))} \d t }^2 \d x \nonumber\\
  & \leq 2 \sbr*{\int_0^r \br*{\int_0^r \abs*{\nabla u(\iota_n(x,t))} \d t }^2 \d x + \int_0^r \br*{\int_h^0 \abs*{\nabla u(\iota_n(x,t))} \d t }^2 \d x }\nonumber \\
  & \leq 2 r \br*{ \norm{\nabla u}^2_{L^2(c_n)} + \norm{\nabla u}^2_{L^2(c_{n+1})}}.
\end{align}
where $c_{n+1}$ is considered to the empty set in the case of a cell-path to a g-cell.
The proof for the case $n \geq 1$ is completed by substituting estimates \eqref{eqpr:tetrimino-1}-\eqref{eqpr:tetrimino-5} into \eqref{eq:L2u-line}.

The case $n = 0$ is similar. Assume that $c_0 = [0,r]^2$ and in the lg-cell case, that $c_1 = [0,r] \times [-r,0]$. Then there exists a function $w:[0,r] \to [-r,r]$ such that
\begin{equation*}
  u(x,y) = \int_{w(x)}^y \frac{\partial }{\partial t} u(x,t) \d t \qquad ((x,y) \in c_0)
\end{equation*}
and the proof proceeds as before.
\end{proof}

\subsection{Construction of the cell-paths}\label{subsec:constr-cell-paths}
Next, we need to construct cell paths from any covering cell to a g-cell or an lg-cell.
The first step is to show that there is a g-cell or an lg-cell in the 1-cell neighbourhood of a filled cell,
provided the path-connected components of $\partial \cO$ all have large enough diameter.
We shall need the fact that
 \begin{equation}\label{eq:diam-defns}
  \diam(A) \leq 2 \inf_{x \in A } \sup_{y \in A} |x - y|.
 \end{equation}
for any bounded set $A \subset \R^d$.

\begin{figure}[h!]
	\centering
	\includegraphics[width=0.9\textwidth]{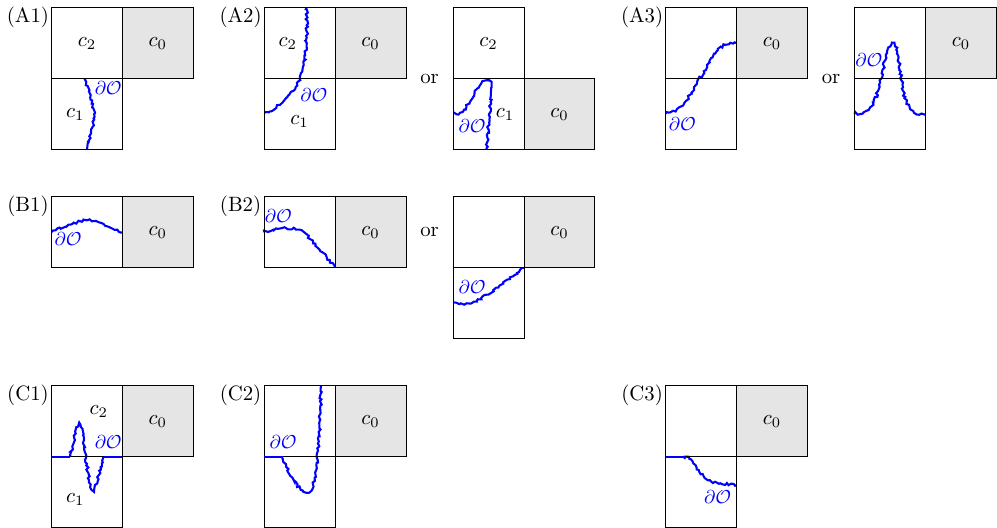}
	\caption{Example sketches for some of the cases (A), (B), (C) in the proof of Lemma \ref{lemma:exists_g-cell}}
\end{figure}

\begin{lemma}\label{lemma:exists_g-cell}
 If $Q(\partial \cO) > 4 \sqrt{2} r$, then for any filled cell $c_0$ there exists a g-cell or an lg-cell contained in $D_1[c_0]$.  
\end{lemma}

\begin{proof}
 
  Let $c_0$ be a filled cell.   
 There  exists a path-connected component $\Gamma \subseteq \partial \cO$ such that $\Gamma \cap c_0 \neq \emptyset$. Let $x \in \Gamma \cap c_0$. 
 Using (\ref{eq:diam-defns}) and the hypothesis on $Q(\partial \cO)$,
 \begin{equation*}
 2\sqrt{2} r < \frac{1}{2} Q(\partial \cO) \leq \frac{1}{2} \diam (\Gamma) \leq \sup_{y \in \Gamma}|x - y|
 \end{equation*}
 so there exists $y \in \Gamma$ with $|x - y| > 2 \sqrt{2} r $.
 In particular, $y\in \Gamma$ lies outside of $D_1 [c_0]$. 
 Since $\Gamma$ is path-connected, there exists a continuous path in $\Gamma$ from $y$ to $x$.
 Restricting this path, we deduce there exist a continuous path $\gamma:[0,1] \to D_1[c_0]$ such that 
 \begin{equation*}
  \gamma(t) \in \begin{cases} \partial D_1[c_0] & \text{if }t = 0 \\ \intt  D_1[c_0] \bs c_0 & \text{if } t \in (0,1) \\ \partial c_0 & \text{if }t = 1 \end{cases}\quad \text{and} \quad \forall t \in [0,1]:\gamma(t) \in \Gamma .
 \end{equation*}

 Let us fix some notions that will allow us to prove the lemma. 
 Firstly, an edge $e$ is a \textit{zeroth edge} if $\gamma(0)\in e$ and $e \subset \partial D_1[c_0]$.
 Since we defined an edge to be closed, there may be up to two zeroth edges.
 
 Let 
 \begin{equation*}
  t_1 := \inf \set*{ t > 0: \exists \text{ edge }e\text{ such that }\gamma(t) \in e } \in [0,1].
 \end{equation*}
 A \textit{first edge} is defined as any edge $e$ such that $\gamma(t_1) \in e$ and $e$ is not a zeroth edge.
 If $t_1 \in (0,1)$, then the first edge is unique since $\gamma(t)$ can belong to at most one edge for $t \in (0,1)$.
 If $t_1 = 1$, then there may be up to four first edges (indeed, this is the case if $\gamma(1)$ lies in a corner of $c_0$).
 If $t_1 = 0$, then the first edge is again unique. This is because $\gamma(0)$ must lie in $\partial D_1[c_0]$
 and hence can only lie in at most one edge which isn't entirely contained in $\partial D_1[c_0]$
 (indeed, an edge containing $\gamma(0)$ which is contained in $\partial D_1[c_0]$ must be a zeroth edge).

If $t_1 < 1$, then there exists a unique first edge $e_1$ so we can make the following definitions.
Let
 \begin{equation*}
  t_2 := \inf \set*{ t > 0: \exists \text{ edge }e\text{ such that }\gamma(t) \in e\text{ and }e \neq e_1} \in (t_1,1].
\end{equation*}
Here, $t_2$ exists and satisfies $t_2 \leq 1$ since $\gamma(1)$ lies in at least one edge which is contained in $\partial c_0$ hence which is not the first edge $e_1$.
Also, $t_2$ satisfies $t_2 > t_1$ since the only edge that $\gamma(t)$ can intersect for $t \in (0, t_1]$ is the first edge $e_1$.
A \textit{second edge} is defined as any edge $e$ such that $\gamma(t_2) \in e$ and $e\neq e_1$.
Note that a second edge cannot be a zeroth edge since $t_1 > 0$.
Finally, let
\begin{equation*}
 \tilde{t}_1 := \sup \set*{ t \leq t_2 : \gamma(t) \in e_1}.
\end{equation*}
If $t_1=1$, then $t_2$, the second edges and $\tilde{t}_1$ are not defined.

 Let us now proceed onto the main part of the proof, in which we repeatedly use the continuity of the path $\gamma$.
 
\begin{enumerate}[label = (\Alph*), wide,labelindent=0pt]
\item Suppose that $t_1 \in (0,1)$. 
  Then, there exists a unique first edge $e_1$ and a unique cell $c_1$ containing $\gamma([0, t_1])$
  (indeed, note that $\gamma((0,t_1))$ must be contained in the interior of a cell).
  $c_1$ must contain $e_1$ - let $c_2$ be the other cell containing $e_1$.
\begin{itemize}[itemindent = 20pt]
    \item[(A1)] If there exists a zeroth edge contained in $c_1$ which is parallel to $e_1$, then $c_1$ is a g-cell.
    \item[(A2)] Suppose there exists a second edge $e$ which is contained in a cell $c \in \set{c_1,c_2}$ and which is parallel to $e_1$.
      By the definition of a second edge, $e \neq e_1$.
      $\gamma([\tilde{t}_1,t_2])$ is contained in $c$ and connects the distinct parallel edges $e$ and $e_1$ of $c$, therefore, $c$ is a g-cell. 
    \item[(A3)] In the only other case, there exists a zeroth edge $e_0$ contained in the cell $c_1$ and a second edge $e_2$ contained in a cell $c \in \set{c_1,c_2}$ such that both $e_0$ and $e_2$ are perpendicular to the edge $e_1$. 
    It follows that $e_0$ and $e_2$ are distinct, parallel edges. 
    Furthermore, the edges $e_0$ and $e_2$ are contained in distinct long edges of the long-cell $\set{c_1,c_2}$, hence the long edges of $\set{c_1,c_2}$ are connected by $\gamma([0,t_2])$.
    Since $\gamma([0,t_2])$ is contained in  $c_1 \cup c_2$, $\set{c_1,c_2}$ is an lg-cell.
\end{itemize}
\noindent
We conclude that if $t_1 \in (0,1)$, then there is a g-cell or an lg-cell contained in $D_1[c_0]$.
 \item 
 Suppose that $t_1 = 1$. Then, $\gamma([0,1])$ is contained entirely in one cell since $\gamma(t)$ does not lie in any edge for every $t \in (0,1)$.
 \begin{itemize}[itemindent = 20pt]
 \item[(B1)]  Suppose $\gamma(1)$ is in the interior of an edge $e$ belonging to $c_0$.
   Then the unique cell $c \neq c_0$  containing $e$ also contains a zeroth edge parallel to $e$, as well as $\gamma([0,1])$ in its entirety.
   In this case, $c$ is a g-cell.
    \item[(B2)]  In the only other case, $\gamma(1)$ is not in the interior an edge so $\gamma(1)$ is a corner of $c_0$. 
      Then, there are four first edges, each of which is parallel and sharing a cell with exactly one of the four possible zeroth edges.
      Consequently, in this case the cell containing $\gamma([0,1])$ in its entirety is a g-cell. 
 \end{itemize}
\noindent
We conclude that if $t_1 = 1$, then there is a g-cell contained in $D_1[c_0]$. 
 
\item

Suppose that $t_1 = 0$.
In this case, there exists a unique first edge $e_1$. 
 \begin{itemize}[itemindent = 20pt]
  \item[(C1)] Suppose that $\tilde{t}_1 = t_2$. 
  Let $c_1$ and $c_2$ be the cells sharing the edge $e_1$. 
  In this case $\gamma(0)\in \partial D_1[c_0]$ and $\gamma(t_2) \in \partial c_0$ belong to opposite extremal points of the edge $e_1$ hence belong to distinct long-edges of the long cell $\set{c_1,c_2}$.
  Furthermore, since the only edge that $\gamma((0,t_2))$ can intersect is  $e_1$, $\gamma([0,t_2]) \subset c_1 \cup c_2$ hence $\set{c_1,c_2}$ forms an lg-cell.
 \end{itemize}
Suppose, on the other hand, that $\tilde{t}_1 < t_2$. 
Then, since $\gamma(t)$ does not lie in any edge for $t \in (\tilde{t}_1,t_2)$, there exists a unique cell $c_1$ containing $\gamma([\tilde{t}_1,t_2])$. 
$c_1$ must contain the edge $e_1$ - let $c_2$ denote the other cell containing the edge $e_1$. 
 $c_1$ must contain at least one second edge so we have the following possibilities. 

 \begin{itemize}[itemindent = 20pt]
  \item[(C2)] If there exists a second edge $e$ contained in $c_1$ which is parallel to $e_1$, then $c_1$ is a g-cell since $\gamma([\tilde{t}_1,t_2])$ connects $e_1$ and $e$. 
  \item[(C3)] In the only other possibility, there exists a second edge contained in $c_1$ which is perpendicular to $e_1$. 
  In this case, $\gamma(0)$ and $\gamma(t_2)$ are contained in distinct long-edges of the long-cell $\set{c_1,c_2}$, hence $\set{c_1,c_2}$ forms an lg-cell.
 \end{itemize}
\noindent
 We conclude that if $t_1 = 0$, then there exists a g-cell or an lg-cell contained in $D_1[c_0]$.
\end{enumerate}
We have covered every possible case, proving the lemma.
\end{proof}

  \begin{figure}[h!]
  	\centering
  	\includegraphics[width=0.9\textwidth]{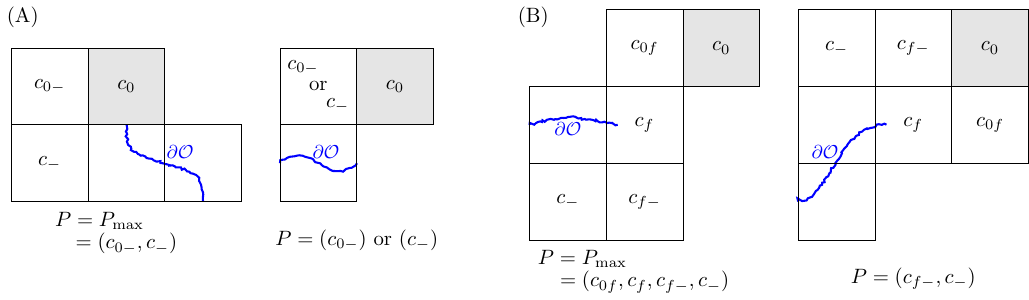}
  	\caption{Examples of cell-paths for different cases in the proof of Lemma \ref{lem:loc-cell-path}.}
      \end{figure}
 
      Next we construct a cell-path for each covering cell, using the above lemma as well as the fact that there is a filled cell in the
      1-cell neighbourhood of any covering cell. 
\begin{lemma}\label{lem:loc-cell-path}
  For any covering cell $c_0$, there exists a g-cell $c_n$ (or an lg-cell $\set{c_n, c_{n+1}}$)  contained in $D_2[c_0]$  and a cell-path $(c_1,...,c_{n-1})$ from $c_0$ to $c_n$
  (or to $\set{c_n,c_{n+1}}$ resp.) such that  $c_j$ is a covering cell contained in $D_2[c_0]$ for all $j \in \set{1,...,n-1}$.
  Furthermore, $n \leq 5$.
\end{lemma}

\begin{proof}
  Let $c_0$ be a covering cell. We aim to construct a g-cell $c_g $ (or an lg-cell $\clg$) and a cell path $P$ from $c_0$ to $c_g$
  (or to $\clg$ resp). If $c_0$ is a g-cell (or in an lg-cell), then we define $P$ to be empty.
  We consider the two remaining possible cases.
  \begin{enumerate}
  \item[(A)] Suppose first that there exists a g-cell $c_g \subset \red{D_1[c_0] \bs \intt(c_0)}$ (or an lg-cell $\clg$ with $c_{lg}^{(j)} \subset \red{D_1[c_0] \bs \intt(c_0)}$ for some $j \in \set{1,2}$).
    Let $\tilde{c}_g := \intt(c_g)$ (or $\tilde{c}_g := \intt(c_{lg}^{(1)} \cup c_{lg}^{(2)})$ resp.). There exists a cell $c_- \subset D_1[c_0] \bs \tilde{c}_g$ which shares a normal edge
    with $c_g$ (or with $\clg$ resp.).  There exists a cell $c_{0-} \subset D_1[c_0] \bs \tilde{c}_g$ which shares an edge with both  $c_0$ and $c_-$.

    \red{Consider the sequence of cells 
    \begin{equation*}
      P_{\max} := (c_{0-},c_-).
    \end{equation*}
    Recall Definition \ref{de:cell-path} for a cell path. 
    $P_{\max}$ satisfies hypothesis (1) of this definition since $c_{0-}$ shares a common edge with $c_0$ and $c_-$ shares an edge with $c_{0-}$.
    Furthermore, $P_{\max}$ also satisfies hypothesis (2) since there exists an edge of $c_-$ which is also a normal edge of $c_g$ (or of $\clg$ resp.).
    However, $P_{\max}$ may not form a cell path since it may contain $c_0$ or contain repeated elements, hence, hypothesis (3) may not be satisfied.
  }
  
    Define $P$ as any length-minimising subsequence of $P_{\max}$
    such that hypotheses (1) and (2) of Definition \ref{de:cell-path} are still satisfied.
    Note that $P$ always exists since $P_{\max}$ satisfies these hypotheses.
    $P$ cannot contain $c_0$ or repeated elements since this would yield a shorter such subsequence.
    $P$ does not contain $c_g$ (or an element of $\clg$ resp.) by definition.
    Consequently, $P$ satisfies hypothesis (3) of Definition \ref{de:cell-path} and is a cell path.
    $c_{0-}$ and $c_0$ are contained in $D_1[c_0]$ and share an edge or corner with a g-cell (or an lg-cell respectively)
    so the elements of $P$ are covering cells contained in $D_1[c_0] \subset D_2[c_0]$. 
    
  \item[(B)] In the other case, there does not exist a g-cell, or an element of an lg-cell, contained in $D_1[c_0]$.
    Since $c_0$ is a covering cell, there exists a filled cell $c_f \subset D_1[c_0]$. There exists a cell $c_{0f} \subset D_1[c_0]$ sharing an edge with both $c_0$ and $c_f$.
    By Lemma \ref{lemma:exists_g-cell}, there exists a g-cell $c_g \subset D_1[c_f]$ (or an lg-cell $\clg$ contained in $D_1[c_f]$).
    Note that, by assumption, $c_f$ and $c_{0f}$ are distinct from $c_g$ (or from both elements of $\clg$ resp.).
    Let $\tilde{c}_g := \intt(c_g)$ (or $\tilde{c}_g := \intt(c_{lg}^{(1)} \cup c_{lg}^{(2)})$ resp.). There exists a cell $c_- \subset D_1[c_f] \bs \tilde{c}_g$ which shares a normal edge
    with $c_g$ (or with $\clg$ resp.).  There exists a cell $c_{f-} \subset D_1[c_f] \bs \tilde{c}_g$ which shares an edge with both  $c_f$ and $c_-$.

    \red{Following a similar strategy as in (A),} define $P$ as any length-minimising subsequence of 
    \begin{equation*}
      P_{\max} := (c_{0f},c_f, c_{f-},c_-)
    \end{equation*}
    such that hypotheses (1) and (2) of Definition \ref{de:cell-path} are satisfied.
    By a similar reasoning as in (A), $P$ is a cell-path and its elements are covering cells contained in $D_2[c_0]$.
  \end{enumerate}

Finally, the fact that $n \leq 5$ follows from the fact that
the length of the sequence $P_{\mathrm{max}}$ is bounded by 4 in each case.
\end{proof}

Finally, we utilise the cell-paths that we have constructed, combined with the Poincar\'e-type inequality
for the cell-paths, to prove the Poincar\'e-type inequality on $\col{r}\cO$.

\subsection{Proof of Theorem \ref{thm:poin}}\label{sec:proof-poin}

Let $\set{c_j}$ be the set of covering cells. Then, $\col{r} \cO \subseteq \bigcup_j c_j$.

For each $c_j$ which is not a g-cell or in an lg-cell, \red{Lemma \ref{lem:loc-cell-path} guarantees the existence of:}
\begin{Red}
  \begin{itemize}
    \item an integer $1 \leq n_j \leq 5$,
    \item an associated g-cell $\red{c_{j,n_j}}$ (or an associated lg-cell $\red{\set{c_{j,n_j},c_{j,n_j+1}}}$) and 
    \item  (if $n_j \geq 2$) an associated cell-path $\red{(c_{j,1},...,c_{j,n_j - 1})}$ from $c_j$ to $\red{c_{j,n_j}}$
      (or to $\red{\set{c_{j,n_j},c_{j,n_j+1}}}$ resp.).
  \end{itemize}
\end{Red}
 If $c_j$ is a g-cell itself then there are no such associated cells and if $c_j$ is in an lg-cell, let $\red{c_{j,1}}$ be the other cell in the lg-cell.
 For each $j$, in the case of an associated g-cell, let $N_j := n_j$ and in the case of an associated lg-cell let $N_j := n_j +1$.
 Additionally, in the case that $c_j$ is a g-cell, let $N_j = 0$ and in the case that $c_j$ is in an lg-cell, let $N_j = 1$.
 \red{Assume that $c_{j,k}$ satisfy the properties prescribed by Lemma \ref{lem:loc-cell-path}, that is, $\red{c_{j,k}}$ is a covering cell contained in $D_2[c_j]$ for each $j$
   and each $k \in \set{1,...,N_j}$, and $(c_j,\red{c_{j,1},...,c_{j,N_j}})$ consists of distinct elements.}
 
 Applying Lemma \ref{lem:tetriminos}, we have, 
\begin{equation}\label{eqpr:poin}
 \norm{u}^2_{L^2(\col{r} \cO)} \leq \sum_j \norm{u}^2_{L^2(c_j)} \leq 12 r^2 \sum_j \br*{ \norm{\nabla u}^2_{L^2(c_j)} + \sum_{k=1}^{N_j}  \norm{\nabla u}^2_{L^2(\red{c_{j,k}})}}
\end{equation}
where the sum over $k$ is empty in the case $N_j = 0$.
Since the associates to a given covering cell are in its 2-cell neighbourhood,
each covering cell can be an associate to at most 24 other covering cells (indeed, there are 25 cells in a 2-cell neighbourhood).
Furthermore, the associates to a given covering cell are distinct and each associate $c_{j,k}$ is a covering cell so it follows from \eqref{eqpr:poin} that 
\begin{equation*}
 \norm{u}^2_{L^2(\col{r} \cO)} \leq 12 \times 25 r^2 \sum_j \norm{\nabla u}^2_{L^2(c_j)} \leq 300 r^2 \norm{\nabla u}^2_{L^2(\col{2 \sqrt{2} r} \cO)}
\end{equation*}
where the last inequality holds since $\intt (\cup_j c_j ) \subseteq \col{2\sqrt{2} r} \cO$.

\section{Mosco convergence}\label{sec:mosco}

In this section, we establish a general Mosco convergence theorem and apply it to pixelated domain approximations. 
We will make use of the notion of an \red{open $\epsilon$-neighbourhood} $\dil_\epsilon(A)$ of a set $A \subset \R^d$ - recall that this is defined by equation (\ref{eq:dil-defn}).

\subsection{From uniform Poincar\'e-type inequalities to Mosco convergence}\label{sec:from-unif-poinc}

The first step is to prove Mosco convergence for sequences of domains $(\cO_n)$ which satisfy a Hausdorff convergence condition
to a limit domain $\cO$ and which verify a certain Poincar\'e-type inequality uniformly for the whole sequence.
Such a uniform Poincar\'e inequality does not follow immediately from the results of the previous section,
but will be established in Section \ref{sec:gener-mosco-conv} under suitable hypotheses. 

The following fact shall be useful.

\begin{lemma}\label{lem:stabwound}
 For any non-empty, bounded sets $A,B \subset \R^d$, we have 
 \begin{equation*}
  \sup_{x \in A^c \cap B} \dist (x, \partial B) \leq \d_H(A, B) + \d_H(\partial A, \partial B).
 \end{equation*}
\end{lemma}

\begin{proof}
 This  holds because
 \begin{align*}
  \sup_{x \in A^c \cap B} \dist (x, \partial B) & \leq \sup_{x \in A^c \cap B} \dist (x, \partial A) + \d_H(\partial A, \partial B) \\
  & = \sup_{x \in A^c \cap B} \dist (x, A) + \d_H(\partial A, \partial B) \\
  & \leq  \d_H(A, B) + \d_H(\partial A, \partial B).
 \end{align*}
\end{proof}

The proof of the next proposition uses a construction of certain cut-off functions to
directly prove that the two conditions in Definition \ref{defn:mosco-convergence} for Mosco convergence hold.  
Note that the regularity of the limit domain $\cO$ is not yet required.
We do require, however, that the Lebesgue measure of the boundary $\partial \cO$ vanishes, which ensures that the Lebesgue measure of the collar neighbourhood
$\col{\epsilon}\cO$ tends to $0$ as $\epsilon \to 0$. 

\begin{prop}\label{prop:mosco}
Let  $\cO \subset \R^d$ and $\cO_n \subset \R^d,\,n \in \N$, be bounded, open sets such that the following holds: 
\begin{enumerate}[label= \rm(\alph*)]
  \item $l(n) := \d_H(\cO, \cO_n) + \d_H(\partial \cO, \partial \cO_n) \to 0$ as $n \to \infty$.
  \item There exist
  \begin{enumerate}[label = \rm(\roman*)]
   \item $(f(n))_{n \in \N}$ such that $2 l(n) \leq f(n)$ for all $n \in \N$ and $f(n) \to 0$ as $n \to \infty$,
   \item  constants $C, \alpha > 0$ independent of $n$ and $u$,
  \end{enumerate}
    such that, if either $V = \cO$ or $V=\cO_n$ for some large enough $n \in \N$, then for all $u \in H^1_0(V)$ we have
  \begin{equation}\label{eq:poin-for-Mosco}
   \norm{u}_{L^2(\col{f(n)} V)} \leq C f(n) \norm{\nabla u}_{L^2(\col{\alpha f(n)} V)}.
 \end{equation}
 \item $\vol(\partial \cO) = 0$. 
 \end{enumerate}
Then, $\cO_n$ converges to $\cO$ in the Mosco sense as $n \to \infty$.
 
\end{prop}

\begin{proof}
 
 Throughout the proof, let $L^p$ denote $L^p(\R^d)$ for $p = 2, \infty$ and let $H^1$ denote $H^1(\R^d)$.
 All limits will be as $n \to \infty$.
 
 Define function $\tilde{\chi}:\R_+ \to [0,1]$ by
 \begin{equation}
    \tilde{\chi}(t) := \begin{cases} 
    t & \text{if } t\in [0,1) \\
                       1 & \text{if }t \in [1,\infty).
                       \end{cases}
\end{equation}
$\tilde{\chi}$ is weakly differentiable with $\norm{\tilde{\chi}'}_{L^\infty} = 1$.
$\tilde{\chi}$ will be used in the construction of a cut-off function $\chi_n$ in both Step 1 and Step 2 below.
We shall also require the following two facts. 
Firstly, for any $A \subset \R^d$ with piecewise smooth boundary, the function $x \mapsto \dist(x,A)$ is continuous and
piecewise smooth hence weakly differentiable.
Furthermore, since
\begin{equation*}
 \abs*{\dist(x,A) - \dist(y, A)} \leq |x - y|\qquad (x,y \in \R^d),
\end{equation*}
the $L^\infty$  norm of $x \mapsto \nabla_x \dist(x, A) $ is bounded by 1. 

\noindent
 \textit{Step 1 (Mosco convergence condition (1)).}
 Let $u_n \in H^1_0(\cO_n)$, $n \in \N$, and suppose that $u_n \wto u$ in $H^1$ for some $u \in H^1$. 
We aim to show that $u \in H^1_0(\cO)$. 

Let $P: H^1 \to H^1_0(\cO)$ be the orthogonal projection.
If $(w_n) \subset H^1_0(\cO)$ and $w_n \wto u$ in $H^1$ then
\begin{equation*}
    \inner{u,(1 - P)\phi}_{H^1} = \lim_{n \to \infty} \inner {w_n, (1 - P) \phi}_{H^1} = 0 \qquad (\phi \in H^1) 
\end{equation*}
so $u \in H^1_0(\cO)$.
Hence it suffices to show that there exists $w_n \in H^1_0(\cO)$ such that $w_n \wto u$ in $H^1$. 

Assume without loss of generality that $f(n) > 0$ for all $n \in \N$.
For each $n \in \N$, let $A_n \subset \R^d$ be an open neighbourhood of $\cO^c$ with piecewise smooth boundary such that $A_n \cap \cO \subseteq \col{f(n)/4} \cO$ ($A_n$ can be constructed by an open cover of balls of radius $f(n)/4$ for instance).

Define a cut-off function $\chi_n:\R^d \to [0,1]$ by 
\begin{equation}
 \chi_n (x) = \tilde{\chi}(4f(n)^{-1} \dist(x, A_n)) \qquad (x \in \R^d).
\end{equation}
Then, $\chi_n = 0$ on an open neighbourhood of $\cO^c$ and $\chi_n(x) = 1$ for any $x \in \cO_n$ outside the set
\begin{equation*}
    \cU_n := \set*{x \in \cO_n: \dist(x, \cO^c) \leq f(n)/2}.
  \end{equation*}
By the piecewise smoothness of $\partial A_n$, $\chi_n$ is weakly differentiable and, by an application of the chain rule,
\begin{equation}\label{eq:grad-chi-n}
    \norm{\nabla \chi_n}_{L^\infty} \leq 4 f(n)^{-1}. 
  \end{equation}
  
By Lemma \ref{lem:stabwound}, we have
\begin{equation}\label{eq:mosco-prop-1}
 \sup_{x \in \cO^c \cap \cO_n} \dist(x, \partial \cO_n) \leq l(n) \leq \frac{f(n)}{2}.
\end{equation}
We claim that $\cU_n \subset \col{f(n)} \cO_n$. To see this, let $x \in \cU_n$. Then there exists $y \in \cO^c$ with $|x - y| \leq f(n)/ 2$.
If $y \in \cO_n$, then inequality (\ref{eq:mosco-prop-1}) implies that $\dist(y, \partial \cO_n) \leq f(n)/2$ so $\dist(x, \partial \cO_n) \leq f(n)$.
If $y \notin \cO_n$ on the other hand, then, since $x \in \cO_n$, $\dist(x,\partial \cO_n) \leq |x - y| \leq f(n)/2$ proving the claim. 

Furthermore, for any $x \in \cU_n$, we have
\begin{equation*} 
  \dist(x, \partial \cO) \leq l(n) + \dist(x, \partial \cO_n) \leq l(n) + f(n)
\end{equation*}
so by hypothesis (c) and continuity of measures from above,
\begin{equation*}
\vol(\cU_n) \leq \vol(\dil_{l(n) + f(n)}(\partial \cO)) \to \vol(\partial \cO) = 0.
\end{equation*}

Let $w_n := \chi_n u_n$. Then, since $\chi_n = 0$ on an open neighbourhood of $\cO^c$, we have $w_n \in H^1_0(\cO)$ and it suffices to show that $w_n = \chi_n u_n \wto u $ in $H^1$.
Let $\phi \in H^1$ be an arbitrary test function. Firstly, we have,
\begin{equation*}
 \abs*{\inner{\chi_n u_n - u, \phi}_{H^1}} \leq  \underbrace{\abs*{\inner{\chi_n u_n - u, \phi}_{L^2}}}_{(A1)} + \underbrace{\abs*{\inner{\nabla(\chi_n u_n) - \nabla u, \nabla \phi}_{L^2}}}_{(A2)}.
\end{equation*}
Focusing on the term (A1), 
\begin{align*}
 \abs*{\inner{\chi_n u_n - u, \phi}_{L^2}} & \leq  \abs*{\inner{\chi_n u_n - u_n, \phi}_{L^2}} +  \abs*{\inner{u_n - u, \phi}_{L^2}} \\
 & \leq \norm{ u_n}_{L^2(\cU_n)} \norm{\phi}_{L^2(\cU_n)} + \abs*{\inner{u_n - u, \phi}_{L^2}} \to 0.
\end{align*}
Here, the second inequality holds since $u_n = 0$ almost everywhere outside $\cO_n$ and so $\chi_n u_n = u_n$ almost everywhere outside $\cU_n$.
The limit holds by the weak convergence of $(u_n)$ (so also $(u_n)$ is bounded in $H^1$) as well as the fact that $\vol(\cU_n) \to 0$.

Focusing on the term (A2), 
\begin{equation*}
 \abs*{\inner{\nabla(\chi_n u_n) - \nabla u, \nabla \phi}_{L^2}} \leq   \underbrace{\abs*{\inner{\nabla(\chi_n u_n) - \nabla u_n, \nabla \phi}_{L^2}}}_{(B1)} +   \underbrace{\abs*{\inner{\nabla u_n - \nabla u, \nabla \phi}_{L^2}}}_{(B2)}.
\end{equation*}
The term (B2) tends to zero by the weak convergence of $(u_n)$.
Focusing on the term (B1),
\begin{align*}
 \abs*{\inner{\nabla(\chi_n u_n) - \nabla u_n, \nabla \phi}_{L^2}} & \leq   \abs*{\inner{\chi_n \nabla u_n - \nabla u_n, \nabla \phi}_{L^2}} + \abs*{\inner{\nabla(\chi_n) u_n, \nabla \phi}_{L^2}}\\
& \leq  \underbrace{ \norm{\nabla u_n}_{L^2(\cU_n)} \norm{\nabla \phi}_{L^2(\cU_n)}}_{(C1)} + \underbrace{\norm{\nabla \chi_n}_{L^\infty} \norm{u_n}_{L^2(\cU_n)} \norm{\nabla \phi}_{L^2(\cU_n)} }_{(C2)}
\end{align*}
where in the second inequality we used the fact that $\chi_n \nabla u_n = \nabla u_n$ almost everywhere outside $\cU_n$ and the fact that $\supp(\nabla(\chi_n)) \cap \cO_n \subseteq \cU_n$.
The term (C1) tends to zero since $(u_n)$ is bounded in $H^1$ and $\vol(\cU_n) \to 0$. 
Focusing on the term (C2), notice first that, by the assumed Poincar\'e-type inequality \eqref{eq:poin-for-Mosco},
\begin{equation*}
 \norm{u_n}_{L^2(\cU_n)} \leq \norm{u_n}_{L^2(\col{f(n)} \cO_n)} \leq  C f(n) \norm{\nabla u_n}_{L^2(\partial \cO_n^{ \alpha f(n)})}, 
\end{equation*}
 and so, using \eqref{eq:grad-chi-n}, 
 \begin{equation*}
  \norm{\nabla \chi_n}_{L^\infty} \norm{u_n}_{L^2(\cU_n)} \norm{\nabla \phi}_{L^2(\cU_n)} \leq 4 C \norm{\nabla u_n}_{L^2(\col{\alpha f(n)} \cO_n)} \norm{\nabla \phi}_{L^2(\cU_n)} \to 0.
 \end{equation*}
It follows that the term (A2) tends to zero, that $w_n \wto u$ in $H^1$ and hence that $u \in H^1_0(\cO)$.

\noindent
\textit{Step 2 (Mosco convergence condition (2)).} 
Let $u \in H^1_0(\cO)$ - we aim to show that there exists $u_n \in H^1_0(\cO_n)$ such that $u_n \to u$ in $H^1$. Note that in this part of the proof, we shall redefine $A_n$, $\chi_n$ and $\cU_n$. 

For each $n \in \N$, let $A_n \subset \R^d$ be an open neighbourhood of $\cO_n^c$ with piecewise smooth boundary such that $A_n \cap \cO_n \subseteq \col{f(n)/4} \cO_n$. 
Define a cut-off function $\chi_n:\R^d \to [0,1]$ by 
\begin{equation}
 \chi_n (x) = \tilde{\chi}(4 f(n)^{-1} \dist(x, A_n)) \qquad (x \in \R^d).
\end{equation}
Then, $\chi_n = 0$ on an open neighbourhood of $\cO_n^c$ and $\chi_n(x) = 1$ for any $x \in \cO$ outside the set
\begin{equation*}
    \cU_n := \set*{x \in \cO: \dist(x, \cO_n^c) \leq f(n)/2}.
\end{equation*}
$\chi_n$ is weakly differentiable and, by an application of the chain rule,
\begin{equation}\label{eq:grad-chi-n-2}
    \norm{\nabla \chi_n}_{L^\infty} \leq 4 f(n)^{-1}. 
\end{equation}
By Lemma \ref{lem:stabwound}, 
\begin{equation*}
 \sup_{x \in \cO_n^c \cap \cO} \dist(x, \partial \cO) \leq l(n) \leq \frac{f(n)}{2}
\end{equation*}
so, by a similar reasoning as in Step 1, we have  $\cU_n \subseteq \col{f(n)} \cO$.
By hypothesis (c) and continuity of measures from above, $\vol(\cU_n) \to 0$.

Let $u_n := \chi_n u$. Then $u_n\in H^1_0(\cO_n)$ since $\chi_n$ vanishes on an open neighbourhood of $\cO_n^c$.
Firstly, 
\begin{equation*}
 \norm{u_n - u}_{H^1} \leq \underbrace{\norm{\chi_n u - u}_{L^2}}_{(D1)} + \underbrace{\norm{\nabla(\chi_n u) - \nabla u}_{L^2}}_{(D2)} . 
\end{equation*}
Focusing on the term (D1) and using the fact that $\chi_n u = u$ almost everywhere outside $\cU_n$, 
\begin{equation*}
 \norm{\chi_n u - u}_{L^2}  = \norm{\chi_n u - u}_{L^2(\cU_n)} \leq  \norm{u}_{L^2(\cU_n)} \to 0.
\end{equation*}
Focusing on the term (D2), we have
\begin{equation*}
 \norm{\nabla(\chi_n u) - \nabla u}_{L^2} \leq \underbrace{\norm{\chi_n \nabla u - \nabla u}_{L^2}}_{(E1)} + \underbrace{\norm{\nabla(\chi_n) u }_{L^2}}_{(E2)}. 
\end{equation*}
The term (E1) tends to zero by the same reasoning that was applied to (D1). 
Focusing on the term (E2), notice first that, by the assumed Poincar\'e-type inequality \eqref{eq:poin-for-Mosco},
\begin{equation*}
 \norm{u}_{L^2(\cU_n)} \leq \norm{u}_{L^2(\col{f(n)} \cO)} \leq C f(n) \norm{\nabla u}_{L^2(\col{\alpha f(n)} \cO)}, 
\end{equation*}
 and so, by \eqref{eq:grad-chi-n-2}, 
\begin{equation*}
\norm{\nabla(\chi_n) u }_{L^2} = \norm{\nabla(\chi_n) u }_{L^2(\cU_n)} \leq 4 f(n)^{-1} \norm{u}_{L^2(\cU_n)} \leq 4 C \norm{\nabla u}_{L^2(\col{\alpha f(n)} \cO)} \to 0.
\end{equation*}
It follows that the term (D2) tends to zero hence $u_n \to u $ strongly in $H^1$ as required.
\end{proof}

\subsection{Characterisation of $\partial \cO_n$ for large $n$}\label{sec:gener-mosco-conv}
In order to verify the uniform Poincar\'e-type inequality needed to apply Proposition \ref{prop:mosco},
we shall require additional hypotheses, such as regularity of the limit domain $\cO$ and $\#_c \intt(\cO^c) =  \#_c(\cO^c) < \infty$.
With these hypotheses at hand, we shall provide in Proposition \ref{prop:large+lonely} a characterisation of some geometric properties
of the boundaries of sequences of domains $\cO_n$, for large $n$.
Roughly speaking, we shall prove that for each connected component $\partial D_j$ of  $\partial \cO$,
there exists a ``large'' path-connected subset $\gamma_n^{(j)}$ of $\partial \cO_n$ such that $\gamma_n^{(j)}$ has comparable diameter to $\partial D_j$
and every other point in $\partial \cO_n$ is close to one of the large subsets $\gamma_n^{(j)}$.
Then, in Section \ref{sec:mosco-proof}, this characterisation is used in conjunction with the explicit Poincar\'e-type inequality
of Theorem \ref{thm:poin} to obtain the general Mosco result Theorem \ref{th:mosco} .

Let us collect some geometric and topological lemmas in preparation for the proof of Proposition \ref{prop:large+lonely}. 
Firstly, we shall require the following basic fact \red{(recall the definition of a regular domain (\ref{eq:regular-defn}))}:
\begin{equation*}
\text{An open set } A \subset \R^d\text{ is regular if and only if }A^c \subset \R^d\text{ is the closure of an open set.}
\end{equation*}
Next, let us solidify a notion of an outer boundary component for a domain.
In particular, this notion shall be crucial in defining boundary subsets $\gamma_n^{(j)}$.
\begin{de}\label{de:outer}
  The \textit{outer boundary}  $\out A$ of a bounded, connected set $A \subset \R^d$ is defined as the boundary $\partial \Gamma$
  of the unique unbounded connected component $\Gamma$ of $A^c$.
\end{de}

The next lemma is required to ensure that the large boundary subsets $\gamma_n^{(j)}$ are path-connected. 
\begin{lemma}
  \label{lem:outer-connected}
  Suppose that $A \subset \R^2$ is bounded, connected and either open or closed.
  If $\partial A$ \red{is} locally connected, then $\out A$ is path-connected.
\end{lemma}
\begin{proof}
  It is a consequence of the Carath\'eodory theorem \cite[Theorem 2.1]{douadyExploringMandelbrotSet} that if $K \subset \R^2$ is a connected, compact set
  with $\R^2 \bs K$ connected and there exists a locally connected, compact set $L$ such that $\partial K \subseteq L \subseteq K$, then there exists
  a continuous, surjective map $\Psi: \R^2 \bs B_1(0) \to \R^2 \bs \intt(K)$. Restricting the map $\Psi$ yields a continuous, surjective
  map $\gamma: \partial B_1(0) \to \partial K$ (the so-called \textit{ Carath\'eodory loop}), showing that $\partial K$ is path-connected.

  Let $A \subset \R^2$ be bounded and connected with $\partial A$ locally connected.
  Let $\Gamma$ denote the unique unbounded connected component of $A^c$ and let $E:=A^c \bs \Gamma$.

  Consider first the case that $A$ is closed.
  Let $K := A \cup E$. Then $K$ is compact, connected and $K^c = \Gamma$ is connected.
  Let $L := \partial A$. Then, $L$ is compact, locally connected and satisfied $\out A = \partial K \subseteq L \subseteq K$ so $\out A$ is path-connected.

  Now suppose that $A$ is open. $\out A = \partial \Gamma$ is connected since $\Gamma$ and $\Gamma^c = A \cup E$ are connected \cite{czarneckiConnectednessBoundaryComplement2011}.
  Furthermore, in this case, $\out A$ is a connected component of $\partial A$ since $\partial \Gamma \subset \Gamma$ and $\Gamma$ is separated
  from any other connected component of $A^c$. It follows that $\out A$ is a connected, locally connected and compact metric space hence path-connected
  \cite[Lemma 16.4]{milnorDynamicsOneComplex1990}.
\end{proof}

The next lemma gives a property of the outer boundary of \red{an open neighbourhood of a} set.
 It shall be utilised in Proposition \ref{prop:large+lonely} to help show that every point in
the boundary $\partial \cO_n$ of the approximating domains is close to a large subset $\gamma_n^{(j)}$ for large $n$.
\begin{lemma}\label{lem:collar-joint-discs}
 
 If $A \subset \R^d$ is a bounded, connected, regular open set such that $\intt(A^c)$ is connected, then 
 \begin{equation*}
  \sup_{x \in \partial A} \dist(x, \out \dil_{\epsilon}(A)) \to 0 \quad \text{as} \quad \epsilon \to 0.
 \end{equation*}
 
\end{lemma}

\begin{proof}

Let $x \in \partial A$. 
By regularity, $A^c$ is the closure of $\intt(A^c)$ so there exists a sequence $(x_n) \subset \intt(A^c)$ with $x_n \to x$.
We claim that for each $n$, there exists $\epsilon_n > 0$ such that $x_n$ lies in the unbounded connected component of $\dil_\epsilon(A)^c$ for all $ \epsilon \in (0,\epsilon_n]$.
To see this first note that, since $\intt(A^c)$ is connected and $A$ is bounded, there exists an unbounded,
connected open set $V_n$ such that $\overline{V}_n \subset \intt(A^c)$ and $x_n \in V_n$.
The claim follows from the fact that $V_n$ is a subset of $\dil_\epsilon(A)^c$ for small enough $\epsilon$. 

Without loss of generality, assume that $\epsilon_{n+1} < \epsilon_n$ for all $n$.
For each $\epsilon \in (\epsilon_{n-1}, \epsilon_n]$, $x$ lies in $\dil_{\epsilon}(A)$ and $x_n$ lies in the unbounded connected component of $\intt(\dil_{\epsilon}(A)^c)$, so, 
\begin{equation*}
 \delta_x(\epsilon) := \dist (x, \out \dil_{\epsilon}(A)) \leq |x - x_n |. 
\end{equation*}
Since $\epsilon_n\to 0$ monotonically  as $n \to \infty$ and $|x - x_n| \to 0$ as $n \to \infty$, we have $\delta_x(\epsilon) \to 0$ as $\epsilon \to 0$.
$\delta_x(\epsilon) $ is equal to the distance from $x$ to the unbounded component of $\dil_\epsilon(A)^c$.
Since the latter set is nested for decreasing $\epsilon > 0$, $\delta_x(\epsilon)$ in fact tends to zero monotonically as $\epsilon \to 0$.
Finally, $\partial A$ is compact and $\delta_x(\epsilon)$ is continuous in $x$ so an application of Dini's theorem yields
\begin{equation*}
 \delta(\epsilon) := \sup_{x \in \partial A} \delta_x(\epsilon) \to 0 \quad \text{as} \quad \epsilon \to 0.
\end{equation*}
\end{proof}

Next, we prove a couple of useful elementary topological facts. 
\begin{lemma}
  \label{lem:containment}
  If $A,B \subset \R^d$ are such that $B$ is open and connected, $A \cap B \neq \emptyset$ and $\partial A \subset B^c$,
  then $B \subset A$. 
\end{lemma}

\begin{proof}
  Suppose for contradiction that $ A^c \cap B \neq \emptyset$. $B$ is path-connected so
  there exists a path in $B$ from a point in $A \cap B$ to a point in $A^c \cap B$.
  Such a path must intersect $\partial A$ which is the desired contradiction.
\end{proof}

\begin{lemma}
  \label{lem:fill-hole}
  If $A \subset \R^d$ is a connected open set such that $\#_c(A^c) < \infty$,
  then the union of $A$ with any connected component of $A^c$ is open and connected.
\end{lemma}

\begin{proof}
  Let $D$ be any connected component of $A^c$.
  Since $\#_c(A^c) < \infty$, there exists an open neighbourhood $U$ of $D$ such that
  $U$ does not intersect any other connected component of $A^c$.
  Consequently, $U \bs D \subset A$ and so $A \cup D = A \cup U$.
  The lemma follows from the fact that the union of two open, connected sets with nonempty intersection is open and connected.
\end{proof}

In Proposition \ref{prop:large+lonely}, we shall assume that the limit domain $\cO$ is bounded, regular and satisfies $\#_c \intt(\cO^c) = \#_c(\cO^c) < \infty$.
The next lemma collects some properties of domains satisfying these hypotheses.
Intuitively, such a domain has a finite number of holes $D_1,...,D_N$ which do not touch each other
and which do not touch the unbounded exterior of the domain. The set $D_{N+1}$ below is essentially the domain $\cO$ with all the holes filled in.

\begin{lemma}
  \label{lem:seperated-holes}
  Suppose that $\cO \subset \R^d$ is a bounded, connected, regular open set such that $ \#_c \intt(\cO^c) = \#_c(\cO^c) < \infty$.
  Let $D_1,...,D_N \subset \R^d$ denote the bounded connected components of $\intt(\cO^c)$.
  Let $D_{N+1} \subset \R^d$ denote the complement of the unbounded connected component of $\cO^c$.
  Then,
  \begin{enumerate}[label= \rm(\alph*)]
  \item the collection of closed sets $\overline{D}_1, ... ,\overline{D}_N, D_{N+1}^c$ is pairwise disjoint,
  \item $\intt(D_j^c)$ is connected for each $j \in \set{1,...,N+1}$,
  \item $D_j$ is regular for each $j \in \set{1,...,N+1}$ and 
  \item $\partial D_1,...,\partial D_{N+1}$ are the connected components of $\partial \cO$.
  \end{enumerate}
\end{lemma}

\begin{proof}
  Let $E_{N+1} \subset \R^d$ denote the unbounded connected component of $\intt(\cO^c)$, so that
  \begin{equation}\label{eq:int-Oc-decomp}
    \intt(\cO^c) = D_1 \cup \cdots \cup D_N \cup E_{N+1}.
  \end{equation}
  By the regularity of $\cO$ and the fact that the closure of the union of two sets is the union of the closure,
  \begin{equation}\label{eq:Oc-decomp}
    \cO^c = \overline{\intt(\cO^c)} = \overline{D}_1 \cup \cdots \cup \overline{D}_N \cup \overline{E}_{N+1}.
  \end{equation}
  By construction, we have that $\#_c\intt(\cO^c) = N+1$. 
  By the hypothesis $\#_c \intt(\cO^c) = \#_c(\cO^c)$, we must in fact have $\#_c(\cO^c) = N+1$
  and this can only hold if the collection of closed sets $\overline{D}_1,...,\overline{D}_N,\overline{E}_{N+1}$ is exactly the collection of connected components
  of $\cO^c$ and hence must be pairwise disjoint.
  In particular, since $\overline{E}_{N+1}$ is the unique unbounded connected component of $\cO^c$, we must have $D_{N+1} = (\overline{E}_{N+1})^c$,
  proving (a).

  Moving on to the proof of (b), first note that we have the disjoint union
  \begin{equation*}
    \R^d  = \cO \cup \overline{D}_1 \cup \cdots \cup \overline{D}_N \cup D_{N+1}^c
  \end{equation*}
  and so, for any $j \in \set{1,...,N}$,
  \begin{equation}\label{eq:int-Dj-c-formula}
    \intt(D_j^c) = (\overline{D}_j)^c = \cO \cup \br*{ \bigcup_{\substack{ k =1 \\ k \neq j}}^N \overline{D}_k} \cup D_{N+1}^c. 
  \end{equation}
  By $N$ successive applications of Lemma \ref{lem:fill-hole}, we see that the right hand side of (\ref{eq:int-Dj-c-formula})
  is connected, proving (b) for $j \in \set{1,...,N}$. The proof of (b) for $j = N+1$ is immediate since $\intt(D_{N+1}^c) = E_{N+1}$.

  Next, focus on the regularity of $D_j$.
  Since the interior of the union of two disjoint closed sets is the union of the interior of those sets, we have
  \begin{equation}\label{eq:int-Oc-decomp-2}
    \intt(\cO^c) = \intt(\overline{D}_1) \cup \cdots \cup \intt(\overline{D}_N) \cup \intt(\overline{E}_{N+1}).
  \end{equation}
  Combined with (\ref{eq:int-Oc-decomp}) and disjointedness, (\ref{eq:int-Oc-decomp-2}) implies that, for any $j \in \set{1,...,N}$,
  \begin{align*}
    \intt(\overline{D}_j) & = \overline{D}_j \cap \br*{  \intt(\overline{D}_1) \cup \cdots \cup \intt(\overline{D}_N) \cup \intt(\overline{E}_{N+1})}  \\
     & = \overline{D}_j \cap \br*{ D_1 \cup \cdots \cup D_N \cup E_{N+1}} = D_j, 
  \end{align*}
  that is, $D_j$ is regular. $D_{N+1}$ is also regular because $D_{N+1}^c = \overline{E}_{N+1}$ and $E_{N+1}$ is open.
  
The fact that $D_j$ is an open, connected subset of $\R^d$ and $D_j^c$ is connected ensures that $\partial D_j$ is connected
for each $j \in \set{1,...,N+1}$  \cite{czarneckiConnectednessBoundaryComplement2011}.
Then (d) follows from (\ref{eq:Oc-decomp}) and the fact that the collection of closed connected sets $\partial D_1,...,\partial D_{N+1}$
is pairwise disjoint.
\end{proof}

The following lemma, concerning Hausdorff convergence for the boundaries of approximations of an open set from below, follows immediately
from Lemma \ref{lem:met-conv-to-hauss} below. 

\begin{lemma}\label{lem:nested}
If $A \subset \R^d$ and $A_n \subset \R^d, \, n \in \N$, are bounded open sets such that
$A_n \subseteq A_{n+1} \subseteq A $ for all $n \in \N$ and $A = \cup_{n=1}^\infty A_n$, then $\d_H (\partial A_n, \partial A) \to 0 \quad \text{as} \quad n \to \infty$.
\end{lemma}

\begin{figure}[htbp]\label{fig:holes}
	\centering
	\includegraphics[scale=0.9]{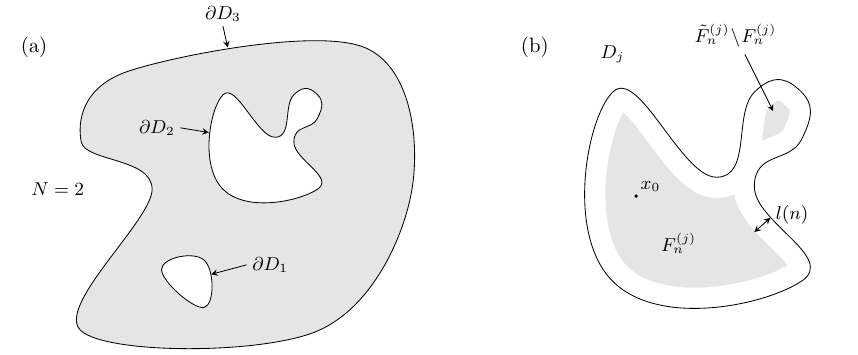}
        \caption{Illustration for the proof of Proposition \ref{prop:large+lonely}}
\end{figure}

\begin{prop}
  \label{prop:large+lonely}
  Suppose that $\cO \subset \R^2$ is a bounded, connected, regular open set  such that $\vol(\partial \cO) = 0$ and $\#_c \intt(\cO^c) = \#_c(\cO^c) < \infty$.
  Suppose that $\cO_n \subset \R^2$, $n \in \N$, is a collection of bounded open sets such that $\partial \cO_n$ is locally connected for all $n \in \N$ and 
  \begin{equation*}
    l(n) = \d_H(\cO_n,\cO) + \d_H(\partial \cO_n,\partial \cO) \to 0 \quad \text{as} \quad n \to \infty.
  \end{equation*}
  Let $D_j \subset \R^2$, $j \in \set{1,...,N+1}$, denote the sets in Lemma \ref{lem:seperated-holes}.
  Then, there exists:
  \begin{itemize}
  \item $n_0 \in \N$,
  \item  a sequence $\epsilon(n) > 0$, $n \geq n_0$, with $\epsilon(n) \to 0$ as $n \to \infty$ and $\epsilon(n) \geq 2 l(n)$
  \item path-connected subsets $\gamma_n^{(j)} \subseteq \partial \cO_n$, $j \in \set{1,...,N+1}$, $n \geq n_0$,
  \end{itemize}
  such that for all  $n \geq n_0$ we have
\begin{equation}
    \label{eq:diam-gamma}
    \diam(\gamma_n^{(j)}) \geq \diam(\partial D_j) - \epsilon(n) \qquad (j \in \set{1,...,N+1})
\end{equation}
and
\begin{equation}
  \label{eq:lonely-are-close}
  \sup_{x \in \partial \cO_n} \dist(x, \gamma_n^{(1)} \cup \cdots \cup \gamma_n^{(N+1)}) \leq \epsilon(n).
\end{equation}
\end{prop}

\begin{proof}
  \textit{Step 1 (Construction of $\gamma_n^{(j)}$).}

 Let
 \begin{equation*}
   \tilde{F}^{(j)}_n := \intt(D_j \bs \col{l(n)}D_j) \qquad (j \in \set{1,...,N+1},\,n \in \N).
 \end{equation*}
 Choose any point $x_0 \in D_j $. Then there exists $n_0 \in \N$ large enough such that  $x_0 \in \tilde{F}_n^{(j)}$ for all $n \geq n_0$.
 For every $j \in \set{1,...N+1}$ and $n \geq n_0$, define $F_n^{(j)}$ as the unique path-connected component of the open set $\tilde{F}_n^{(j)}$ containing the point $x_0$.
 
 $F_n^{(j)}$ is open, bounded, connected and satisfies
 \begin{equation*}
    F_n^{(j)} \subset F_{n+1}^{(j)} \subset  D_j
  \end{equation*}
 for all $n$.
 Furthermore, since any path in $D_j$ from $x_0$ to any point $x \in  D_j$ lies in $\tilde{F}_n^{(j)}$ for all large enough $n$,
 we have $x \in F_n^{(j)}$ for all large enough $n$ and so
 \begin{equation*}
   \bigcup_{n=1}^\infty F_n^{(j)} =  D_j.
 \end{equation*}
 By Lemma \ref{lem:nested}, we have
 \begin{equation}\label{eq:hauss-for-Fn}
   \epsilon_1^{(j)}(n) := \d_H(\partial F^{(j)}_n, \partial D_j ) \to 0 \quad \text{as} \quad n \to \infty \qquad (j \in \set{1,...,N+1}, n \geq n_0).
 \end{equation}
 
 Let us now focus on the case $j \in \set{1,...,N}$.
 The definitions of $F_n^{(j)}$ and $l(n)$ ensure that $\partial \cO_n$ does not intersect $F_n^{(j)}$ and $F_n^{(j)} \cap \cO_n^c \neq \emptyset$ hence
 $F_n^{(j)} \subset \cO_n^c$ by Lemma \ref{lem:containment}.
 Define $(\cO_n^c)_j$ as the unique connected component of $\cO_n^c$ such that the connected set $F_n^{(j)}$ is contained in $(\cO_n^c)_j$.

 By Lemma \ref{lem:seperated-holes} (a), we can ensure that $n_0$ is large enough so that the collection of open sets
  \begin{equation*}
    \dil_{2l(n)}(D_1),...,\dil_{2l(n)}(D_N), \dil_{2l(n)}(D_{N+1}^c)
  \end{equation*}
 is pairwise disjoint for $n \geq n_0$.
 Then, for $n \geq n_0$, every point in $\partial \dil_{2l(n)}(D_j)$ lies in $\cO$ at a distance $\geq 2 l(n)$ from $\partial \cO$ so
\begin{equation*}
   \partial \dil_{2l(n)}(D_j) \subset \cO_n \subset (\cO_n^c)_j^c. 
 \end{equation*}
 Since also $(\cO^c_n)_j \cap \dil_{2l(n)}(D_j) \neq \emptyset$, an application of Lemma \ref{lem:containment} yields
 \begin{equation}
   \label{eq:Oncj-containment}
  (\cO_n^c)_j \subset \dil_{2l(n)}(D_j).
 \end{equation}
for all $n \geq n_0$.
 
 Consequently, $(\cO_n^c)_j$ is bounded for $n \geq n_0$ and we
 can use the notion of outer boundary (cf. Definition \ref{de:outer}) to make the definition 
 \begin{equation}
   \label{eq:gam_n^j}
    \gamma_n^{(j)} := \out (\cO_n^c)_j \qquad (j \in \set{1,...,N}, \, n \geq n_0). 
 \end{equation}
By Lemma \ref{lem:outer-connected}, $\gamma_n^{(j)}$ is path-connected.
 Since $F_n^{(j)} \subset (\cO_n^c)_j$ for all $n \geq n_0$, we have
 \begin{equation}
   \label{eq:diam-gamma_n^j}
   \diam(\gamma_n^{(j)}) \geq \diam(\partial F_n^{(j)}) \geq \diam(\partial D_j) - 2 \epsilon_1^{(j)}(n)
 \end{equation}
 where the final inequality holds by (\ref{eq:hauss-for-Fn}).

 The construction of $\gamma_n^{(N+1)}$ is very similar. In this case, for every $n \geq n_0$, $F_n^{(N+1)}$ is contained inside $\cO_n$
 and we define $\cO_{n,0}\subset \R^d$ as the unique connected component of $\cO_n$ containing $F_n^{(N+1)}$ .
 $\cO_{n,0}$ is bounded because $\cO_n$ is bounded hence we can make the definition 
 \begin{equation}
   \label{eq:gam_n^Np1}
   \gamma_n^{(N+1)} := \out \cO_{n,0} \qquad (n \geq n_0).
 \end{equation}
 By Lemma \ref{lem:outer-connected}, $\gamma_n^{(N+1)}$ is path-connected.
 We have $F_{n}^{(N+1)} \subset \cO_{n,0}$ so (\ref{eq:diam-gamma_n^j}) holds for $j = N+1$ and $n \geq n_0$.

 \noindent
 \textit{Step 2 (Properties of $\gamma_n^{(j)}$).}

 Let 
 \begin{equation}
   \label{eq:eps_2^j}
   \epsilon_2^{(j)}(n) := \sup_{x \in \partial D_j} \dist(x, \out \dil_{2 l(n)}(D_j)) \qquad (j \in \set{1,...,N+1},n \in \N). 
 \end{equation}
 For each $j \in \set{1,...,N+1}$, $D_j$ satisfies the hypotheses of Lemma \ref{lem:collar-joint-discs}
 by Lemma \ref{lem:seperated-holes}, hence $\epsilon_2^{(j)}(n) \to 0$ as $n \to \infty$.

 We claim that
 \begin{equation}
   \label{eq:lonely-close-local}
   \sup_{x \in \partial D_j} \dist(x, \gamma_n^{(j)}) \leq \max \set{\epsilon_1^{(j)}(n),\epsilon_2^{(j)}(n)}
 \end{equation}
 for each $j \in \set{1,...,N+1}$ and large enough $n$.
 Fix $x \in \partial D_j$.
 By the definition of $\epsilon_1^{(j)}(n)$, there exists $y_1 \in \partial F_n^{(j)}$ such that $|y_1 - x| \leq \epsilon_1^{(j)}(n)$.
 By the definition of $\epsilon_2^{(j)}(n)$, there exists $y_2 \in \out \dil_{2l(n)}(D_j)$ such that $|y_2 - x| \leq \epsilon_2^{(j)}(n)$.
 
 Focus first on the case $j \in \set{1,...,N}$.
 By (\ref{eq:Oncj-containment})  and the fact that  $y_2$ lies in the unbounded connected component of $\dil_{2l(n)}(D_j)^c$,
  $y_2$ lies in the unbounded connected component of the complement of $(\cO_n^c)_j$ for all $n \geq n_0$. 
 In addition, we have that $y_1 \in (\cO_n^c)_j$.
 Consequently, the path $\gamma$ consisting the union of a straight line from $y_1$ to $x$
 and a straight line from $x$ to $y_2$ must intersect $\gamma_n^{(j)} = \out (\cO_n^c)_j$.
 Inequality (\ref{eq:lonely-close-local}) for $j \in \set{1,...,N}$ follows from the fact that every point $y$ in
 the path $\gamma$ satisfies $|y - x| \leq \max \set{\epsilon_1^{(j)}(n),\epsilon_2^{(j)}(n)}$.

 The proof of (\ref{eq:lonely-close-local}) for $j = N+1$ is very similar.  $y_1$ lies in  $\cO_{n,0}$ and $y_2$ lies in
 the unbounded connected component of $(\cO_{n,0})^c$ so the path consisting of the union of a straight line from $y_1$ to $x$ and
 a straight line from $x$ to $y_2$ intersects $\gamma_n^{(N+1)} = \out \cO_{n,0}$.

Let
 \begin{equation}
   \label{eq:epsilon-n}
   \epsilon(n) := 2 \max \set{\epsilon_1^{(1)}(n),...,\epsilon_1^{(N+1)}(n),\epsilon_2^{(1)}(n),...,\epsilon_2^{(N+1)}(n),l(n)} \qquad (n \geq n_0).
\end{equation} 
Then (\ref{eq:diam-gamma}) is satisfied since (\ref{eq:diam-gamma_n^j}) holds for all $j \in \set{1,...,N+1}$
so it remains to prove (\ref{eq:lonely-are-close}).
But (\ref{eq:lonely-are-close}) follows from (\ref{eq:lonely-close-local}) by the observation that, for large enough $n$,
\begin{equation*}
  \sup_{x \in \partial \cO_n} \dist(x, \gamma_n^{(1)} \cup \cdots \cup \gamma_n^{(N+1)}) \leq l(n) + \sup_{x \in \partial \cO} \dist(x, \gamma_n^{(1)} \cup \cdots \cup \gamma_n^{(N+1)})
\end{equation*}
and, using Lemma \ref{lem:seperated-holes} (d),
\begin{align*}
  \sup_{x \in \partial \cO} \dist(x, \gamma_n^{(1)} \cup \cdots \cup  \gamma_n^{(N+1)}) &
                                                                                        \leq \max \set{ \sup_{x \in \partial D_1}\dist(x,\gamma_n^{(1)}),...,  \sup_{x \in \partial D_{N+1}}\dist(x,\gamma_n^{(N+1)}) } \\
  & \leq \max \set{\epsilon_1^{(1)}(n),...,\epsilon_1^{(N+1)}(n),\epsilon_2^{(1)}(n),...,\epsilon_2^{(N+1)}(n)} \leq \epsilon(n).
\end{align*}

\end{proof}

\subsection{Proof of Theorem \ref{th:mosco}}\label{sec:mosco-proof}

Firstly, $\cO$ and $\cO_n$, $n \in \N$, satisfy the hypotheses of Proposition \ref{prop:large+lonely}.
  Let $\epsilon(n), \,D_j$, $N$ and $\gamma_n^{(j)}$ be as in that proposition.
  $\epsilon(n)$ satisfies $\epsilon(n) \geq 2 l(n)$ and $\epsilon(n) \to 0$ as $n \to \infty$ so by Proposition \ref{prop:mosco}
  it suffices to show that there exist numerical constants $C, \alpha > 0$ such that for large enough $n$,
  \begin{equation}
    \label{eq:poin-for-mosco-th}
    \forall \, u \in H^1_0(\cO) : \quad \norm{u}_{L^2(\col{\epsilon(n)} \cO)} \leq C \epsilon(n) \norm{\nabla u}_{L^2(\col{\alpha \epsilon(n)} \cO)}
  \end{equation}
  and
  \begin{equation}
    \label{eq:uniform-poin-for-mosco-th}
    \forall \, u \in H^1_0(\cO_n) : \quad \norm{u}_{L^2(\col{\epsilon(n)} \cO_n)} \leq C \epsilon(n) \norm{\nabla u}_{L^2(\col{\alpha \epsilon(n)} \cO_n)}.
  \end{equation}
  (\ref{eq:poin-for-mosco-th}) follows immediately from Theorem \ref{thm:poin} (with $C=10 \sqrt{3}$ and $\alpha= 2\sqrt{2}$)
  so it remains to show (\ref{eq:uniform-poin-for-mosco-th}). 

  Let 
  \begin{equation}
    \mathcal{V}_n := \br*{\bigcup_{j=1}^{N+1} \gamma_n^{(j)}}^c \qquad (n \in \N).
  \end{equation}
  By inequality (\ref{eq:lonely-are-close}) in Proposition \ref{prop:large+lonely}, every point in $\partial \cO_n$
  is at most a distance $\epsilon(n)$ from $\partial \mathcal{V}_n = \cup_{j=1}^{N+1} \gamma_n^{(j)}$ so
  \begin{equation}
    \col{\epsilon(n)} \cO_n \subseteq \dil_{2 \epsilon(n)} (\partial \mathcal{V}_n) = \col{2 \epsilon(n)} \mathcal{V}_n.
  \end{equation}
  Inequality (\ref{eq:diam-gamma}) yields
  \begin{equation}
    Q(\partial \mathcal{V}_n) \geq \min \set*{ \diam (\partial D_j) - \epsilon(n): j \in \set{1,...,N+1}}.
  \end{equation}
  Since $D_j$ are bounded open sets, we have $\diam(\partial D_j) > 0$ and so  $4 \sqrt{2} \epsilon(n) < Q(\partial \mathcal{V}_n)$ for large enough $n$.
  Consequently, an application of Theorem \ref{thm:poin} shows that
  \begin{equation}
    \norm{u}_{L^2(\col{\epsilon(n)} \cO_n)}  \leq \norm{u}_{L^2(\col{2 \epsilon(n)} \mathcal{V}_n)}\leq C\epsilon(n) \norm{\nabla u}_{L^2(\col{4 \sqrt{2} \epsilon(n)} \mathcal{V}_n)}
  \end{equation}
  for all large enough $n$ and all $u \in H^1_0(\cO_n)$.
  Noting that any $u \in H^1_0(\cO_n)$ must vanish almost everywhere on $\cO_n^c$, we see that we have established (\ref{eq:uniform-poin-for-mosco-th})
  (with $C=20 \sqrt{3}$ and $\alpha = 4\sqrt{2}$),
  completing the proof.

\begin{remark}\label{rem:poin-not-regular}
  Recall that the explicit Poincar\'e-type inequality of Theorem \ref{thm:poin} does not require regularity of the domain.
  Interestingly, this is exploited in the proof of Theorem \ref{th:mosco}. There, Theorem \ref{thm:poin} is applied to $u \in H^1_0(\cO_n)$
  as a function in $H^1_0(\mathcal{V}_n)$, and $\mathcal{V}_n$ is certainly not regular in general.
\end{remark}

\begin{remark}\label{rem:Q-hypothesis-replacement}
  We can replace the hypothesis $Q(\partial \cO) > 0$ with the hypothesis that the connected components
  of $\partial \cO$ are path-connected.
  Indeed, with this replacement, the path-connected components of $\partial \cO$ are $\partial D_1,...,\partial D_{N+1}$,
  where $D_1,...,D_{N+1}$ are the bounded, open sets of Lemma \ref{lem:seperated-holes}, hence
  \begin{equation}
    Q(\partial \cO) = \min\set{ \diam (\partial D_j) : j \in \set{1,...,N+1} } > 0. 
  \end{equation}
\end{remark}

\subsection{Hausdorff convergence for pixelated domains}\label{sec:hausd-conv-pixel}

We finish the section by showing that pixelation approximations (cf. Definition \ref{defn:pixel})
satisfy the Hausdorff convergence condition of Theorem \ref{th:mosco} (under suitable hypotheses).
From this, we will be able to conclude that the pixelation approximations converge in the Mosco sense, which will be utilised in the study of
computational spectral problems in Section \ref{sec:spectral}.

\begin{lemma}\label{lem:met-conv-to-hauss}
  If $A \subset \R^d$ and $A_n \subset \R^d$, $n \in \N$, are bounded open sets such that $A_n \subset A$ for all $n \in \N$
  and any compact set $F \subset A$ is a subset of $A_n$ for all large enough $n$, then
  \begin{equation*}
    \d_H(A,A_n) + \d_H(\partial A, \partial A_n) \to 0 \quad \text{as} \quad n \to \infty. 
  \end{equation*}
\end{lemma}

\begin{proof}

We have that $\d_H(\partial A_n, \partial A) = \max \set{D_1,D_2}$, where 
\begin{equation*}
 D_1 := \sup_{x \in \partial A} \dist(x, \partial A_n) \quad \text{and} \quad D_2 := \sup_{x \in \partial A_n} \dist(x, \partial A).
\end{equation*}
Focusing on $D_1$, let $\epsilon > 0$ and $x \in \partial A$.
By hypothesis, we can let $N( \epsilon,x) \in  \N$ be large enough so that $B_\epsilon(x) \cap A_n \neq \emptyset$ for all $n \geq N( \epsilon,x)$.
The ball $B_\epsilon(x)$ also intersects $A^c \subseteq A_n^c$ for all $n$ so in fact $B_\epsilon(x)$
intersects $\partial A_n$ for all $n \geq N(\epsilon, x)$. 
This shows that $\dist(x, \partial A_n)< \epsilon$ for all $n \geq N(\epsilon ,x)$.
  By compactness of $\partial A$, we can let $N(\epsilon) := \sup_{x \in \partial A} N(\epsilon, x)< \infty$.
  Then,
\begin{equation*}
 \forall n \geq N(\epsilon): \sup_{x \in \partial A} \dist(x, \partial A_n) < \epsilon
\end{equation*}
hence  $D_1 \to 0$ as $n \to \infty$.

Focusing on $D_2$, suppose for contradiction that there exists a subsequence $(\partial A_{n_k})_{k \in \N}$ such that 
\begin{equation*}
 \sup_{x \in \partial A_{n_k}} \dist (x , \partial A) \geq C
\end{equation*}
for some $C > 0$ independent of $k$.
Then there exists $x_{n_k} \in \partial A_{n_k}$, $k \in \N$, such that $\dist(x_{n_k}, \partial A) \geq C$. 
$B_{C/2}(x_{n_k})$ is contained in $A$ and intersects $A_{n_k}^c$ for all $k$ so there exists $y_{n_k} \in B_{C/2}(x_{n_k})$, $k \in \N$, such that $y_{n_k} \in A_{n_k}^c \cap A$ for all $k$.
$(y_{n_k})$ satisfies $\dist(y_{n_k}, \partial A) \geq C/ 2> 0$ for all $k$.
Let $y \in  \overline{A}$ be an accumulation point of $(y_{n_k})$.
 $y$ must satisfy $\dist(y, \partial A) \geq C/2$ so there exists $\delta > 0$ such that $\overline{B}_{\delta}(y) \subset \intt(A)=A$. 
By hypothesis,  $B_\delta(y) \subset A_n$ for all large enough $n$.
But this is a contradiction to fact that $y$ is an accumulation point of $y_{n_k} \in A_{n_k}^c$, $k \in \N$. 
It follows that $D_2 \to 0$ as $n \to \infty$ hence $\d_H(\partial A_n, \partial A) \to 0$ as $n \to \infty$.

Since $A_n \subseteq A$, it remains to show that $\sup_{x \in A} \dist(x, A_n) \to 0$.
Let $\epsilon > 0$. By hypothesis, there exists $N \in \N$ such that $A \bs \col{\epsilon} A \subset A_n$, hence
$\sup_{x \in A} \dist(x, A_n) \leq \epsilon $, for all $n \geq N$, completing the proof.
\end{proof}

In the next proposition, the hypothesis that the limit domain $\cO$ is regular is crucial.
Indeed, Proposition \ref{prop:counterexample} features an example of a non-regular domain for which the pixelation approximations
do not converge in the Hausdorff sense.
The basic idea of the proof is to introduce approximations from below $\tilde{A}_n$ and $E_n$ for the sets $\cO$ and $\intt(\cO^c)$
respectively which ``sandwich'' the boundary $\partial \cO_n$ of the pixelated domain.
 
\begin{prop}\label{prop:pixel}
If $\cO \subseteq \R^d $ is a bounded, regular open set such that $\vol(\partial \cO) = 0$, and $\cO_n,\,n \in \N$, are the pixelated domains for $\cO$, then 
\begin{equation*}
 l(n) = \d_H(\cO, \cO_n) + \d_H(\partial \cO, \partial \cO_n) \to 0 \quad \text{as} \quad n \to \infty.
\end{equation*}

\end{prop}

\begin{proof}

All limits in the proof are as $n \to \infty$.

Define the following collection of open sets 
\begin{equation}
 \mathcal{B} := \bigcup_{n \in \N} \mathcal{B}_n \quad \text{where} \quad \mathcal{B}_n := \set*{j + (- \tfrac{1}{n}, \tfrac{1}{n})^d: j \in \Z^d_n}. 
\end{equation}
The elements of $\mathcal{B}_n$ are open boxes of side-length $2/n$ and hence overlap.
Let $A_n$ denote the union of all elements of $\mathcal{B}_n$ which are subsets $\cO$ and
let $E_n$ denote the union of all elements of $\mathcal{B}_n$ which are subsets of
$\intt(\cO^c)$.

We claim that for any compact set $F \subset \cO$, we have $F \subset A_n$ for all large enough $n$.
Let $\epsilon > 0$ be small enough so that $\dil_\epsilon(F) \subset \cO$ and let $n$ be any positive integer which is large enough
such that $ 2\tfrac{\sqrt{d}}{n} < \epsilon$.
Let $x \in F$. There exists $j \in \Z_n^d$ such that $|x - j| \leq \tfrac{\sqrt{d}}{2} \tfrac{1}{n}$.
Then,
\begin{equation}
  j + (-\tfrac{1}{n}, \tfrac{1}{n})^d \subset \dil_{2\sqrt{d}/n}(F) \subset \dil_{\epsilon}(F) \subset \cO
\end{equation}
so the box $j + (-\tfrac{1}{n},\tfrac{1}{n})^d$ is a subset of $A_n$ and consequently $x \in A_n$.
It follows that $F \subset A_n $, proving the claim.

We also have $A_n \subset \cO$ so $A_n$ and $\cO$ satisfy the hypotheses of Lemma \ref{lem:met-conv-to-hauss}.
We similarly have $E_n \subset \intt(\cO^c)$ and, for any compact set $F \subset \intt(\cO^c)$, $F \subset E_n$ for all large enough $n$. 
Applying Lemma \ref{lem:met-conv-to-hauss} to $ A_n$ and $B_X(0) \cap E_n$ for large enough $X > 0$, we obtain
\begin{equation}\label{eq:pixel-dH-dAn-dO}  
 \d_H(\cO, A_n) + \d_H(\partial \cO,\partial A_n) \to 0 \quad \text{and} \quad \d_H(B_X(0)\cap \intt(\cO^c), B_X(0) \cap E_n) + \d_H(\partial \cO,\partial E_n) \to 0,
\end{equation}
where regularity was used in the second limit to ensure that $\partial \intt(\cO^c) = \partial \cO$.

Define also the following subset of $A_n$,
\begin{equation}
  \label{eq:An-tilde-defn}
  \tilde{A}_n := \intt \br*{ \bigcup_{j \in A_n \cap \Z_n^d}(j + [- \tfrac{1}{2n}, \tfrac{1}{2n}]^d) } \qquad (n \in \N).
\end{equation}
Any point in $A_n$ is in a box $j + (-\tfrac{1}{n},\tfrac{1}{n})^d$  for some $j \in A_n \cap \Z_n^d$ hence at most a distance $\sqrt{d}/n$ from a point in $\tilde{A}_n$. Consequently
\begin{equation}
  \label{eq:dH-An-An-tilde}
  \d_H(A_n,\tilde{A}_n) \leq \frac{\sqrt{d}}{n}.
\end{equation}
\begin{figure}[b]
	\centering
%
%

\begin{tikzpicture}[>=stealth, scale=0.6]
	\begin{scope}
	\clip(-1.04,-1.1) rectangle (7,7);
	\filldraw[fill=gray, fill opacity=0.2] plot [smooth, tension=0.1] coordinates {(5.1,0)
		(5.3192,0.49721)
		(5.1635,1.5522)
		(4.527,2.1218)
		(4.3176,2.9303)
		(4.2272,3.1484)
		(3.6278,2.7389)
		(3.4583,3.6534)
		(2.3859,4.4292)
		(2.1654,4.6494)
		(1.5091,5.3821)
		(1.2156,5.3578)
		(-0.24339,5.1465)
		(-0.47989,5.0017)
		(-1.0966,4.4101)
		(-2.4294,5.2513)
		(-1.9383,4.4732)
		(-2.0178,3.6371)
		(-3.4312,3.8314)
		(-3.7211,3.272)
		(-4.3298,2.5566)
		(-4.3895,2.3173)
		(-4.8592,1.3633)
		(-4.9383,0.55139)
		(-4.5043,-0.20594)
		(-5.2439,-0.42949)
		(-5.0806,-1.1439)
		(-4.7681,-1.4435)
		(-4.5446,-2.6202)
		(-3.7586,-2.7901)
		(-3.1979,-3.6927)
		(-3.655,-3.1585)
		(-3.0529,-4.136)
		(-2.1565,-4.0651)
		(-1.841,-4.5439)
		(-0.51475,-4.6021)
		(0.010646,-5.3232)
		(0.32796,-4.6202)
		(0.96709,-4.5001)
		(1.3211,-5.4164)
		(2.253,-4.625)
		(2.9302,-4.4713)
		(3.1599,-3.8622)
		(4.1702,-3.7335)
		(4.2083,-2.9068)
		(4.3236,-2.2745)
		(4.7459,-1.8029)
		(5.1655,-1.3038)
		(4.8395,-1.0262)
		(5.1,0)};
	\foreach \x in {-1,0,1,2,3,4,5,6}
    \foreach \y in {-1,0,1,2,3,4,5,6}
    {
    \fill (\x,\y) circle (0.04);
    }
	\fill[fill=blue, opacity=0.25] (-2,-2) rectangle (3,4);
	\fill[fill=blue, opacity=0.25] (3,-2) rectangle (4,2);
	
	\draw (0.5,0.5) node{$A_n$};
	\draw (2.65,4.65) node{\footnotesize$\partial \mathcal O$};
	\end{scope}

	\begin{scope}[shift={(9,0)}]
	\clip(-1.1,-1.1) rectangle (7,7);
	\foreach \x in {-1,0,1,2,3,4,5,6}
    \foreach \y in {-1,0,1,2,3,4,5,6}
    {
    \fill (\x,\y) circle (0.04);
    }
	\filldraw[fill=gray, fill opacity=0.2] plot [smooth, tension=0.1] coordinates {(5.1,0)
		(5.3192,0.49721)
		(5.1635,1.5522)
		(4.527,2.1218)
		(4.3176,2.9303)
		(4.2272,3.1484)
		(3.6278,2.7389)
		(3.4583,3.6534)
		(2.3859,4.4292)
		(2.1654,4.6494)
		(1.5091,5.3821)
		(1.2156,5.3578)
		(-0.24339,5.1465)
		(-0.47989,5.0017)
		(-1.0966,4.4101)
		(-2.4294,5.2513)
		(-1.9383,4.4732)
		(-2.0178,3.6371)
		(-3.4312,3.8314)
		(-3.7211,3.272)
		(-4.3298,2.5566)
		(-4.3895,2.3173)
		(-4.8592,1.3633)
		(-4.9383,0.55139)
		(-4.5043,-0.20594)
		(-5.2439,-0.42949)
		(-5.0806,-1.1439)
		(-4.7681,-1.4435)
		(-4.5446,-2.6202)
		(-3.7586,-2.7901)
		(-3.1979,-3.6927)
		(-3.655,-3.1585)
		(-3.0529,-4.136)
		(-2.1565,-4.0651)
		(-1.841,-4.5439)
		(-0.51475,-4.6021)
		(0.010646,-5.3232)
		(0.32796,-4.6202)
		(0.96709,-4.5001)
		(1.3211,-5.4164)
		(2.253,-4.625)
		(2.9302,-4.4713)
		(3.1599,-3.8622)
		(4.1702,-3.7335)
		(4.2083,-2.9068)
		(4.3236,-2.2745)
		(4.7459,-1.8029)
		(5.1655,-1.3038)
		(4.8395,-1.0262)
		(5.1,0)};
	\fill[fill=blue, opacity=0.25] (-2,-2) rectangle (3.5,3.5);
	\fill[fill=blue, opacity=0.25] (3.5,-2) rectangle (4.5,2.5);
	\fill[fill=blue, opacity=0.25] (4.5,-0.5) rectangle (5.5,1.5);
	\fill[fill=blue, opacity=0.25] (-2,3.5) rectangle (2.5,4.5);
	\fill[fill=blue, opacity=0.25] (-0.5,4.5) rectangle (1.5,5.5);
	
	\draw (0.5,0.5) node{$\mathcal O_n$};
	\draw (2.8,4.6) node{\footnotesize$\partial \mathcal O$};
	\end{scope}

	\begin{scope}[shift={(17,-1)}]
	\clip(-0.1,-0.1) rectangle (7.1,8);
	\foreach \x in {-1,0,1,2,3,4,5,6,7}
    \foreach \y in {-1,0,1,2,3,4,5,6,7}
    {
    \fill (\x,\y) circle (0.04);
    }
	\filldraw[fill=gray, fill opacity=0.2] plot [smooth, tension=0.1] coordinates {(5.1,0)
		(5.3192,0.49721)
		(5.1635,1.5522)
		(4.527,2.1218)
		(4.3176,2.9303)
		(4.2272,3.1484)
		(3.6278,2.7389)
		(3.4583,3.6534)
		(2.3859,4.4292)
		(2.1654,4.6494)
		(1.5091,5.3821)
		(1.2156,5.3578)
		(-0.24339,5.1465)
		(-0.47989,5.0017)
		(-1.0966,4.4101)
		(-2.4294,5.2513)
		(-1.9383,4.4732)
		(-2.0178,3.6371)
		(-3.4312,3.8314)
		(-3.7211,3.272)
		(-4.3298,2.5566)
		(-4.3895,2.3173)
		(-4.8592,1.3633)
		(-4.9383,0.55139)
		(-4.5043,-0.20594)
		(-5.2439,-0.42949)
		(-5.0806,-1.1439)
		(-4.7681,-1.4435)
		(-4.5446,-2.6202)
		(-3.7586,-2.7901)
		(-3.1979,-3.6927)
		(-3.655,-3.1585)
		(-3.0529,-4.136)
		(-2.1565,-4.0651)
		(-1.841,-4.5439)
		(-0.51475,-4.6021)
		(0.010646,-5.3232)
		(0.32796,-4.6202)
		(0.96709,-4.5001)
		(1.3211,-5.4164)
		(2.253,-4.625)
		(2.9302,-4.4713)
		(3.1599,-3.8622)
		(4.1702,-3.7335)
		(4.2083,-2.9068)
		(4.3236,-2.2745)
		(4.7459,-1.8029)
		(5.1655,-1.3038)
		(4.8395,-1.0262)
		(5.1,0)};
	\fill[fill=blue, opacity=0.25] (3,4) rectangle (8,7);
	\fill[fill=blue, opacity=0.25] (5,2) rectangle (8,4);
	\fill[fill=blue, opacity=0.25] (6,-1) rectangle (8,2);
	\fill[fill=blue, opacity=0.25] (0,6) rectangle (3,7);
	\fill[fill=blue, opacity=0.25] (2,5) rectangle (3,6);
	
	\draw (5.5,5.5) node{$E_n$};
	\draw (2.65,4.65) node{\footnotesize$\partial \mathcal O$};
	\end{scope}

\end{tikzpicture}

	\caption{Sketch of the domains $A_n$ (left), $\cO_n$ (centre) and $E_n$ (right).}
\end{figure}
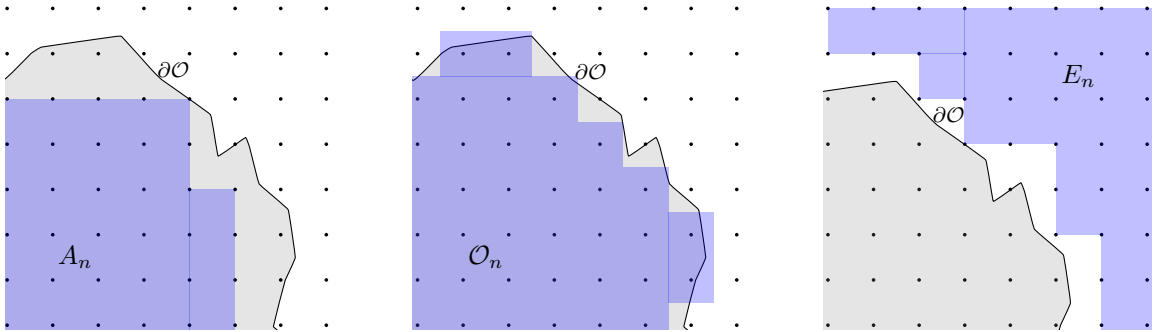
Firstly, we claim that 
\begin{equation}\label{eqpr:pixel-1}
   \tilde{A}_n \subseteq \cO_n \subseteq E^c_n.
\end{equation}
To see the first inclusion in \eqref{eqpr:pixel-1}, note that any grid point $j \in A_n \cap \Z_n^d$ is in $\cO$ so the 
the corresponding cell  $j + [- 1/(2n),1/(2n)]^d$ is a subset of $\overline{\cO}_n$.
Focus now on the second inclusion.
Any point in $x$ in $E_n$ lies in $j + (\tfrac{1}{n},\tfrac{1}{n})^d$ for some $j \in \Z_n^d \cap E_n$.
Since the corners of the closed box $j + [- \tfrac{1}{n}, \tfrac{1}{n}]^d$ lie in $\Z_n^d \cap \overline{E}_n$,
$x$ lies in $j' + [-\tfrac{1}{2n}, \tfrac{1}{2n}]^d$ for some $j' \in \Z_n^d \cap \overline{E}_n$.
This shows that
\begin{equation*}
   E_n \subseteq \bigcup_{j \in \overline{E}_n \cap \Z_n^d} \big(j + [-\tfrac{1}{2n}, \tfrac{1}{2n}]^d\big).
 \end{equation*}
The fact that any point in $\Z_n^d \cap \overline{E}_n$ lies in $\cO^c$ implies that $E_n \subseteq \cO_n^c$, proving the claim.
 
Secondly, we claim that
\begin{equation}\label{eqpr:pixel-2}
    \vol(E^c_n \bs A_n) \to 0.
  \end{equation}
Let $X > 0$ be large enough so that $E_n^c \bs A_n \subset B_X(0)$.
Then,
\begin{equation}
  \label{eq:vol-Bnc-An-decomp}
 \vol(E^c_n \bs A_n) = \vol(B_X(0)) - \vol(B_X(0) \cap A_n) - \vol(B_X(0) \cap E_n).
\end{equation}
Using continuity of measures from below, the hypothesis that $\vol(\partial \cO) = 0$ and regularity, we have
\begin{equation*}
 \vol (B_X(0) \cap A_n) \to \vol(B_X(0) \cap \cO)
\end{equation*}
and
\begin{equation*}
  \vol (B_X(0) \cap E_n) \to \vol(B_X(0) \cap \intt(\cO^c)) = \vol(B_X(0) \cap \cO^c).
\end{equation*}
It follows that the right hand side of (\ref{eq:vol-Bnc-An-decomp}) tends to zero, proving the claim.

Next we claim that 
\begin{equation}\label{eqpr:pixel-3}
 \sup_{x \in \partial \cO_n} \dist(x, \partial A_n \cup \partial E_n) \to 0. 
\end{equation}
This can be seen by considering an expanding ball around any point in $E^c_n \bs A_n$.
More precisely, for any $x \in E^c_n \bs A_n$, define the quantity
\begin{equation*}
 r(x) := \inf \set{ r > 0 :  B_r(x) \cap ( \partial A_n \cup \partial E_n ) \neq \emptyset } \in [0,\infty).
\end{equation*}
For all $x \in E_n^c \bs A_n$, we have 
\begin{equation*}
 \dist(x, \partial A_n \cup \partial E_n) = r(x) \quad \text{and} \quad B_{r(x)}(x) \subseteq E_n^c \bs A_n.
\end{equation*}
Note that we consider that $B_0(x) = x$.
The set $E^c_n \bs A_n$ is compact so we can define
\begin{equation*}
 r_n := \sup_{x \in E_n^c \bs A_n} r(x) < \infty
\end{equation*}
and there exists $x_n \in E_n^c \bs A_n $ such that $r_n = r(x_n)$.  
The limit \eqref{eqpr:pixel-2}, combined with the fact that $B_{r_n}(x_n) \subseteq E_n^c \bs A_n$, yields 
\begin{equation*}
 \vol(B_{r_n}(x_n)) \leq \vol(E^c_n \bs A_n) \to 0
\end{equation*}
which implies that $r_n \to 0$.
\eqref{eqpr:pixel-3} is obtained by applying the inclusions \eqref{eqpr:pixel-1} to get
\begin{align*}
 \sup_{x \in \partial \cO_n} \dist(x, \partial A_n \cup \partial E_n) & \leq \sup_{x \in E_n^c \bs \tilde{A}_n} \dist (x , \partial A_n \cup \partial E_n) \\
 & \leq \max \set*{ \frac{\sqrt{d}}{ n} ,  \sup_{x \in E_n^c \bs A_n} \dist (x , \partial A_n \cup \partial E_n) } \\
 & = \max \set*{ \frac{\sqrt{d}}{ n} , r_n} \to 0, 
\end{align*}
where the second inequality holds because any point in $A_n \bs \tilde{A}_n$ is at most a distance $ \sqrt{d} / n$ from $\partial A_n$.

Combining the two limits in (\ref{eq:pixel-dH-dAn-dO}) gives
\begin{equation}\label{eqpr:pixel-5}
  \d_H(\partial \cO, \partial A_n \cup \partial E_n) \to 0
\end{equation}
and
\begin{equation}
  \label{eqpr:pixel-6}
  \d_H(\partial A_n,\partial E_n) \to 0.
\end{equation}   

Furthermore, we claim that 
\begin{equation}\label{eqpr:pixel-7}
 \sup_{x \in \partial A_n \cup \partial E_n } \dist(x, \partial \cO_n) \to 0.
\end{equation}
This can be seen by considering length minimising lines between $\partial\tilde{A}_n$ and $\partial E_n$. 
More precisely, let $x \in \partial E_n $ and let $y = y(x) \in \partial \tilde{A}_n$ be such that $|x-y| = \dist(x, \partial \tilde{A}_n)$. 
The straight line connecting $x$ and $y$ must intersect $\partial \cO_n$ since one end is in $\cO_n$  and the other is in $\cO_n^c$. 
This fact implies that 
\begin{align*}
 \sup_{x \in \partial E_n} \dist(x, \partial \cO_n) & \leq \sup_{x \in \partial E_n} \dist(x, \partial \tilde{A}_n)\\
 & \leq \sup_{x \in \partial E_n} \dist(x, \partial A_n) + \d_H(\partial A_n, \partial \tilde{A}_n) \to 0
\end{align*}
where the limit holds by \eqref{eqpr:pixel-6} and (\ref{eq:dH-An-An-tilde}).
It can be similarly seen that 
\begin{equation*}
 \sup_{x \in \partial A_n} \dist(x, \partial \cO_n) \to 0.
\end{equation*}
giving us \eqref{eqpr:pixel-7}.
The limits \eqref{eqpr:pixel-3} and \eqref{eqpr:pixel-7} prove that $\d_H(\partial \cO_n, \partial A_n \cup \partial E_n) \to 0$ which,
combined with \eqref{eqpr:pixel-5}, gives $\d_H(\partial \cO, \partial \cO_n) \to 0$.

By (\ref{eq:pixel-dH-dAn-dO}) and regularity of $\cO$,
\begin{equation}\label{eqpr:pixel-8}
  \d_H(A_n,\cO) \to 0 \quad \text{and} \quad \d_H(E^c_n,\cO)  \to 0
\end{equation}
so, in particular, $\d_H(A_n,E_n^c) \to 0$. 
Using this combined with the inclusions \eqref{eqpr:pixel-1} and inequality (\ref{eq:dH-An-An-tilde}) for $\d_H(A_n,\tilde{A}_n)$, we have
\begin{align*}
 \d_H(E_n^c,\cO_n)  = \sup_{x \in E_n^c} \dist(x, \cO_n) \leq \sup_{x \in E_n^c } \dist(x, \tilde{A}_n) \to 0.
\end{align*}
Combining this with the second limit in (\ref{eqpr:pixel-8}), shows that $\d_H(\cO, \cO_n) \to 0$, completing the proof. 
\end{proof}

\section{Arithmetic algorithms for the spectral problem}\label{sec:spectral}

In this section we study the computability of the Dirichlet spectrum. Subsections \ref{subsec:spec-algorithm1} and \ref{subsec:spec-algorithm2} are devoted to the proof of Theorem \ref{thm:spec-exist}, whereas Section \ref{subsec:spec-counter} provides a counterexample that proves Proposition \ref{prop:spec-not-exist}. We begin with the study of generalised matrix eigenvalue problems which arise naturally from finite element approximations.

\red{In this section, it is convenient for us to explicitly indicate the dependence of arithmetic algorithms on the evaluation set.
  Precisely, an arithmetic algorithm $\Gamma:\Omega \to \mathcal M$ with input $\Lambda$ refers to  a tower of arithmetic algorithms of height 0
  for a computational problem $(\Omega,\Lambda,\mathcal M,\Xi)$.}

\subsection{Matrix pencil eigenvalue problem}\label{subsec:spec-algorithm1}

First, we show that there exists a family of arithmetic algorithms capable of solving the \textit{computational matrix pencil eigenvalue problem} to
arbitrary specified precision. 
Let  
\begin{equation*}
\Omega_{\textrm{mat},M} := \set*{(A,B) \in  (\R^{M \times M})^2: \, A, B \text{ symmetric, } B \text{ positive definite}},
\end{equation*}  
 \begin{equation*}
 \Lambda_{\mat,M} := \set*{ (A,B) \mapsto A_{j,k} + \mathrm{i} B_{j,k} : A_{j,k}, B_{j,k} \text{ matrix element of }A,B \text{ resp.}}.
\end{equation*}
and
\begin{equation*}
\Lambda_{\mat,M}^\epsilon := \Lambda_{\mat,M} \cup \set{(A,B) \mapsto \epsilon}.
\end{equation*} 
Applying a-posteriori bounds of Oishi \cite{Oishi2001} for the matrix eigenvalue problem gives the following family of arithmetic algorithms:
\begin{lemma}\label{lem:alg-matrix}
 For each $\epsilon > 0$ and $m \in \N$, there exists an arithmetic algorithm $\Gamma_{\mat,M}^\epsilon: \Omega_{\mat,M} \to \R^M$ with input $\Lambda_{\mat,M}^\epsilon$, such that 
 \begin{equation*}
  |\Gamma^\epsilon_{\mat,M}(A,B)_k - \lambda_k| \leq \epsilon \quad \text{for all} \quad (A,B) \in \Omega_{\mat,M} \quad \text{and} \quad k\in \set{1,...,M},
 \end{equation*}
 where $\lambda_k$ denotes the eigenvalues of the matrix pencil $(A,B)$.
\end{lemma}
\begin{proof}

Since $\Lambda_{\mat,M}^\epsilon$ is a finite set, we can define the information available to the (arithmetic) algorithm $\Gamma_{\mat,M}^\epsilon$ as $\Lambda_{\Gamma_{\mat,M}^\epsilon}(A,B) = \Lambda_{\mat,M}^\epsilon$.
By Gaussian elimination, the matrix elements of $B^{-1}$ can be computed with a finite number of arithmetic operations.
Then the eigenvalues of the matrix pencil $(A,B)$ are exactly the eigenvalues $\br{\lambda_k}_{k=1}^M$  of $E:= B^{-1}A$ and the matrix elements of $E$ are 
accessible to the algorithm.

By the Jacobi eigenvalue algorithm (cf. \cite{Schonhage}) there exists a family of approximations $(\tilde{\lambda}_k^m,\tilde{x}_k^m) \in \R \times \R^M$, $ k \in \set{1,...,M},\, m \in \N$, such that
$(\tilde{\lambda}_k^m,\tilde{x}_k^m)$ can be computed with finitely many arithmetic operations and
 \begin{equation}\label{eq:mat-error-to-zero}
  \norm{P_m^T D_m P_m - E}_ F + \norm{P_m^T P_m - I}_F \to 0 \quad \text{as} \quad m \to \infty
\end{equation}  
where
\begin{equation*}
  D_m := \diag(\tilde{\lambda}_1^m,...,\tilde{\lambda}_M^m) \quad \text{and} \quad P_m := (\tilde{x}_1^m,...,\tilde{x}_M^m).
\end{equation*}
Here, $\norm{\cdot}_F$ denotes the Frobenius matrix norm.
Note that (\ref{eq:mat-error-to-zero}) is equivalent to the statement that $(\tilde{\lambda}_k^m)_{k=1}^M$ tends to $(\lambda_k)_{k=1}^M$ as $m \to \infty$ and
$(\tilde{x}_k^m)_{k=1}^M$ tends to the corresponding orthonormal basis of eigenvectors.
By \cite[Theorem 2]{Oishi2001},
\begin{equation*}
  |\lambda_k - \tilde{\lambda}_k^m| \leq |\tilde{\lambda}_k^m| \norm{P_m^T P_m - I}_F + \norm{P_m^T D_m P_m - E}_F =: \mathcal{E}_k(m)
\end{equation*}  
for all $k \in \set{1,...,M}$.
Let $m(\epsilon)$ denote the smallest positive integer such that $\mathcal{E}_k(m(\epsilon)) \leq \epsilon$ for all $k \in \set{1,...,M}$.
$m(\epsilon)$ can be determined by the algorithm since, for each $m$, $\mathcal{E}_k(m)$ can be computed in finitely many arithmetic operations.
The proof is completed by letting
\begin{equation*}
\Gamma_{\mat,M}^\epsilon(A,B) := (\tilde{\lambda}_k^{m(\epsilon)})_{k=1}^M. 
\end{equation*} 
\end{proof}

\subsection{Algorithm for computing Laplacian eigenvalues}\label{subsec:spec-algorithm2}

Next, we show that there exists a family of arithmetic algorithms capable of computing, to arbitrary specified precision,
the spectrum of the Dirichlet Laplacian on domains of the form 
\begin{equation}\label{eq:gen-pix-form}
 \mathcal{U} = \intt \br*{ \bigcup_{j = 1}^N  \left( x_j + \left[ - \tfrac{1}{2n}, \tfrac{1}{2n}\right]^2 \right) }
\end{equation}
with $n, N \in \N$ and $(x_1,...,x_N) \in (\Z^2_n)^N$. 
Let 
\begin{equation*}
 \Omega_{\pix,n} := \set*{ \mathcal{U} \subset \R^2 : \exists N \in \N \text{ and }(x_1,...,x_N) \in (\Z^2_n)^N \text{ such that \eqref{eq:gen-pix-form} holds} },
\end{equation*} 
\begin{equation*}
 \Lambda_{\pix,n} := \set*{ \cU \mapsto \chi_\cU(x): x \in \Z_n^2 }.
\end{equation*}
and
\begin{equation}
  \Lambda_{\pix,n}^\epsilon := \Lambda_{\pix,n} \cup  \set*{ \cU \mapsto N(\cU) } \cup \set*{\cU \mapsto \epsilon},
\end{equation} 
where $N(\cU):=|\cU \cap \Z_n^2|$ denotes the number of ``pixels'' that make up $\cU$.
The results of Liu and Oishi \cite{LiuOishi2013} combined with Lemma \ref{lem:alg-matrix} yield the following:
\begin{lemma}\label{lem:alg-pix}
 Let $n\in\N$ be fixed. For each $\epsilon > 0$, there exists an arithmetic algorithm $\Gamma_{\pix,n}^\epsilon:\Omega_{\pix,n} \to \cl(\C)$ with input $\Lambda_{\pix,n}^\epsilon$, such that 
 \begin{equation*}
  \dAW{\Gamma_{\pix,n}^\epsilon(\cU)}{\sigma(\cU)} \leq \epsilon \quad \text{for all} \quad \cU \in \Omega_{\pix,n}.
 \end{equation*}
\end{lemma}

\begin{proof}
 
 First, we fix a finite subset $\Lambda_{\Gamma_{\pix, n}^\epsilon}(\cU) \subset \Lambda_{\pix, n}^\epsilon$ which defines the information available to the (arithmetic) algorithm $\Gamma_{\pix, n}^\epsilon$, which we aim to construct.
 Choose 
 \begin{equation*}
    \Lambda_{\Gamma_{\pix,n}^\epsilon}(\cU) := \set{ \cU \mapsto \chi_\cU(x) : x\in \Z_n^2 \cap \overline{B}_{\kappa(\cU)}(0)} \cup  \set*{ \cU \mapsto N(\cU) } \cup \set*{\cU \mapsto \epsilon},
 \end{equation*}  
 where $\kappa(\cU)$ is defined as the smallest positive integer such that $|\overline{B}_{\kappa(\cU)}(0) \cap \cU \cap \Z_n^2| = N(\cU)$.
 The consistency hypothesis (cf. Definition \ref{def:Algorithm}) holds because  $\kappa(\cU)$ can be computed with a finite number of arithmetic computations
 from subsets of $\set{ f(\cU) : f \in \Lambda_{\Gamma_{\pix,n}^\epsilon}(\cU)}$.
  With this choice, the algorithm has access to the list $(x_1,...,x_N) \in (\R^2)^N$ for which \eqref{eq:gen-pix-form} holds.  
 
 \begin{figure}
 	\centering
 	\includegraphics[width=0.7\textwidth]{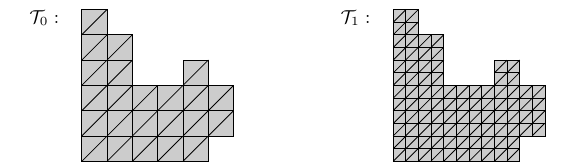}
 	\caption{Sketch of the triangulation $\mathcal T_m$ and refinement for $m\in\{0,1\}$ and fixed $n\in\N$.}
 	\label{fig:refinement}
 \end{figure}
 
 Using this list, we can construct, for each $m \in \N$, a uniform triangulation $\mathcal{T}^m$ of $\cU$ such that the elements of $\mathcal{T}^m$ have diameter $\sqrt{2}/(n2^{m}) =: \sqrt{2}h$ (cf. Figure \ref{fig:refinement}).
 Let $V^m \subset H^1_0(\cU)$ denote the piecewise-linear continuous finite element space for the triangulation $\mathcal{T}^m$. 
 Let $\set{\phi_k^m}_{m = 1}^{M_0}$ denote the basis of `hat' functions for $V^m$, where $M_0 := \dim(V^m)$.
 Let $\set{\lambda_k}_{k=1}^\infty$ denote the Dirichlet eigenvalues of the domain $\cU$, ordered such that $\lambda_k \leq \lambda_{k+1}$ for all $k \in \N$. 
 
 The Ritz-Galerkin finite element approximations for $\set{\lambda_k}_{k=1}^\infty$ are the eigenvalues $\set{\lambda_k^m}_{k=1}^{M_0}$, $m \in \N$, of the matrix pencil $(A^m, B^m)$,
 where the matrix elements of the matrices $A^m$ and $B^m$ read
 \begin{equation*}
  A_{j,k}^m := \inner{\nabla \phi_j^m,\nabla \phi_k^m}_{L^2(\cU)} \quad \text{and} \quad B_{j,k}^m := \inner{ \phi_j^m,\ \phi_k^m}_{L^2(\cU)} \qquad (j,k \in \set{1,...,M_0})
 \end{equation*}
 respectively.
These matrix elements can be computed from the information $\br*{x_1,...,x_N}$, $n$ and $m$ with a finite number of arithmetic computations.
Note that $A$ and $B$ are symmetric and $B$ is positive definite.

 In \cite{LiuOishi2013}, the authors introduce a quantity 
 \begin{equation*}
  q^m := l^m + (C_0/n2^{m})^2, 
 \end{equation*}
 where,
\begin{itemize}
 \item $C_0 > 0$ can be bounded above by an explicit expression \cite[Section 2]{LiuOishi2013} and,
 \item $l^m$ is the maximum eigenvalue of a matrix pencil $(D^m,E^m)$ \cite[eq. (3.22)]{LiuOishi2013}. Here, the matrix elements of $D^m$ and $E^m$ are
   explicitly constructed from inner products between overlapping basis functions of piecewise-linear finite element spaces on $\mathcal{T}^m$ and hence can be computed with
   a finite number of arithmetic computations on $(x_1,...,x_N)$, $n$ and $m$. Also, $E^m$ is diagonal and positive definite.
\end{itemize}
By \cite[Remark 3.3]{LiuOishi2013}, it holds that $ q^m \to 0 \quad \text{as} \quad m \to \infty$.
 \cite[Theorem 4.3]{LiuOishi2013} states that, for each $m \in \N$ and each $k \in \set{1,...,M_0}$, if $q^m \lambda_k^m < 1$, then 
 \begin{equation}\label{eq:eig-apost}
  \lambda_k^m / (1 + \lambda_k^m q^m) \leq \lambda_k \leq \lambda_k^m.
 \end{equation}
 
 Since the matrix elements of $A^m,B^m, D^m$ and $E^m$ are available to the algorithm $\Gamma_{\pix,n}^\epsilon$, by Lemma \ref{lem:alg-matrix}, the approximations
 \begin{equation*}
  \lambda_k^{m,\delta} := \Gamma_{\mat,M_0}^\delta(A^m,B^m)_k \qquad \text{and} \qquad q^{m,\delta} := \Gamma_{\mat,M_0}^\delta(D^m,E^m)_{M_0}
 \end{equation*}
 are also available to $\Gamma^\epsilon_{\pix,n}$, for any $\delta > 0$.
 These approximations provide upper and lower bounds for $\lambda_k^m$ and $q^m$
 \begin{equation}\label{eq:mateig-interval}
  \lambda_k^m \in [\lambda_k^{m,\delta} - \delta, \lambda_k^{m,\delta} + \delta] , \qquad q^m \in [q^{m,\delta} - \delta, q^{m,\delta} + \delta].
 \end{equation}

 We claim that if $(M,m,\delta) \in \N \times \N \times \R_+$ is such that $(\lambda_M^{m,\delta} + \delta)(q^{m,\delta} + \delta) < 1$, then
 \begin{equation}\label{eq:dAW-apost}
  \dAW{\set{\lambda_k^{m,\delta}}_{k=1}^M}{\sigma(\cU)} \leq \delta + \mathcal{E}_1(M,m,\delta) + \mathcal{E}_2(M,m,\delta)
 \end{equation} 
 where
 \begin{equation*}
  \mathcal{E}_1(M,m,\delta) := \max \set*{(q^{m,\delta} + \delta )(\lambda^{m,\delta}_k + \delta)^2 /(1 + (q^{m,\delta} - \delta)_+ (\lambda^{m,\delta}_k - \delta)_+): k \in \set{1,...,M}}
 \end{equation*}
 and
 \begin{equation*}
  \mathcal{E}_2(M,m,\delta) := 2^{  - (\lambda_M^{m,\delta} - \delta)/2 + 1 }.
 \end{equation*}
 To see this, first note that, using the formula
 \begin{equation*}
   \dAW{A}{B} \leq \d_H(A,B) \qquad (A,B \subset \C \textrm{ bounded}),
 \end{equation*}
 we have
\begin{align}\label{eq:dAW-apost-1}
  \mathrm{d_{AW}}\big( \set{ \lambda_k^{m,\delta} }_{k=1}^M & , \sigma(\cU) \big)
   \leq \d_H(\set{\lambda_k^{m,\delta}}_{k=1}^M,\set{\lambda_k^m}_{k=1}^M) +  \dAW{ \set{ \lambda_k^m }_{k=1}^M }{\sigma( \cU) }
   \nonumber 
  \\ 
  &\leq \delta + \sum_{j=1}^{\floor{\lambda_M}} 2^{-j} \min \set*{ 1 , \sup_{|x| \leq j} \abs*{ \dist \br*{ x, \set{\lambda_k^m}_{k=1}^M }
    - \dist \br*{ x, \set{\lambda_k}_{k=1}^\infty } } } + \sum_{j = \ceil{\lambda_M}}^\infty 2^{-j}.
\end{align} 
Since $ \dist(x , \set{ \lambda_k }_{k=1}^\infty ) =  \dist(x , \set{ \lambda_k }_{k=1}^M ) $ for $|x|\leq \floor{\lambda_M}$, the second term on the right hand side of \eqref{eq:dAW-apost-1} is bounded by
\begin{align*}
  \dAW{\set{\lambda_k^m}_{k=1}^M}{ \set{\lambda_k}_{k=1}^M } &\leq \d_H(\set{\lambda_k^m}_{k=1}^M, \set{\lambda_k}_{k=1}^M )\\
  &\leq \max \set*{|\lambda_k - \lambda_k^m| : k \in \set{1,...,M}}
  \leq \mathcal{E}_1(\delta,M,m)
\end{align*}
where the final inequality follows from \eqref{eq:eig-apost} and \eqref{eq:mateig-interval}.
Applying \eqref{eq:eig-apost} and \eqref{eq:mateig-interval} again and noting that the condition $(\lambda_M^{m,\delta} + \delta)(q^{m,\delta} + \delta) < 1$ ensures that $\lambda^m_M q^m \leq 1$,
we have
\begin{equation*}
\sum_{j = \ceil{\lambda_M}}^\infty 2^{-j} \leq 2^{- \lambda_M + 1} \leq 2^{- \lambda^m_M/2 + 1} \leq \mathcal{E}_2(M,\delta,m),
\end{equation*} 
bounding the third term on the right hand side of \eqref{eq:dAW-apost-1} and proving \eqref{eq:dAW-apost}.

Let $\delta(M) := 1/M$.
Define $m(M)$ as the smallest $m \in \N$ such that
\begin{equation*}
  (q^{m,\delta(M)} + \delta(M) )(\lambda^{m,\delta(M)}_M + \delta(M))^2 \leq 1/M.
\end{equation*} 
Then, $\mathcal{E}_1(M, m( M), \delta( M)) \leq 1/M$.
For each $M \in \N$, $m(M)$ can be determined by the algorithm since it can compute the quantities $ q^{m,\delta(M)}$ and $\lambda_M^{m,\delta(M)}$.
Define $M(\epsilon)$ as the smallest positive integer such that
\begin{equation*}
  \delta \circ M(\epsilon) + \mathcal{E}_1(M(\epsilon), m \circ M(\epsilon), \delta \circ M(\epsilon))
  + \mathcal{E}_2(M(\epsilon), m \circ M(\epsilon), \delta \circ M(\epsilon)) \leq \epsilon,
\end{equation*} 
where $m\circ M(\epsilon) = m(M(\epsilon))$ and $\delta\circ M(\epsilon) = \delta(M(\epsilon))$.
For each $\epsilon > 0$, $M(\epsilon)$ can be determined by the algorithm since $\mathcal E_1$ and $\mathcal E_2$ can be computed.
The proof of the lemma is completed by letting
\begin{equation}\label{eq:pix-alg-defn}
\Gamma_{\pix,n}^\epsilon(\cU) := \big\{ \lambda_k^{m \circ M(\epsilon), \delta \circ M(\epsilon)} \big\}_{k=1}^{M(\epsilon)}.
\end{equation}  

\end{proof}
\begin{remark}
	The results of \cite{LiuOishi2013} are formulated for connected domains only. While this assumption is not necessarily satisfied for domains in $\Omega_{\pix,n}$, the results from \cite{LiuOishi2013} can be applied to every connected component of a set $\cU \in \Omega_{\pix,n}$ separately. This is justified, because the Dirichlet spectrum of $\cU$ is simply the union of the Dirichlet spectra of all connected components of $\cU$. Moreover, any $\cU \in \Omega_{\pix,n}$ consists of only finitely many connected components, which can be determined in a finite number of steps from the information given in $\Lambda_{\pix,n}$.
\end{remark}
We have shown that the Dirichlet spectrum of an arbitrary (but fixed) pixelated domain from $\Omega_{\pix,n}$ is computable via an arithmetic algorithm. It remains to pass from pixelated domains to arbitrary domains in $\Om_1$ (recall \eqref{eq:Omega1_def}). 
To this end, we combine the Mosco convergence results from Section \ref{sec:gener-mosco-conv} with the following proposition.
\begin{prop}\label{prop:alg-mosco}
  Let
  \begin{equation*}
    \Omega_M := \set*{ \cO \subset \R^2: \cO \text{\rm{ open, bounded and }}\cO_n \Mto \cO \text{\rm{  where }}\cO_n\text{\rm{  pixelated domains for }}\cO }.
  \end{equation*}
  Then there exists a sequence of arithmetic algorithms $\Gamma_n:\Omega_M \to \cl(\C)$, $n \in \N$, with input $\Lambda_0$ such that 
 \begin{equation*}
  \dAW{\Gamma_n(\cO) }{\sigma(\cO)} \to 0 \quad \text{as} \quad n \to \infty \qquad \text{for all} \qquad  \cO \in \Omega_M. 
 \end{equation*}
\end{prop}

\begin{proof}
Let $\cO \in \Omega_M$. 
Let $\cO_n \subset \R^2$, $n \in \N$, denote the corresponding pixelated domains (cf. Definition \ref{defn:pixel}). 

We shall construct a family of (arithmetic) algorithms $\Gamma_n : \Omega_1 \to \cl(\C)$ with input $\Lambda_0$.
First, we define the information available to each algorithm $\Gamma_n$, by fixing a finite subset $\Lambda_{\Gamma_n} \subset \Lambda_0$.
Let 
\begin{equation*}
 \Lambda_{\Gamma_n} := \set*{ \cO \mapsto \chi_\cO(x): x \in \Z^2_n \text{ with } |x| \leq n},
\end{equation*}
so that the algorithm $\Gamma_n$ has access to the set $\set{x_1,...,x_N} := \cO \cap B_n(0) \cap \Z_n^2$.

Let 
\begin{equation*}
 \tilde{\cO}_n := \intt \br*{ \bigcup_{j=1}^N \Big( x_j + \left[-\tfrac{1}{2n},\tfrac{1}{2n} \right]^2 \Big) }.
\end{equation*}
It holds that $\tilde{\cO}_n \in \Omega_{\pix,n}$ for each $n$ and, since $\cO$ is bounded, $\tilde{O}_n = \cO_n$ for all sufficiently large $n$.
Hence using the hypothesis that $\cO_n \Mto \cO$, $\tilde{\cO}_n$ converges to $\cO$ in the Mosco sense as $n \to \infty$. 
By Lemma \ref{lem:mosco-implies-Hauss}, 
\begin{equation*}
 \dAW{ \sigma(\tilde{\cO}_n)}{\sigma(\cO)} \to 0 \quad \text{as} \quad n \to \infty.
\end{equation*}
Since $\tilde{\cO}_n$ is defined entirely by $\Lambda_{\Gamma_n}$, we can define
\begin{equation*}
 \Gamma_n(\cO) = \Gamma_{\pix,n}^{1/n}(\tilde{O}_n)
\end{equation*}
for each $n \in \N$.
Note that the consistency property Definition \ref{def:Algorithm} (ii) holds trivially since $\Lambda_{\Gamma_n}$ does not depend on $\cO$.
The proposition is proved by the fact that 
\begin{align}
 \dAW{\Gamma_{\pix,n}^{1/n}(\tilde{\cO}_n)}{\sigma( \cO)} & \leq \dAW{\Gamma_{\pix,n}^{1/n}(\tilde{\cO}_n)}{\sigma(\tilde{\cO}_n)} + \dAW{\sigma(\tilde{\cO}_n)}{\sigma(\cO)} 
 \nonumber
 \\
 & \leq \frac{1}{n} + o(1) \to 0 \quad \text{as} \quad n \to \infty.
 \label{eq:final_error}
\end{align} 
\end{proof}
Theorem \ref{thm:spec-exist} now follows by combining Proposition \ref{prop:alg-mosco} with Theorem \ref{th:mosco} and Proposition \ref{prop:pixel}.

\begin{remark}
	From the perspective of the SCI hierarchy it is natural to ask whether any finer classification might be possible, specifically in terms of explicit error bounds. This remains an open question and it is clear that our method cannot yield any explicit error bounds for general domains in $\Om_1$. This is perhaps most clearly demonstrated in the final error bound \eqref{eq:final_error}. The behaviour $\mathrm{d_{AW}}(\sigma(\tilde{\cO}_n),\sigma(\cO)) = o(1)$ on the right hand side of \eqref{eq:final_error} was deduced from the Mosco convergence $\cO_n\xrightarrow{\mathrm{M}} \cO$, for which no uniform convergence rate can be expected in general. 
	
	While the question remains open, the fact that no uniform regularity assumption (such as H\"older continuity with fixed exponent) was made on $\del\cO$  suggests that no explicit error bound is possible.
\end{remark}

\subsection{Counter-example}\label{subsec:spec-counter}

In this section we give a counterexample showing that if the regularity assumptions on the domain are relaxed too much, computability fails to be true, i.e. the Dirichlet spectrum cannot be computed by a single sequence of algorithms \red{any more} and at least two limits are necessary. 
Proposition \ref{prop:spec-not-exist} follows immediately from the following result. 
 
\begin{prop}\label{prop:counterexample}
  Let $\Gamma_n: \Omega_0 \to \mathcal{M}$,  $n \in \N$, be any family of arithmetic algorithms with input $\Lambda_0$.
  Then, for any $\cO \in \Omega_0$ with $\vol(\partial\cO) = 0$ and any $\epsilon > 0$, there exists $\cO_\epsilon \in \Omega_0$ with $\cO_\epsilon \subseteq  \cO$ and $\vol(\cO_\epsilon) \leq \epsilon$ such that
  $\Gamma_n(\cO) = \Gamma_n(\cO_\epsilon)$ for all $n$ and, for sufficiently small $\epsilon > 0$, $\sigma(\cO) \neq \sigma(\cO_\epsilon)$. 

\end{prop}

\begin{proof}
	Let $\Gamma_n$ be as hypothesised, let $\eps>0$ and let $\cO \in \Om_0$.
        Define the geometric quantity
	\begin{align*}
		r_{\mathrm{int}}(\cO) &:= \sup\{r>0\,: \,\exists\, \text{square }[s,s+r]\times[t,t+r]\subset \cO \}.
	\end{align*}  
	Openness of $\cO$ implies that $r_{\mathrm{int}}(\cO)>0$.
	For any fixed $n$, $\Gamma_{n}(\cO)$  depends only on finitely many values of $\chi_\cO(x)$, say $x_1^n,\dots,x_{k_n}^n$.
	We assume without loss of generality that the set $\{x_1^n,\dots,x_{k_n}^n\}$ is growing with $n$, i.e. that $\{x_1^n,\dots,x_{k_n}^n\}\subset \{x_1^{n+1},\dots,x_{k_{n+1}}^{n+1}\}$ for all $n$.
        Thus, we may drop the superscript $n$ and merely write $\{x_1,\dots,x_{k_n}\}$. Let us denote by $\{y_1,\dots,y_{l_n}\}$ the subset of points for which $\chi_\cO(y_i)=1$.
	Now, define a new domain $\cO_\eps$ as follows. For $t>0$ define the strips
	$S_t^k := \big(\big((y_k)_1-\f t2,\, (y_k)_1+\f t2\big)\times \R\big)\cap \cO$. Next, let
	\begin{align*}
		\cO_\eps^n := \bigcup_{k=1}^{l_n} S_{2^{-k}\eps}^k \quad \text{and} \quad \cO_\eps := \bigg(\bigcup_{n=1}^\infty \cO_\eps^n\bigg) \cup \col{\epsilon} \cO
	\end{align*}      
	where, recall that $\col{\epsilon} \cO = \{x\in \cO\,:\,\dist(x,\del \cO)<\epsilon\}$. Note that $\cO_\eps$ is bounded, open and connected for any $\eps>0$.
	One has $\chi_\cO(x_k) = \chi_{\cO_\eps}(x_k)$ for all $k\in\{1,\dots,{k_n}\}$ and all $n\in\N$, and therefore, by consistency of algorithms
        (cf. Definition \ref{def:Algorithm} (ii)), $\Gamma_{n}(\cO_\eps) = \Gamma_{n}(\cO)$ for all $n\in\N$.
	
	However, it is easily seen from the min-max principle that for the lowest eigenvalue  $\lambda_1(\cO)$, of $- \Delta_{\cO}$, one has
	\begin{align*}
		\lambda_1(\cO)\leq \f{\pi^2}{r_{\mathrm{int}}(\cO)^2},
	\end{align*}
        Next, we use Poincar\'e's inequality \cite[eq. (7.44)]{GT} to get
	\begin{align*}
		\| u\|_{L^2(\cO_\eps)} \leq C\vol(\cO_\eps)^{\f12} \|\nabla u\|_{L^2(\cO_\eps)}
		\leq C\bigg(\vol(\col{\epsilon} \cO) + \sum_{k=1}^\infty 2^{-k}\eps \diam(\cO) \bigg)^\f12\|\nabla u\|_{L^2(\cO_\eps)}
	\end{align*}
        for some constant $C > 0$ independent of $\epsilon$ and all $u \in H^1_0(\cO_\epsilon)$. Since $\vol(\col{\epsilon} \cO) \to 0$ as $\epsilon \to 0$ by continuity of measures,
        we conclude using the min-max principle that $\lambda_1(\cO_\epsilon) \to \infty$ as $\epsilon \to 0 $ and hence $\sigma(\cO) \neq \sigma(\cO_\epsilon)$ for small enough $\epsilon > 0$.
\end{proof}
\begin{remark}
The counterexample in the proof of Proposition \ref{prop:counterexample} is pathological in the sense that the complement of the domain $\cO_\eps$ may have
infinitely many connected components. This is not crucial. Indeed, one can easily construct a counterexample whose complement has only one connected component:
Let $\cO=(0,1)^2$ be the unit square and 
	\begin{align*}
		\cO_\eps &:= \bigg(\bigcup_{n=1}^\infty \bigcup_{k=1}^{k_n} S_{2^{-k}\eps}^k\bigg) \cup \big((0,1)\times(0,\eps)\big),
	\end{align*}
with the notation from the previous proof. Then $\Gamma_n(\cO_\eps) = \Gamma_n(\cO)$ for all $n$ and $\sigma((0,1)^2)\neq\sigma(\cO_\eps)$ for sufficiently small $\epsilon > 0$ while $ \cO_\eps^c$ is connected. 
\end{remark}

%
%


%
%
%

\section{Numerical Results}\label{sec:numerical-results}
\begin{figure}[h!]
	\centering
	\includegraphics[width=0.45\textwidth, height=0.3\textwidth]{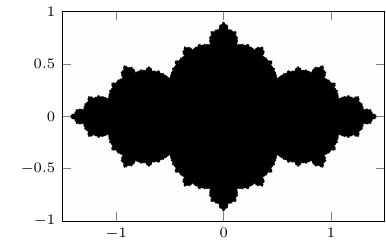}
	\includegraphics[width=0.48\textwidth, height=0.3\textwidth]{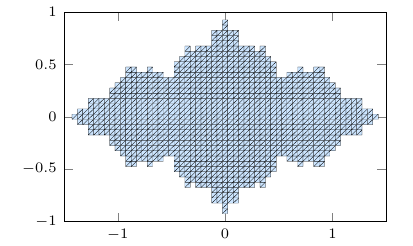}
	\caption{The domain $\cO$ \red{defined by equations (\ref{eq:K-julia}) and  (\ref{eq:O-julia})} (left)
          and \red{a mesh of} \red{a corresponding pixelation approximation} $\cO_n$ for $n = 20$ and $h=1/n$ (right).}
	\label{fig:Julia_O}
      \end{figure}

      \begin{figure}[h!]
	\centering


\pgfplotsset{every tick label/.append style={font=\footnotesize},
	title style={font=\footnotesize}}
 \pgfplotsset{compat=1.3} 


\definecolor{mycolor1}{rgb}{0.00000,0.44700,0.74100}%
\definecolor{mycolor2}{rgb}{0.78331,0.87928,0.12540}%
\definecolor{mycolor3}{rgb}{0.85000,0.32500,0.09800}%
\definecolor{mycolor4}{rgb}{0.92900,0.69400,0.12500}%
\definecolor{mycolor5}{rgb}{0.49400,0.18400,0.55600}%
\definecolor{mycolor6}{rgb}{0.46600,0.67400,0.18800}%
\definecolor{mycolor7}{rgb}{0.30100,0.74500,0.93300}%
\definecolor{mycolor8}{rgb}{0.63500,0.07800,0.18400}%
\definecolor{mycolor9}{rgb}{0.30415,0.76470,0.41994}%
\definecolor{mycolor10}{rgb}{0.12183,0.58905,0.54562}%
\definecolor{mycolor11}{rgb}{0.19236,0.40320,0.55584}%
\definecolor{mycolor12}{rgb}{0.27713,0.18523,0.48990}%

\setlength\figurewidth{0.93\linewidth}
\setlength\figureheight{0.2\figurewidth}

\begin{tikzpicture}

\begin{axis}[%
width=\figurewidth,
height=\figureheight,
at={(0,0)},
scale only axis,
xmin=10.715,
xmax=11.23,
ymin=1,
ymax=6,
yticklabels={20,40,80,160},
ytick={2,3,4,5,6},
xlabel={$\lambda_0^m$},
ylabel={$n$},
ylabel shift = -1pt,
axis background/.style={fill=white},
legend style={at={(0.48,-0.4), \footnotesize, color=black}, anchor=north, legend columns=30, legend cell align=left, align=left, draw=white!15!black}
]

\addlegendimage{only marks, mark=o}
\addlegendimage{only marks, mark=triangle}
\addlegendimage{only marks, mark=triangle, mark options={rotate=180}}
\addlegendimage{only marks, mark=diamond}
\addlegendimage{only marks, mark=triangle, mark options={rotate=90}}
\addlegendimage{only marks, mark=pentagon}
\addlegendimage{only marks, mark=star}

%

\addplot [color=mycolor3, draw=none, mark size=1.6pt, only marks, mark=o, mark options={solid, fill=mycolor2, mycolor2}]
  table[row sep=crcr]{%
12.1664832191615	1\\
};
\addlegendentry{$h=0.050 $\quad}

\addplot [color=mycolor4, draw=none, mark size=1.8pt, only marks, mark=triangle , mark options={solid, fill=mycolor2, mycolor2}]
  table[row sep=crcr]{%
11.9764586561916	1\\
};
\addlegendentry{$h=0.025 $\quad}

\addplot [color=mycolor5, draw=none, mark size=1.8pt, only marks, mark=triangle , mark options={solid, rotate=180, fill=mycolor2, mycolor2}]
  table[row sep=crcr]{%
11.9052332374085	1\\
};
\addlegendentry{$h=0.013 $\quad}

\addplot [color=mycolor6, draw=none, mark size=1.7pt, only marks, mark=diamond , mark options={solid, fill=mycolor2, mycolor2}]
  table[row sep=crcr]{%
11.8780421469148	1\\
};
\addlegendentry{$h=0.006 $\quad}

\addplot [color=mycolor7, draw=none, mark size=1.8pt, only marks, mark=triangle , mark options={solid, rotate=90, fill=mycolor2, mycolor2}]
  table[row sep=crcr]{%
11.8675281793742	1\\
};
\addlegendentry{$h=0.003 $\quad}

\addplot [color=mycolor8, draw=none, mark size=1.6pt, only marks, mark=pentagon , mark options={solid, fill=mycolor2, mycolor2}]
  table[row sep=crcr]{%
11.8634265786792	1\\
};
\addlegendentry{$h=0.002$\quad}

\addplot [color=mycolor1, draw=none, mark size=1.6pt, only marks, mark=star, mark options={solid, fill=mycolor2, mycolor2}]
  table[row sep=crcr]{%
11.8618168752569	1\\
};
\addlegendentry{$h=0.001 $}

\addplot [color=mycolor3, draw=none, mark size=1.6pt, only marks, mark=o , mark options={solid, fill=mycolor9, mycolor9}]
  table[row sep=crcr]{%
11.2226897308728	2\\
};

\addplot [color=mycolor4, draw=none, mark size=1.8pt, only marks, mark=triangle , mark options={solid, fill=mycolor9, mycolor9}]
  table[row sep=crcr]{%
11.050392810324	2\\
};

\addplot [color=mycolor5, draw=none, mark size=1.8pt, only marks, mark=triangle , mark options={solid, rotate=180, fill=mycolor9, mycolor9}]
  table[row sep=crcr]{%
10.9845184252621	2\\
};

\addplot [color=mycolor6, draw=none, mark size=1.7pt, only marks, mark=diamond , mark options={solid, fill=mycolor9, mycolor9}]
  table[row sep=crcr]{%
10.9591151036865	2\\
};

\addplot [color=mycolor7, draw=none, mark size=1.8pt, only marks, mark=triangle , mark options={solid, rotate=90, fill=mycolor9, mycolor9}]
  table[row sep=crcr]{%
10.9492558498216	2\\
};

\addplot [color=mycolor8, draw=none, mark size=1.6pt, only marks, mark=pentagon , mark options={solid, fill=mycolor9, mycolor9}]
  table[row sep=crcr]{%
10.945405126706	2\\
};

\addplot [color=mycolor1, draw=none, mark size=1.6pt, only marks, mark=star, mark options={solid, fill=mycolor9, mycolor9}]
  table[row sep=crcr]{%
10.9438934710698	2\\
};

\addplot [color=mycolor3, draw=none, mark size=1.8pt, only marks, mark=triangle , mark options={solid, fill=mycolor10, mycolor10}]
  table[row sep=crcr]{%
10.984437160706	3\\
};

\addplot [color=mycolor4, draw=none, mark size=1.8pt, only marks, mark=triangle , mark options={solid, rotate=180, fill=mycolor10, mycolor10}]
  table[row sep=crcr]{%
10.8980980918762	3\\
};

\addplot [color=mycolor5, draw=none, mark size=1.7pt, only marks, mark=diamond , mark options={solid, fill=mycolor10, mycolor10}]
  table[row sep=crcr]{%
10.8637348711317	3\\
};

\addplot [color=mycolor6, draw=none, mark size=1.8pt, only marks, mark=triangle , mark options={solid, rotate=90, fill=mycolor10, mycolor10}]
  table[row sep=crcr]{%
10.8502028322227	3\\
};

\addplot [color=mycolor7, draw=none, mark size=1.6pt, only marks, mark=pentagon , mark options={solid, fill=mycolor10, mycolor10}]
  table[row sep=crcr]{%
10.8448829658247	3\\
};

\addplot [color=mycolor8, draw=none, mark size=1.6pt, only marks, mark=star, mark options={solid, fill=mycolor10, mycolor10}]
  table[row sep=crcr]{%
10.8427884581298	3\\
};

\addplot [color=mycolor1, draw=none, mark size=1.8pt, only marks, mark=triangle , mark options={solid, rotate=180, fill=mycolor11, mycolor11}]
  table[row sep=crcr]{%
10.8321553672065	4\\
};

\addplot [color=mycolor3, draw=none, mark size=1.7pt, only marks, mark=diamond , mark options={solid, fill=mycolor11, mycolor11}]
  table[row sep=crcr]{%
10.7888517764149	4\\
};

\addplot [color=mycolor4, draw=none, mark size=1.8pt, only marks, mark=triangle , mark options={solid, rotate=90, fill=mycolor11, mycolor11}]
  table[row sep=crcr]{%
10.7711342107942	4\\
};

\addplot [color=mycolor5, draw=none, mark size=1.6pt, only marks, mark=pentagon , mark options={solid, fill=mycolor11, mycolor11}]
  table[row sep=crcr]{%
10.7640516602363	4\\
};

\addplot [color=mycolor6, draw=none, mark size=1.6pt, only marks, mark=star, mark options={solid, fill=mycolor11, mycolor11}]
  table[row sep=crcr]{%
10.7612476168865	4\\
};

\addplot [color=mycolor7, draw=none, mark size=1.7pt, only marks, mark=diamond , mark options={solid, fill=mycolor12, mycolor12}]
  table[row sep=crcr]{%
10.7613039424554	5\\
};

\addplot [color=mycolor8, draw=none, mark size=1.8pt, only marks, mark=triangle , mark options={solid, rotate=90, fill=mycolor12, mycolor12}]
  table[row sep=crcr]{%
10.7394608156734	5\\
};

\addplot [color=mycolor1, draw=none, mark size=1.6pt, only marks, mark=pentagon , mark options={solid, fill=mycolor12, mycolor12}]
  table[row sep=crcr]{%
10.730423703538	5\\
};

\addplot [color=mycolor3, draw=none, mark size=1.6pt, only marks, mark=star, mark options={solid, fill=mycolor12, mycolor12}]
  table[row sep=crcr]{%
10.7267918815195	5\\
};
\end{axis}

\end{tikzpicture}%
	\caption{FEM approximations $\lambda_0^m$ of \red{the} lowest eigenvalue \red{of $- \Delta_{\cO_n}$} for different values of $h=2^{-m}$ and $n$.}
	\label{fig:EVs}
\end{figure}
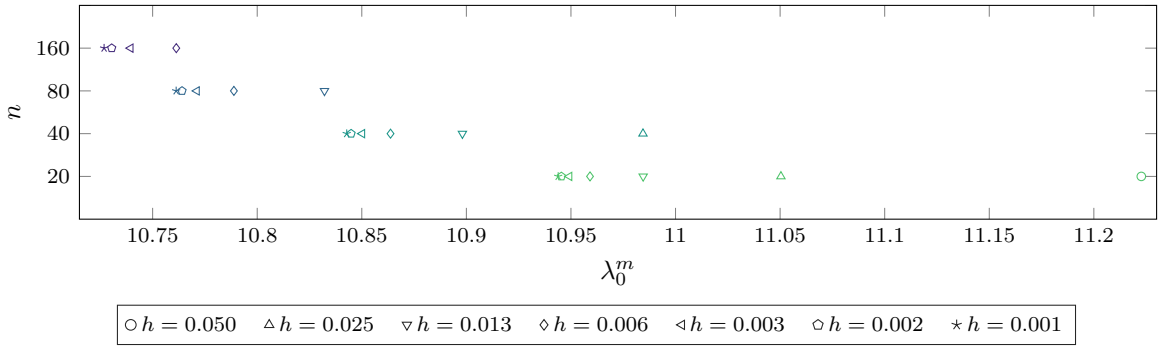

In this section we illustrate the abstract ideas from the previous sections with a concrete numerical example.
The closure of the domain we study belongs to the class of filled Julia sets described in Example \ref{ex:filled-julia} hence
the pixelation approximations converge in the Mosco sense and the sequence of arithmetic algorithms constructed
in Section \ref{subsec:spec-algorithm2} converge in the Attouch-Wets metric.
Numerical experiments for the Laplacian on filled Julia sets were also recently performed
in \cite{strichartzSpectralPropertiesLaplacians2020}.

We consider the filled Julia set $K$ defined by
\begin{align}\label{eq:K-julia}
	K = \left\{z\in\C \,:\, |f^{\circ n}(z)|\leq 2 \, \forall n\in\N \right\}, \quad\text{ where }\quad	f(z) = z^2+\f{\sqrt 5 - 1}{2}.
\end{align}
It can be shown that $K$ has a fractal boundary (cf. \cite{mandelbrot}).
Mandelbrot suggested the name ``San Marco Set'' for $K$, because it resembles the Basilica of Venice together with its reflection in a flooded piazza (see Figure \ref{fig:Julia_O}).

We implemented a finite element method on subsequent pixelated domains for
\begin{equation}
  \label{eq:O-julia}
  \cO=\intt(K)
\end{equation}
and computed approximations to the lowest eigenvalue with increasingly fine meshes. To be more precise,  for $n\in \{20,40,80,160\}$ \red{we consider a FEM approximation $\lambda_0^m$ to the lowest eigenvalue} of $-\Delta_{\cO_n}$ for meshes $\mathcal T^m$ with $m\in\{0,1,\dots,\f{160}{n}\}$ (or equivalently $h\in\{2^{-m}n^{-1}\,:\, m\in\{0,\dots,\f{160}{n}\}\}$, recall Section \ref{subsec:spec-algorithm2}).

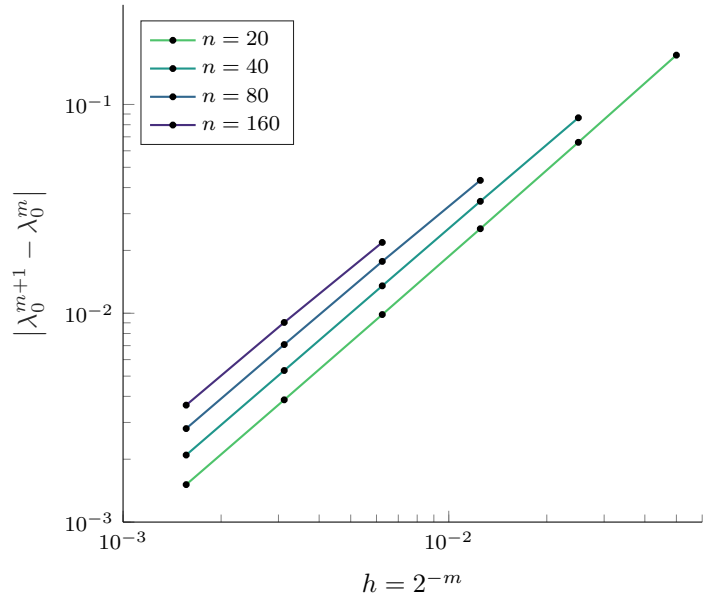
\begin{figure}[t]
	\centering
%
%
\pgfplotsset{every tick label/.append style={font=\footnotesize}}

\setlength\figurewidth{0.5\textwidth}

\definecolor{mycolor1}{rgb}{0.78331,0.87928,0.12540}%
\definecolor{mycolor2}{rgb}{0.30415,0.76470,0.41994}%
\definecolor{mycolor3}{rgb}{0.12183,0.58905,0.54562}%
\definecolor{mycolor4}{rgb}{0.19236,0.40320,0.55584}%
\definecolor{mycolor5}{rgb}{0.27713,0.18523,0.48990}%
\begin{tikzpicture}

\begin{axis}[%
width=\figurewidth,
height=0.9\figurewidth,
at={(1.193in,0.837in)},
scale only axis,
xmode=log,
xmin=0.001,
xmax=0.06,
xminorticks=true,
xlabel={$h=2^{-m}$},
ymode=log,
ymin=0.001,
ymax=0.3,
yminorticks=true,
ylabel style={font=\color{white!15!black}},
ylabel={$\big|\lambda^{m+1}_0 - \lambda^m_0\big|$},
axis background/.style={fill=white},
axis x line*=bottom,
axis y line*=left,
legend style={at={(0.03,0.97)}, anchor=north west, legend cell align=left, align=left, draw=white!15!black, font=\footnotesize}
]

\addplot [color=mycolor2, line width=0.8pt, mark size=0.9pt, mark=*, mark options={solid, black}]
  table[row sep=crcr]{%
0.05	0.172296920548757\\
0.025	0.0658743850618944\\
0.0125	0.0254033215756309\\
0.00625	0.00985925386492781\\
0.003125	0.00385072311557799\\
0.0015625	0.0015116556361594\\
};
\addlegendentry{$n=20$}

\addplot [color=mycolor3, line width=0.8pt, mark size=0.9pt, mark=*, mark options={solid, black}]
  table[row sep=crcr]{%
0.025	0.0863390688298544\\
0.0125	0.0343632207444955\\
0.00625	0.0135320389089646\\
0.003125	0.00531986639797211\\
0.0015625	0.00209450769494879\\
};
\addlegendentry{$n=40$}

\addplot [color=mycolor4, line width=0.8pt, mark size=0.9pt, mark=*, mark options={solid, black}]
  table[row sep=crcr]{%
0.0125	0.043303590791675\\
0.00625	0.0177175656206927\\
0.003125	0.00708255055786289\\
0.0015625	0.00280404334979067\\
};
\addlegendentry{$n=80$}

\addplot [color=mycolor5, line width=0.8pt, mark size=0.9pt, mark=*, mark options={solid, black}]
  table[row sep=crcr]{%
0.00625	0.0218431267820449\\
0.003125	0.00903711213543446\\
0.0015625	0.00363182201842527\\
};
\addlegendentry{$n=160$}

\end{axis}

\end{tikzpicture}%
	\caption{Double logarithmic plot of relative differences $|\lambda^{m+1}_0-\lambda^{m}_0|$ for $n\in\{20,40, 80,160\}$, $m\in\{0,\dots,160/n\}$.}
	\label{fig:conv_plot}
\end{figure}
\begin{figure}[t]
	\centering
	\includegraphics[width=0.99\textwidth]{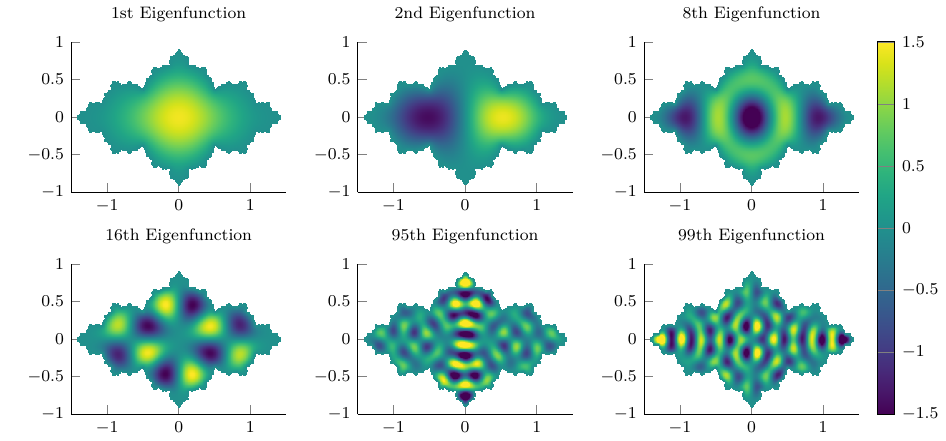}
	\caption{Selected approximated eigenfunctions of $\cO$ for $n=100$, $h=1/n$ (normalised such that $v_i^{\top} B^m v_i=1$).}
	\label{fig:efs}
\end{figure}

Because the emphasis of this paper is theoretical, implementing the full a-posteriori error computation of \cite{LiuOishi2013} would be beyond the scope of the current work. 
Instead, the approximation to the lowest eigenvalue in each case was computed using the Rayleigh-Ritz method for the pair $(A^m,B^m)$ of stiffness and mass matrix: the Rayleigh quotient $(v^{\top}A^m v)/(v^{\top}B^m v)$ was minimised via a straightforward gradient descent method. The gradient descent was iterated until the derivative of the Rayleigh quotient was less than $10^{-10}$. The results are shown in Figures \ref{fig:EVs} and \ref{fig:conv_plot}. The data points in Figure \ref{fig:EVs} suggest that for each fixed $n$ the refinement of the mesh leads to a convergent sequence of approximations. 

The decay of the successive differences between the $\lambda_0^m$ in Figure \ref{fig:conv_plot} suggests a convergence rate of approximately $1.33 \approx 2\cdot\f{2}{3}$, which is in accordance with the regularity of the pixelated domain, which has reentrant corners of angle $\f32\pi$ (cf. \cite{BBX}). Figure \ref{fig:conv_plot} also suggests a worsening of the convergence rate if both $n$ and $h$ are increased simultaneously. This is reflected in the fact that the lines in Figure \ref{fig:conv_plot} move to the left as $n$ increases. This degradation of the convergence rate is to be expected as the rough boundary of $\cO$ is better and better approximated by $\cO_n$ as $n$ increases.

For triangulations which are not prohibitively fine, the approximations of higher eigenvalues and eigenfunctions \red{also} can be computed. Figure \ref{fig:efs} shows 6 selected approximated eigenfunctions $v_1,v_2, v_8, v_{16}, v_{95}$ and $v_{99}$ for $n=100$, $h=n^{-1}$ (i.e. $m=0$). The approximations are normalised such that $v_i^{\top} B^m v_i=1$, where $B^m$ is the mass matrix associated with $\mathcal T^m$.

The Matlab Code that produced Figures \ref{fig:Julia_O} - \ref{fig:efs} is openly available at \url{https://github.com/frank-roesler/PixelSpectra} and can easily be adapted to arbitrary domains $\cO$.
%
%
%



\bibliographystyle{amsplain}
\bibliography{Citations.bib}


\end{document}